\documentclass[a4paper,11pt,reqno]{amsart}

\pdfoutput=1

\usepackage{amsmath}
\usepackage{amsthm}
\usepackage{amssymb}
\usepackage{latexsym}
\usepackage{graphicx,color}
\usepackage{booktabs}
\usepackage{hyperref}
\usepackage[foot]{amsaddr}

\numberwithin{equation}{section}
\newtheorem{thm}{Theorem}[section]
\newtheorem{prop}[thm]{Proposition}
\newtheorem{lem}[thm]{Lemma}

\theoremstyle{definition}
\newtheorem{defn}[thm]{Definition}

\newtheorem{rem}[thm]{Remark}

\allowdisplaybreaks[4]

\newcommand{\al}{\alpha}

\newcommand{\de}{\delta}

\newcommand{\e}{\varepsilon}

\newcommand{\la}{\lambda}

\renewcommand{\th}{\theta}
\newcommand{\om}{\omega}
\newcommand{\Om}{\Omega}
\newcommand{\ze}{\zeta}

\newcommand{\p}{\partial}
\newcommand{\I}{\infty}
\newcommand{\Sc}[1]{\mathcal{#1}}
\newcommand{\F}{\Sc{F}}

\newcommand{\FR}[1]{\mathfrak{#1}}
\newcommand{\Bo}[1]{\mathbb{#1}}
\newcommand{\R}{\Bo{R}}
\newcommand{\T}{\Bo{T}}

\newcommand{\lec}{\lesssim}
\newcommand{\gec}{\gtrsim}
\newcommand{\hhat}{\widehat}
\newcommand{\bbar}{\overline}
\newcommand{\ti}{\widetilde}

\newcommand{\shugo}[1]{\{ #1\}}
\newcommand{\Shugo}[2]{\big\{ \, #1 \, \big| \, #2 \, \big\}}

\newcommand{\LR}[1]{{\langle #1 \rangle }}
\newcommand{\BLR}[2]{{\left\langle #1,\,#2 \right\rangle }}
\newcommand{\chf}[1]{\chi_{#1}}
\newcommand{\med}{\mathrm{med}}
\newcommand{\bfmu}{{\boldsymbol \mu}}
\newcommand{\bfxi}{{\boldsymbol \xi}}

\newcommand{\norm}[2]{\big\| #1 \big\| _{#2}}
\newcommand{\tnorm}[2]{\| #1 \| _{#2}}

\newcommand{\eq}[2]{\begin{equation} \label{#1} \begin{split} #2 \end{split} \end{equation}}
\newcommand{\eqq}[1]{\begin{align*} #1 \end{align*}}

\newcommand{\eqs}[1]{\begin{gather*} #1 \end{gather*}}
\newcommand{\mat}[1]{\begin{smallmatrix} #1 \end{smallmatrix}}

\newcommand{\hx}{\hspace{10pt}}

\renewcommand{\Re}{\mathrm{Re}}
\renewcommand{\Im}{\mathrm{Im}}

\title[Unconditional uniqueness for nonlinear dispersive equations]{Unconditional uniqueness of solutions for nonlinear dispersive equations}
\author[N. Kishimoto]{Nobu Kishimoto}
\address{Research Institute for Mathematical Sciences, Kyoto University\\
Kyoto 606-8502, Japan}
\email{nobu@kurims.kyoto-u.ac.jp}
\thanks{This article is to complement the announcement \cite{K-announcement} with detailed proofs.}%}

\begin{document}

\begin{abstract}
When a solution to the Cauchy problem for nonlinear dispersive equations is obtained by a fixed point argument using auxiliary function spaces, it is non-trivial to ensure uniqueness of solutions in a natural space such as the class of continuous curves in the data space. %45
This property is called \emph{unconditional uniqueness}, and proving it often requires some additional work. %14
In the last decade, unconditional uniqueness has been shown for some canonical nonlinear dispersive equations by an integration-by-parts technique, which can be regarded as a variant of the (Poincar\'e-Dulac) normal form reduction. %32

In this article, we aim to provide an abstract framework for establishing unconditional uniqueness as well as existence of certain weak solutions via infinite iteration of the normal form reduction. %30
In particular, in an abstract setting we find two sets of fundamental estimates, each of which can be used repeatedly to generate all multilinear estimates of arbitrarily high degrees required in this scheme. %33
Then, we confirm versatility of the framework by applying it to various equations, including the cubic nonlinear Schr\"odinger equation (NLS) in higher dimension, the cubic NLS with fractional Laplacians, the cubic derivative NLS, and the Zakharov system, for which new results on unconditional uniqueness are obtained under the periodic boundary condition. %51
\end{abstract}

\maketitle

\tableofcontents

%%%%%%%%%%%%%%%%%%%%%%%%%%%%%%%%%%%%%%%%%%%%
%%%%%%%%%%%%%%%%%%%%%%%%%%%%%%%%%%%%%%%%%%%%
%%%%%%%%%%%%%%%%%%%%%%%%%%%%%%%%%%%%%%%%%%%%

\section{Introduction}

%\subsection{General setting}

In the present article, we consider \emph{unconditional uniqueness} (UU) of solutions to the Cauchy problem for general nonlinear dispersive equations. 
Here, UU in a Banach space $B$ means uniqueness of solutions (in the sense of distribution) to the Cauchy problem, with initial data given in $B$, in the class of continuous $B$-valued functions $C([0,T];B)$.
Hereafter, we write $C_TB$ to denote $C([0,T];B)$.

To begin with, we note that UU is sometimes trivial, especially if the solution is obtained by an iteration argument in $C_TB$ itself.
As an example, consider the following nonlinear Schr\"odinger equation (NLS) with the cubic nonlinearity:
\[ i\p _tu+\Delta u=|u|^2u,\qquad (t,x)\in [0,T]\times \R ^d,\]
with initial data in Sobolev space $H^s(\mathbb{R}^d)$.
Then, UU in $H^s$ is trivial in the case $s>\frac{d}{2}$ for which $H^s$ is an algebra.
The concept of \emph{unconditional well-posedness} (\mbox{i.e.}, well-posedness with unconditional uniqueness) was introduced by T. Kato \cite{K95}, and he pointed out that UU becomes meaningful in the case that the solution is obtained by iteration but using an auxiliary function space in addition to $C_TH^s$.
In the above NLS example, one can still construct solutions for $s<\frac{d}{2}$ in a certain range by using the Strichartz estimates, but then uniqueness is obtained initially in the intersection of $C_TH^s$ with some mixed Lebesgue space $L^p_TL^q$ used as an auxiliary space.
In such a case, an additional argument is often required to establish UU.
As a notion of uniqueness which does not depend on how the solution is constructed, UU can be used to identify a solution obtained by some method (e.g., a compactness argument) with another solution constructed by a different method (e.g., an iteration argument).

We also remark that two kinds of criticality may be relevant in the problem of UU.
To see this, let us again consider Sobolev space $H^s$ as the data space $B$.
First, if the equation is invariant under the scaling transform, then the scale-invariant Sobolev regularity $s=s_{\textrm{scl}}$ is initially expected to be the lowest regularity that admits well-posedness of the Cauchy problem. (However, there are many cases where the Cauchy problem becomes ill-posed at some regularity higher than the scaling.)
Secondly, there may exist the regularity threshold $s=s_{\mathrm{embd}}$ below which the nonlinear part does not make sense in the distributional framework.
Therefore, we naturally focus on UU in $H^s$ for $s\ge \max \{ s_{\textrm{scl}},s_{\textrm{embd}}\}$.
As for the above NLS example, the equation is invariant under the scaling transform $u(t,x)\mapsto \la u(\la ^2t,\la x)$ ($\la >0$), which preserves the $\dot{H}^s(\R ^d)$ norm if $s=s_{\textrm{scl}}:=\frac{d}{2}-1$, whereas the embedding $H^s(\R^d)\hookrightarrow L^3(\R^d)$ holds if and only if $s\ge s_{\textrm{embd}}:=\frac{d}{6}$.
Therefore, it is natural to consider UU in $H^s$ only for $s\ge \max \shugo{\frac{d}{2}-1,\,\frac{d}{6}}$.

In the literature, the notion of UU is sometimes used for uniqueness in $L^\infty ((0,T);B)=:L^\infty _TB$.
In most of the applications, however, uniqueness in $C_TB$ implies that in $L^\infty _TB'$ for a slightly smaller space $B'$, as we see below in the NLS example:
Let $u\in L^\infty _TH^s(\mathbb{R}^d)$ be a solution to the cubic NLS for some $s\geq s_{\textrm{embd}}$.
By the equation and the Sobolev embedding, $u$ belongs to $W^{1,\infty}_TH^{s_0}$ for any $s_0$ satisfying $s_0\leq s-2$ and $s_0<-\frac{d}{2}$, and in particular (after modifying it on a set of measure zero) $u\in C_TH^{s_0}$.
By interpolation, we deduce that $u\in C_TH^{s'}$ for any $s'<s$.
Therefore, uniqueness of solutions in $L^\infty _TH^s$ follows once we have uniqueness in $C_TH^{s'}$ for some $s'<s$.
Based on this observation, in the present article we intend to consider uniqueness in the class $C_TB$ rather than $L^\infty _TB$.%
\footnote{We point out that uniqueness in $L^\infty_TH^{s}$ would not be obtained in this way at the lowest regularity $s=\max \{ s_{\textrm{scl}},\,s_{\textrm{embd}}\}$.
In fact, uniqueness in the larger class $L^\infty_TH^s$ at the critical regularity is known to be a delicate issue; see, e.g., \cite[Remark~1.2 (iii)]{ST08}.}

There are many results on UU in the non-periodic case (\mbox{i.e.}, the Cauchy problem on $\R ^d$).
For NLS (with general power-type nonlinearities), the first result of T. Kato \cite{K95} has been improved by Furioli and Terraneo \cite{FT03}, Rogers \cite{R07}, Win and Tsutsumi \cite{ST08}, Han and Fang~\cite{HF13}.
These results settled the UU problem for most of $s\ge \max \shugo{s_{\textrm{scl}},\,s_{\textrm{embd}}}$. 
For other equations, see, e.g., Zhou \cite{Z97} (the KdV equation), Win \cite{S08} (the cubic derivative NLS), and Masmoudi and Nakanishi \cite{MN09} (the Zakharov system). 

Compared to the non-periodic case, the study of UU in the periodic setting had been less developed.
However, in the last decade, several results have been obtained by successive applications of integration by parts (or differentiation by parts) in the time variable.%
\footnote{For results on UU in the periodic setting by a different approach, we mention the recent works of Chen, Holmer~\cite{CH19} and of Herr, Sohinger~\cite{HS19} on NLS based on the analysis of the Gross-Pitaevskii hierarchy.}
This technique has an underlying idea --- exploiting nonlinear smoothing effects due to the time oscillation of the non-resonant interactions --- in common with the Fourier restriction method, whereas it does not need any auxiliary space and thus is suitable for UU.
This method can also be regarded as a variant of the Poincar\'e-Dulac \emph{normal form reduction} (NFR); we refer to \cite{A88} for details of the Poincar\'e-Dulac NFR. 
For the KdV equation Babin et \mbox{al.} \cite{BIT11} obtained the result by applying NFR, which was followed by Kwon and Oh \cite{KO12} (the modified KdV equation), T.K. Kato and Tsugawa \cite{KTp} (the fifth order KdV-type equations), and the author \cite{K-BOp} (the Benjamin-Ono equation).

The result of Guo et \mbox{al.} \cite{GKO13} on one-dimensional periodic cubic NLS was a breakthrough in this direction.
It is worth noticing that they had to invoke NFR \emph{infinitely} many times to make all the nonlinear estimates closed in $C_TH^s$, in contrast to the previous works for the KdV-type equations in which, despite of the derivative losses in the nonlinearities, the results were obtained by applying such integration-by-parts procedure \emph{finitely} many times.
Such a difference comes from the difference of resonance structure between the NLS and the KdV type equations.
This technique of unlimitedly iterating NFR introduced in \cite{GKO13} has been motivating many studies on UU; \cite{CGKO17} and \cite{OWp} for instance, and adaptation of the technique to the non-periodic setting was achieved in \cite{P19} and \cite{KOY20},%
\footnote{%
To be precise, an adaptation of infinite NFR technique to the non-periodic setting had appeared first in the Ph.D.~thesis \cite{Yth} of the third author of~\cite{KOY20}, which was later refined and announced as \cite{KOY20}, while the study in \cite{P19,CHKP19a} had been started separately from these works.
A difference from the periodic case appears in justification of some formal calculations for rough functions.
In particular, they needed to justify the use of the product rule 
\[ \p _t\big[ \hat{v}(t,\xi _1)\hat{\bar{v}}(t,\xi _2)\hat{v}(t,\xi _3)\big] = (\p_t\hat{v})(t,\xi _1)\hat{\bar{v}}(t,\xi _2)\hat{v}(t,\xi _3)+\hat{v}(t,\xi _1)(\p _t\hat{\bar{v}})(t,\xi _2)\hat{v}(t,\xi _3)+\hat{v}(t,\xi _1)\hat{\bar{v}}(t,\xi _2)(\p_t\hat{v})(t,\xi _3) \]
inside the integral over $\xi _j$'s for a general function $v$ satisfying $\p _tv\in C_tL^1(\R)$. 
(Here, $\hat{v}$ denotes the spatial Fourier transform of $v$.)
In their situation, as explained in~\cite{Yth,KOY20}, the above identity holds in the classical sense because $\p _tv\in C_tL^1(\R )$ implies $\hat{v}(\cdot ,\xi )\in C^1_t$ for each $\xi \in \R$. 
We will consider in Appendix~\ref{appendix:R} how and in what sense the above calculation can be justified in the situations where $\p _tv\in C_tL^1$ does not hold in general.} 
which were followed by \cite{CHKP19a,CHKP19b}, and \cite{MYp}.

We notice that the previous studies mentioned above are restricted to a few specific equations such as the cubic NLS and the modified KdV equations,%
\footnote{A certain quadratic derivative NLS was studied in \cite{CGKO17} exploiting its special structure.
In \cite{MYp} the cubic derivative NLS on $\R$ was studied; as the authors mentioned, their result was built upon a former version of the present article concerning the same problem on $\T$.}
all in one dimension.
There are many potential difficulties in this machinery.
Some of them are as follows:
\begin{enumerate}
\item[(a)] Each application of NFR will produce higher and higher order nonlinear terms.
For instance, in the case of cubic NLS, nonlinear terms of order $2k+3$ will appear after the $k$-th application of NFR.
Then, one needs to establish multilinear estimates with higher and higher degrees of nonlinearities. 
\item[(b)] As the degree of nonlinearities increases, resonance structure becomes different and more and more complicated. 
Since NFR can be applied only to the \emph{non-resonant} part of nonlinear terms, one cannot neglect keeping track of varying resonance structure.
\item[(c)] The number of terms after the $k$-th NFR grows in a factorial order $(k!)^C$, which is faster than an exponential order $C^k$.
\item[(d)] One has to justify the limiting procedure of ``applying NFR indefinitely'', namely, find the \emph{limit equation} and show that any distributional solution of the original equation in $C_TH^s$ is also a solution of it.
\end{enumerate}
Guo et \mbox{al.} \cite{GKO13} could deal with the above difficulties for the simplest NLS, \mbox{i.e.}, in the one-dimensional cubic case, by explicitly writing down all the nonlinear terms and making delicate resonance/non-resonance decompositions of them.
Since their proof was highly dependent on simplicity of the equation, it is by no means easy to adapt their argument to more general settings, even to the \emph{two-dimensional} cubic NLS.

\medskip
In the present article, we aim to generalize the infinite NFR machinery so that it can be applied to a wide range of nonlinear dispersive equations.
Our main result, as stated below, gives two different criteria for the infinite NFR machinery to work.
Each of them consists of several simple multilinear estimates of the lowest degree, and we can show that these estimates are actually enough to yield all the required higher-degree multilinear estimates by an induction on the degree, and also enough to justify the limit equation.
Such an idea of reducing all the matters to several ``fundamental estimates'' has recently been demonstrated for some specific equations by Kwon et \mbox{al.}~\cite{KOY20}, while we realize it in an abstract framework.

%\subsection{Main results}
To state the main theorem, let us concentrate on the periodic case $x\in \T ^d:=(\R /2\pi \Bo{Z})^d$.
By the Fourier series expansion, we move to the frequency space and consider the following abstract equation:
\eq{Abs}{\p _t\om _n(t)=\sum _{n=n_1+\dots +n_p}e^{it\phi}m\om_{n_1}(t)\om_{n_2}(t)\cdots \om_{n_p}(t)+\Sc{R}[\om ]_n(t),\quad n\in \Bo{Z}^d,}
where $p\geq 2$ is the degree of (the principal part of) the nonlinearity, $\phi =\phi (n,n_1,\dots ,n_p)\in \R$ denotes the phase part, $m=m(n,n_1,\dots ,n_p)\in \Bo{C}$ is the multiplier part, and $\Sc{R}[\om ]$ is the remainder part.
For example, the KdV equation 
\[ \p _tu+\p _x^3u=\p _x(u^2),\qquad (t,x)\in [0,T]\times \T\]
is, by setting $\om _n(t):=\frac{1}{2\pi}\int _{0}^{2\pi}[U(-t)u(t)](x)e^{-inx}\,dx$ with $U(t)=e^{-t\p_x^3}$ being the propagator for the Airy equation, equivalent to 
\[ \p _t\om _n(t)=in\sum _{n=n_1+n_2}e^{it(n^3-n_1^3-n_2^3)}\om_{n_1}(t)\om_{n_2}(t),\qquad (t,n)\in [0,T]\times \Bo{Z}.\]
This is of the form \eqref{Abs} with $p=2$, $\phi =n^3-n_1^3-n_2^3$, $m=in$, and $\Sc{R}=0$.
In such a way, nonlinear dispersive equations can be represented as \eqref{Abs} if the nonlinearity is a polynomial in $u$, $\bar{u}$ and derivatives of them with constant coefficients.
The initial data $\om _n(0)$ is now given in weighted $\ell ^2$ spaces, $\ell ^2_s(\mathbb{Z}^d)$, instead of $H^s(\T^d)$: 
\[ \ell ^2_s(\mathbb{Z}^d):=\LR{\cdot }^{-s}\ell ^2(\Bo{Z}^d),\quad \tnorm{\om}{\ell ^2_s}:=\tnorm{\LR{\cdot }^s\om }{\ell ^2}\quad (s\in \R );\qquad \LR{\cdot} :=(1+|\cdot |^2)^{\frac{1}{2}}.\] 
UU for the original equation in $H^s$ is now replaced with that for \eqref{Abs} in $\ell ^2_s$.

We say $\om \in C_T\ell ^2_s$ is a solution to the Cauchy problem associated with \eqref{Abs} if the right-hand side of \eqref{Abs} is well-defined as a (temporal) distribution and it satisfies \eqref{Abs} in $\Sc{D}'((0,T))$ for each $n\in \Bo{Z}^d$, with its value at $t=0$ being the same as the given initial datum.

The following is our main theorem:
\begin{thm}\label{thm:abstract}
Let $s\in \R$ and $T>0$.
Assume that $\Sc{R}[\om ]\in C_T\ell ^2_s$ for any $\om \in C_T\ell ^2_s$ and it holds 
\[
(R)\qquad \left\{
\begin{aligned}
\norm{\Sc{R}[\om ]}{C_T\ell ^2_s}&\le C\big( \tnorm{\om}{C_T\ell ^2_s}\big) ,\\
\quad \norm{\Sc{R}[\om ]-\Sc{R}[\ti{\om}]}{C_T\ell ^2_s}&\le C\big( \tnorm{\om}{C_T\ell ^2_s},\,\tnorm{\ti{\om}}{C_T\ell ^2_s}\big) \tnorm{\om-\ti{\om}}{C_T\ell ^2_s}.
\end{aligned}
\right. 
\]
Assume further that for some Banach space%
\footnote{In the applications discussed in Sections \ref{section:cNLS}--\ref{section:Z}, we always take $X$ to be a suitable weighted $\ell ^p$ space.}
$X$ of functions on $\Bo{Z}^d$ with the property
\eq{absolutenorm}{|\om _n|\le |\ti{\om}_n|\quad (n\in \Bo{Z}^d)\qquad \Longrightarrow \qquad \tnorm{\om}{X}\le C\tnorm{\ti{\om}}{X},}
we have one of the following $[A]$, $[B]$:

$[A]$ There exists $\de \in (0,\frac{1}{2})$ such that 
\eqq{&(A1)\quad \norm{\sum _{n=n_1+\dots +n_p}\frac{|m|}{\LR{\phi}^{1/2}}\om ^{(1)}_{n_1}\cdots \om^{(p)}_{n_p}}{\ell ^2_s}\le C\prod _{j=1}^p\norm{\om ^{(j)}}{\ell ^2_s},\\[5pt]
&(A2)\quad \norm{\sum _{n=n_1+\dots +n_p}\frac{|m|}{\LR{\phi}^{1-\de}}\om ^{(1)}_{n_1}\cdots \om ^{(p)}_{n_p}}{X}\le C\min _{1\le j\le p}\Big[ \norm{\om ^{(j)}}{X}\prod _{\mat{l=1\\ l\neq j}}^p\norm{\om ^{(l)}}{\ell ^2_s}\Big] ,\\[5pt]
&(A3)\quad \norm{\sum _{n=n_1+\dots +n_p}|m|\om ^{(1)}_{n_1}\cdots \om ^{(p)}_{n_p}}{X}\le C\prod _{j=1}^p\norm{\om ^{(j)}}{\ell ^2_s}.
}

$[B]$ There exist $s_1,s_2\in \R$ satisfying $s_1<s<s_2$ such that
\eqq{&(B1)\quad \sup _{\mu \in \Bo{Z}}\norm{\sum _{\mat{n=n_1+\dots +n_p\\ \mu \le \phi <\mu +1}}|m|\om ^{(1)}_{n_1}\cdots \om ^{(p)}_{n_p}}{\ell ^2_{s_1}}\le C\prod _{j=1}^p\norm{\om ^{(j)}}{\ell ^2_{s_1}},\\[5pt]
&(B1)'\quad \norm{\sum _{n=n_1+\dots +n_p}|m|\om ^{(1)}_{n_1}\cdots \om ^{(p)}_{n_p}}{\ell ^2_{s_2}}\le C\prod _{j=1}^p\norm{\om ^{(j)}}{\ell ^2_{s_2}},\\[5pt]
&(B2)\quad \sup _{\mu \in \Bo{Z}}\norm{\sum _{\mat{n=n_1+\dots +n_p\\ \mu \le \phi <\mu +1}}|m|\om ^{(1)}_{n_1}\cdots \om ^{(p)}_{n_p}}{X}\le C\min _{1\le j\le p}\Big[ \norm{\om ^{(j)}}{X}\prod _{\mat{l=1\\ l\neq j}}^p\norm{\om ^{(l)}}{\ell ^2_{s_1}}\Big] ,\\[5pt]
&(B2)'\quad \norm{\sum _{n=n_1+\dots +n_p}|m|\om ^{(1)}_{n_1}\cdots \om ^{(p)}_{n_p}}{X}\le C\min _{1\le j\le p}\Big[ \norm{\om ^{(j)}}{X}\prod _{\mat{l=1\\ l\neq j}}^p\norm{\om ^{(l)}}{\ell ^2_{s_2}}\Big] ,\\[5pt]
&(B3)\quad \norm{\sum _{n=n_1+\dots +n_p}|m|\om ^{(1)}_{n_1}\cdots \om ^{(p)}_{n_p}}{X}\le C\prod _{j=1}^p\norm{\om ^{(j)}}{\ell ^2_s}.\qquad \textup{(same as (A3))}
}

Then, there is at most one solution%
\footnote{The estimate $(A3)$ (or $(B3)$) ensures that the right-hand side of \eqref{Abs} is well-defined as a bounded function in $t$ for each $n\in \mathbb{Z}^d$ whenever $\om \in C_T\ell ^2_s$.} 
for the Cauchy problem associated with \eqref{Abs} in $C_T\ell ^2_s$.
\end{thm}

In the previous works concerning \emph{specific} equations, the following tasks were carried out to establish UU in their \emph{specific} contexts: (a) To find suitable ``fundamental estimates'', (b) to prove them, (c) to deduce required multilinear estimates of higher degrees from these fundamental estimates, and (d) to operate the infinite NFR machinery using these multilinear estimates.
When adapting this argument to a different equation, these works can be useful guidelines on one hand, but one still has to verify (a)--(d) in the \emph{new} context on the other hand.

The point of our result is that it completely automates the processes (a), (c), and (d) in an \emph{abstract} setting.%
\footnote{Precisely, part of (a) still remains under control; there is freedom on how to choose the remainder part $\Sc{R}[\om ]$ in the equation \eqref{Abs}, and when necessary one can adjust it so that the ``fundamental estimates'' may become easier (or possible) to verify.}
In particular, given our theorem, the goal for proving UU is simply to check these ``fundamental estimates'' to hold (i.e., the task (b)), and one may forget about the subsequent NFR procedure (i.e., (c) and (d)).
This reduction will surely make the method of NFR for establishing UU more accessible to a broad audience.

Our result, which proposes a criterion for UU in an abstract setting, is essentially different from the previous results, each of which proved UU directly in a specific context.
Of course, such an abstract framework will not be useful to problems for which the criterion is difficult to check (i.e., the imposed ``fundamental estimates'' are too strong).
However, our framework turns out to be surprisingly robust and have wide applicability, as can be seen from various applications to be given in later sections.

\begin{rem}\label{rem:condAB}
Here are some comments on the conditions $[A]$, $[B]$.

(i) The condition $[B]$ was originally discovered through refining the idea of \cite{GKO13} for one-dimensional cubic NLS, and typically it is effective for equations with nonlinearities in which derivative loss does not occur.
On the other hand, the condition $[A]$ seems new, and it keeps a certain negative power of the modulation factor $\phi$ so that it can be used for the nonlinearities with derivative losses.

(ii) By $(A1)$, we impose the condition that a \emph{half} power of $\phi$ should be sufficient to control the nonlinearity in $\ell ^2_s$.
This assumption may seem to be unreasonably restrictive; in fact, nonlinear terms after the first application of NFR have an entire power of $\phi$ in the denominator (see, e.g., $\Sc{N}^{(1)}_0[\om]$ in \eqref{eq:2nd} below), and what we have assumed seems fairly stronger than required for estimating these terms.
However, in order to control every terms arising through the infinite NFR procedure only by using the fundamental estimates, we need to deal with varying resonance structures in a unified manner.
We will see in the proof that the sets of fundamental estimates $[A]$, $[B]$ are in fact suitable for this purpose.

(iii) The estimate $(A1)$, which uses half of the modulation factor, has a remarkable similarity to the standard multilinear estimate in Fourier restriction spaces (Bourgain spaces):
\eq{est:xsb}{\norm{N[u_1,\dots ,u_p]}{X^{s,-\frac{1}{2}+}}\lec \prod _{l=1}^p\tnorm{u_l}{X^{s,\frac{1}{2}+}},}where $\tnorm{u}{X^{s,b}}:=\norm{U(-t)u(t)}{H^b_tH^s_x}$, $U(t)$ is the linear propagator, and $N[u,\dots ,u]$ denotes the nonlinearity (power type of order $p$). 
In fact, if we set $\om =\F [U(-t)u(t)]$ (where $\F$ and $\F ^{-1}$ denote the spatial Fourier transformation and its inverse, respectively) and define $\phi$, $m$ so that
\[ \sum _{n=n_1+\cdots +n_p}e^{it\phi}m\om _{n_1}\cdots \om _{n_p}~=~\F U(-t)N[u,\dots ,u],\]
a familiar argument using the elementary inequality
\[ \int _{\R}\frac{d\tau}{\LR{\tau -a}^{1+}\LR{\tau -b}^{1+}}\lec \frac{1}{\LR{a-b}^{1+}}\]
reduces the estimate \eqref{est:xsb} to the bounds on weights such as
\[ \sup _{n}\sum _{n=n_1+\cdots +n_p}\frac{|m|^2\LR{n}^{2s}}{\LR{\phi}^{1-}\LR{n_1}^{2s}\cdots \LR{n_p}^{2s}}<\I ,\]
while $(A1)$ is reduced by the Cauchy-Schwarz inequality to almost the same statement
\[ \sup _{n}\sum _{n=n_1+\cdots +n_p}\frac{|m|^2\LR{n}^{2s}}{\LR{\phi}\LR{n_1}^{2s}\cdots \LR{n_p}^{2s}}<\I .\]
In this respect, it is reasonable to expect that, in the case where \emph{conditional} well-posedness is shown by a fixed point argument in Bourgain spaces, one may show $(A1)$ by almost the same argument, and then $(A2)$ is essentially the only additional  condition for \emph{unconditional} uniqueness.
(Note that in many cases $(A3)$ follows simply from embedding estimates.)
This will be demonstrated in applications to the cubic derivative NLS and the Zakharov system in Sections~\ref{section:DNLS} and \ref{section:Z} below. 
\end{rem}

\begin{rem}\label{rem:thm2}
Let us make some additional remarks on Theorem~\ref{thm:abstract}.

(i) The remainder part $\Sc{R}[\om ]$ in \eqref{Abs} basically includes easily controlled terms or the specific part of the main term which is in itself easily estimated but causes trouble in establishing the multilinear estimates $(A1)$--$(A3)$ or $(B1)$--$(B3)$ if it remains in the main term.
It is sometimes important to detect such a problematic part in the main term and put it in $\Sc{R} [\om ]$ before carrying on NFR.
We see such an example in Section~\ref{section:FNLS}.

(ii) The normal form reduction is effective to extract nonlinear smoothing effect in non-resonant interactions; while it does not work if there exist resonant interactions with derivative losses.
In such a case, however, one may apply the framework to a certain equivalent equation in which resonant interactions are removed or become tamer.
For example, this is the case for the cubic derivative NLS equation with a suitable gauge transform applied; see Section~\ref{section:DNLS}.

(iii) It is straightforward to extend the result to the problem \eqref{Abs} posed on a rescaled lattice $\lambda _1^{-1}\mathbb{Z}\times \cdots \times \lambda _d^{-1}\mathbb{Z}$ for any $\lambda _1,\dots ,\lambda_d>0$, which corresponds to nonlinear dispersive equations posed on a rescaled torus $(\R /2\pi \lambda_1\mathbb{Z})\times \cdots \times (\mathbb{R}/2\pi \lambda _d\mathbb{Z})$.
One can also easily formulate analogous statements for systems of equations and in the case of multiple (principal) nonlinear terms (for which we need to assume the same one of the conditions $[A]$, $[B]$ for all terms).
We will see in Section~\ref{subsec:multiple} how the proof should be modified for these cases.
See Section~\ref{section:Z} for an application of the framework to a system.

(iv) As done in the aforementioned works, one can adapt the infinite NFR scheme to the non-periodic setting, which requires an additional care in justification of formal calculations.
Our result also extends to the non-periodic case; we will see in Appendix~\ref{appendix:R} the idea on how to make such justification in the non-periodic case.

(v) In the theorem, we are concerned with UU for \eqref{Abs} in the space $\ell ^2_s$, which corresponds to UU for nonlinear dispersive equations in Sobolev space $H^s$.
In fact, the same argument works if $\ell^2_s$ is replaced with any Banach space $Y$ of functions on $\mathbb{Z}^d$ satisfying the property \eqref{absolutenorm}.%
\footnote{See Remark~\ref{rem:Bgeneral} below on how the spaces $\ell ^2_{s_1}$ and $\ell^2_{s_2}$ should be changed in the assumption $[B]$.} 
In particular, our framework can also be used to prove UU for nonlinear dispersive equations in some different scales such as Fourier-Lebesgue spaces, ($L^2$-based) Besov and modulation spaces.
We note that the NFR method has already been used in these settings different from $H^s$; see, e.g., \cite{CGKO17,P19,OWp}.
\end{rem}

Let us briefly see how the infinite NFR machinery proceeds with the above fundamental estimates.
All of required (infinitely many) multilinear estimates are obtained inductively by using these fundamental estimates. 
The estimate $(A1)$ or $(B1)+(B1)'$ (together with $(R)$) is the main tool to obtain $\ell^2_s$-control for all the nonlinear terms in each NFR step, except for one term which is rougher than the others.
Then $(A2)$ or $(B2)+(B2)'$ (with $(R)$) enables us to show that this term vanishes in ``weaker'' $X$-norm in the limit equation.
It is essential in the proof of unconditional uniqueness to notice that one cannot rely on approximation by smooth solutions; one needs to justify every formal calculation for a solution in $C_T\ell^2_s$ directly without approximation (by solutions), because a general solution in $C_T\ell^2_s$ is not necessarily approximated by smooth solutions. 
However, this can be done by using $(A2)+(A3)$ or $(B2)+(B2)'+(B3)$ (with $(R)$).

\medskip
It was observed in \cite{GKO13} and subsequent works (e.g., \cite{P19,KOY20,OWp,CHKP19b}) that the infinite NFR scheme can be used to construct a certain kind of weak solutions for rough initial data by approximating with smooth solutions.
For this purpose, it is enough to establish various estimates only on smooth solutions.
In particular, one does not need estimates in ``weaker'' space (prepared to justify formal calculations for rough solutions).
In our abstract setting, assuming $(R)+(A1)$ or $(R)+(B1)+(B1)'$ is basically sufficient for such a use of NFR.
More precisely, we can show the following result:
\begin{thm}\label{thm:weak'}
Let $s\in \R$.
Assume that for any $s'\ge s$ and $T>0$, $\Sc{R}[\om]\in C_T\ell^2_{s'}$ if $\om \in C_T\ell^2_{s'}$ and%
\footnote{We note that the constant in the first estimate of $(R)'$ should not depend on the $C_T\ell ^2_{s'}$ norm, while it is allowed for the second estimate.
This requirement seems reasonable when $\mathcal{R}[\om]$ comes from power-type nonlinearities.} 
\[ (R)'\qquad \left\{
\begin{aligned}
\norm{\Sc{R}[\om ]}{C_T\ell ^2_{s'}}&\le C\big( s',\tnorm{\om}{C_T\ell ^2_s}\big) \tnorm{\om}{C_T\ell ^2_{s'}},\\
~\norm{\Sc{R}[\om ]-\Sc{R}[\ti{\om}]}{C_T\ell ^2_{s'}}&\le C\big( s',\tnorm{\om}{C_T\ell ^2_{s'}},\,\tnorm{\ti{\om}}{C_T\ell ^2_{s'}}\big) \tnorm{\om-\ti{\om}}{C_T\ell ^2_{s'}}.
\end{aligned}
\right. %}
\]
Moreover, assume one of the following $[A]'$, $[B]'$:%
\footnote{The estimates $(A1)$, $(B1)$, $(B1)'$ are the same as those in Theorem~\ref{thm:abstract}.}

$[A]'$ There exists $s_2>s$ such that
\eqq{
&(A1)\qquad \norm{\sum _{n=n_1+\dots +n_p}\frac{|m|}{\LR{\phi}^{1/2}}\om ^{(1)}_{n_1}\cdots \om^{(p)}_{n_p}}{\ell^2_s}\le C\prod _{j=1}^p\norm{\om ^{(j)}}{\ell^2_s},\\
&(A1)'\qquad \norm{\sum _{n=n_1+\dots +n_p}|m|\om ^{(1)}_{n_1}\cdots \om^{(p)}_{n_p}}{\ell^2_s}\le C\prod _{j=1}^p\norm{\om ^{(j)}}{\ell^2_{s_2}}.
}

$[B]'$ There exist $s_1<s$ and $s_2>s$ such that
\eqq{
&(B1)\qquad \sup _{\mu \in \Bo{Z}}\norm{\sum _{\mat{n=n_1+\dots +n_p\\ \mu \le \phi <\mu +1}}|m|\om ^{(1)}_{n_1}\cdots \om^{(p)}_{n_p}}{\ell^2_{s_1}}\le C\prod _{j=1}^p\norm{\om ^{(j)}}{\ell^2_{s_1}},\\
&(B1)'\qquad \norm{\sum _{n=n_1+\dots +n_p}|m|\om ^{(1)}_{n_1}\cdots \om^{(p)}_{n_p}}{\ell^2_{s_2}}\le C\prod _{j=1}^p\norm{\om ^{(j)}}{\ell^2_{s_2}}.
}

Then, for any $\om _0\in \ell ^2_{s}$ there exist $T>0$ depending on $\tnorm{\om_0}{\ell ^2_s}$ and a weak solution $\om \in C_T\ell ^2_s$ to \eqref{Abs} with $\om (0)=\om _0$.
We also have continuous dependence on initial data and persistence of regularity for this weak solution.
\end{thm}
Definition of weak solutions to \eqref{Abs} and a precise statement of the above theorem will be given in Section~\ref{section:W} as Definition~\ref{defn:weak} and Theorem~\ref{thm:weak}.

In the proof of existence of weak solutions $\om \in C_T\ell ^2_s$, the fundamental estimates such as $[A]'$ and $[B]'$ (and multilinear estimates of various degrees obtained from the fundamental ones) are mainly used to verify that
\begin{quote}
(a) the limit equation holds in the sense of $C_T\ell ^2_s$ for \emph{regular} solutions in $C_T\ell^2_{s_2}$.
\end{quote}
This will yield an a priori Lipschitz bound in $C_T\ell ^2_s$ for regular solutions $\om _N$ with approximating initial data $\om _N(0)\in \ell ^2_{s_2}$, $\lim\limits _{N\to \I}\om _N(0)=\om (0)$ in $\ell^2_s$, by which we can take the limit $\om :=\lim\limits _{N\to \I}\om _N$ and obtain a weak solution $\om\in C_T\ell^2_s$.
Here, the length $T$ of the time interval (on which the limit equation is valid) is determined in terms of $\tnorm{\om_N}{C_T\ell ^2_s}$, so we need to show that
\begin{quote}
(b) the approximating solutions $\{ \om _N\} \subset C_T\ell ^2_{s_2}$ are bounded in $C_T\ell ^2_s$.
\end{quote}
As observed in the previous works, this can also be shown based on the limit equation (combined with a continuity argument). 
Now, we notice that the previous works treated such equations as the one-dimensional cubic NLS and the modified KdV, for which smooth solutions were known to exist globally in time, both in non-periodic and in periodic settings.
In general, however, we first need to show that
\begin{quote}
(c) approximating regular solutions $\{ \om _N\}$ exist on a uniform time interval $[0,T]$.
\end{quote}
This is not trivial at all.%
\footnote{In \cite{CHKP19b}, the one-dimensional cubic NLS was considered in $H^{s'}(\R )+H^{s''}(\T )$, i.e., with initial data given by sums of decaying and periodic functions. 
It is not clear in this setting whether smooth solutions exist globally in time, and thus the claim (c) would be non-trivial.
Since one has local existence of regular solutions on a time interval determined by the $\ell ^2_{s_2}$ norm of initial data, what is needed here is an a priori bound in $\ell^2_{s_2}$ on $[0,T]$, with $T$ depending only on the $\ell^2_s$ norm.
Note that the argument showing (a), (b) only yields an a priori estimate in $\ell^2_s$, and hence is not sufficient.
It seems that this point was not taken into consideration in \cite{CHKP19b}.
} 
Indeed, these solutions should have unbounded initial data; i.e., $\tnorm{\om _N(0)}{\ell ^2_{s_2}}\to \I$, so that a standard local well-posedness in $\ell ^2_{s_2}$ is not sufficient by itself.
Furthermore, (especially in Case $[A]'$) we do not even know local-in-time existence of regular solutions,%
\footnote{As mentioned in Remark~\ref{rem:condAB}~(iii), in Case $[A]'$ it is reasonable to expect that the relevant multilinear estimate in Bourgain spaces would follow from an argument similar to the proof of $(A1)$.
If this is the case, we can construct solutions in the usual distributional sense and show (conditional) local well-posedness in $\ell ^2_s$.
In Theorem~\ref{thm:weak}, however, we do not assume this situation to occur and construct weak solutions via NFR approach.} 
so at the very beginning we have to show that
\begin{quote}
(d) a solution $\om _N\in C_{T_N}\ell ^2_{s_2}$ exists on a time interval $[0,T_N]$ for each $\om _N(0)\in \ell ^2_{s_2}$.
\end{quote}
We do not assume any of (d), (c), (b) and (a) in the theorem; instead, we will see that all of them are consequences of the fundamental estimates assumed in the theorem.

\medskip
%\subsection{Plan of the article}
At the end of this section, we give the plan of this article.
In Section~\ref{section:proof}, we prove Theorem~\ref{thm:abstract}.
After that, we see the convenience and versatility of our framework through various applications:
In Sections~\ref{section:cNLS}, \ref{section:FNLS} we apply Theorem~\ref{thm:abstract} $[B]$ to the problems with no derivative losses; the higher-dimensional cubic NLS and the one-dimensional cubic fractional NLS, respectively.
Applications of Case $[A]$ are given in Sections~\ref{section:DNLS}, \ref{section:Z}, where we consider the one-dimensional models with derivative losses; the cubic derivative NLS and the Zakharov system. 
Finally, Theorem~\ref{thm:weak'} is restated as Theorem~\ref{thm:weak} and proved in Section~\ref{section:W}.
In Appendix~\ref{appendix:R} we discuss how to adapt Theorem~\ref{thm:abstract} to the non-periodic setting.

%%%%%%%%%%%%%%%%%%%%%%%%%%%%%%%%%%%%%%%%%%%%
%%%%%%%%%%%%%%%%%%%%%%%%%%%%%%%%%%%%%%%%%%%%
%%%%%%%%%%%%%%%%%%%%%%%%%%%%%%%%%%%%%%%%%%%%

%\bigskip
\section{Abstract theory}\label{section:proof}

In this section, we shall prove Theorem~\ref{thm:abstract}.

\subsection{Notation}
Following \cite{GKO13}, we use the notation of \emph{ordered tree}, which is useful to give precise definition of infinitely many nonlinear terms created in the NFR procedure. 

\begin{defn}[ordered tree]
Let $p\geq 2$ be a given integer.
For $J\in \Bo{N}$, define $\FR{T}(J)$ by the set of all rooted $p$-ary trees with $J$ nodes in which an ordering is specified for the $p$ children of each node and the $J$ nodes are also labeled in a manner consistent with the tree order.

More precisely, $\Sc{T}\in \FR{T}(J)$ is a partially ordered set (with a partial order $\preceq$) satisfying the following properties:
\begin{enumerate}
\item $\Sc{T}$ has the (unique) least element $r$ (i.e., $r\preceq a$ for all $a\in \Sc{T}$), which is called the \emph{root}.
\item For each element $a\in \Sc{T}\setminus \{ \text{root}\}$, there exists a unique element $b\in \Sc{T}$ such that $b\neq a$, $b\preceq a$, and that $b\preceq c\preceq a$ implies $c=a$ or $c=b$.
We say $b$ is the \emph{parent} of $a$ and $a$ is a \emph{child} of $b$.
\item An element of $\Sc{T}$ is called a \emph{node} if it has a child; otherwise, it is called a \emph{leaf}.
$\Sc{T}$ has exactly $J$ nodes, which are numbered from $1$ to $J$ so that $a_{j_1}\preceq a_{j_2}$ implies $j_1\le j_2$, denoting the $j$-th node by $a_j$.
\item Each node of $\Sc{T}$ has exactly $p$ children, which are numbered from $1$ to $p$.
\end{enumerate}
\end{defn}

We write $\Sc{T}_0$, $\Sc{T}_\I$ to denote the subset of $\Sc{T}$ consisting of all nodes and of all leaves, respectively.
We easily see the following properties:
\begin{itemize}
\item For $\Sc{T}\in \FR{T}(J)$, $\# \Sc{T}=pJ+1$, $\# \Sc{T}_0=J$, and $\# \Sc{T}_\I =(p-1)J+1$.
\item $\# \FR{T}(J)=\prod\limits _{j=0}^{J-1}\big\{ (p-1)j+1\big\} \le (p-1)^JJ!$.
\end{itemize}

\begin{defn}
\mbox{}
\begin{itemize}
\item Let $J\in \Bo{N}$ and $\Sc{T}\in \FR{T}(J)$.
We call a map $\mathbf{n}=\{ n_a\} _{a\in \Sc{T}}:\Sc{T}\to \Bo{Z}^d$ an \emph{index function} if for each $a\in \Sc{T}_0$, with its children being denoted by $a^1,a^2,\dots ,a^p$, it holds that 
\[ n_a=n_{a^1}+n_{a^2}+\cdots +n_{a^p}.\]
We write $\FR{N}(\Sc{T})$ to denote the set of all index functions on $\Sc{T}$, and for $n\in \mathbb{Z}^d$ we define $\mathfrak{N}_n(\mathcal{T}):=\{ \mathbf{n}\in \mathfrak{N}(\mathcal{T})\,|\,n_{\mathrm{root}}=n\}$.
\item Given $\Sc{T}\in \FR{T}(J)$ and $\mathbf{n}\in \FR{N}(\Sc{T})$, we write 
\[ \phi ^j:=\phi( n_{a_j}, n_{a_j^1},\dots ,n_{a_j^p}),\qquad m^j:=m( n_{a_j}, n_{a_j^1},\dots ,n_{a_j^p})\]
for $1\le j\le J$, where $a_j$ is the $j$-th node and $a_j^1,a_j^2,\dots ,a_j^p$ are its children.
\end{itemize}
\end{defn}

\subsection{Normal form reduction}

Assume the hypotheses in Theorem~\ref{thm:abstract}, and let $\om \in C_T\ell ^2_s$ be a solution of \eqref{Abs}.
By the above definition, \eqref{Abs} can be rewritten as
\eqq{\p _t\om _n(t)&=\sum _{\Sc{T}\in \mathfrak{T}(1)}\sum _{\mathbf{n}\in \mathfrak{N}_n(\Sc{T})}e^{it\phi ^1}m^1\prod _{a\in \Sc{T}_\I}\om _{n_a}(t)+\Sc{R}[\om (t)]_n\\
&=:\Sc{N}^{(1)}[\om (t)]_n+\Sc{R}[\om (t)]_n,\qquad n\in \Bo{Z}^d.
}
First of all, the estimate $(A3)$ (or $(B3)$) ensures that the series in $\Sc{N}^{(1)}[\om ]_n$ is absolutely convergent:
\eq{conv:N1}{\sup _{t\in [0,T]}\sum _{\mathbf{n}\in \mathfrak{N}_n(\Sc{T})}|m^1|\prod _{a\in \Sc{T}_\I}|\om _{n_a}(t)|~\lesssim_n~ \sup _{t\in [0,T]}\norm{\sum _{\mathbf{n}\in \mathfrak{N}_n(\Sc{T})}|m^1|\prod _{a\in \Sc{T}_\I}|\om _{n_a}(t)|}{X}\lesssim \| \om \|_{C_T\ell ^2_s}^p<\infty .}
Thus, $\Sc{N}^{(1)}[\om (t)]_n+\Sc{R}[\om (t)]_n$ is well-defined as a bounded function on $[0,T]$ for each $n\in \Bo{Z}^d$. 
Since $\om _n(\cdot)$ is a solution of the equation in $\mathcal{D}'((0,T))$ and continuous on $[0,T]$, it satisfies the equation in the integral form:
\eq{eq:1st}{\om _n(\cdot )\Big| _0^t=\int _0^t\Big( \Sc{N}^{(1)}[\om (\tau )]_n+\Sc{R}[\om (\tau )]_n\Big) d\tau ,\qquad t\in [0,T],~~n\in \Bo{Z}^d.}
We call \eqref{eq:1st} the \emph{equation of the first generation}.

Uniqueness of the solutions would follow if we could have an estimate for $\Sc{N}^{(1)}[\om ]$ which is closed in $\ell ^2_s$; however, the only estimate available is of the $X$ norm in terms of the $\ell ^2_s$ norm.
Thus, we decompose $\Sc{N}^{(1)}$ into slowly oscillating terms (which we call \emph{resonant} terms) and rapidly oscillating ones (\emph{non-resonant} terms), and then apply an integration by parts in $t$ to the rapidly oscillating part to get a large factor in the denominator.
Namely, we first divide the nonlinearity as
\[ \om _n\Big| _0^t=\int _0^t\,\Big( \Sc{N}^{(1)}_R[\om ]_n+\Sc{N}^{(1)}_{N\!R}[\om ]_n+\Sc{R}[\om ]_n\Big) ,\]
where 
\[ \Sc{N}^{(1)}_R[\om ]_n:=\sum _{\Sc{T}\in \mathfrak{T}(1)}\sum _{\mat{\mathbf{n}\in \mathfrak{N}_n(\Sc{T})\\ |\phi ^1|\text{\,:\,small}}}e^{it\phi ^1}m^1\prod _{a\in \Sc{T}_\I}\om _{n_a},\quad \Sc{N}^{(1)}_{N\!R}[\om ]_n:=\sum _{\Sc{T}\in \mathfrak{T}(1)}\sum _{\mat{\mathbf{n}\in \mathfrak{N}_n(\Sc{T})\\ |\phi ^1|\text{\,:\,large}}}e^{it\phi ^1}m^1\prod _{a\in \Sc{T}_\I}\om _{n_a},\]
and the precise meaning of `small' or `large' will be specified later.
Then, by an integration by parts, we \emph{formally} have 
\eq{eq:just}{
\int _0^t\Sc{N}^{(1)}_{N\!R}[\om ]_n&=\sum _{\Sc{T}\in \mathfrak{T}(1)}\sum _{\mat{\mathbf{n}\in \mathfrak{N}_n(\Sc{T})\\ |\phi ^1|\text{\,:\,large}}}\int _0^te^{i\tau \phi ^1}m^1\prod _{a\in \Sc{T}_\I}\om _{n_a}(\tau )\,d\tau \\
&=\sum _{\Sc{T}\in \mathfrak{T}(1)}\sum _{\mat{\mathbf{n}\in \mathfrak{N}_n(\Sc{T})\\ |\phi ^1|\text{\,:\,large}}}\!\!\bigg( \Big[ \frac{e^{i\tau \phi ^1}}{i\phi ^1}m^1\prod _{a\in \Sc{T}_\I}\om _{n_a}(\tau )\Big] _0^t\\[-5pt]
&\hspace{120pt} -\int _0^t\frac{e^{i\tau \phi ^1}}{i\phi ^1}m^1\sum _{a\in \Sc{T}_\I}\Big[ \prod _{\mat{b\in \Sc{T}_\I \\b\neq a}}\om _{n_b}(\tau )\Big] (\p _t\om _{n_a})(\tau )\,d\tau \bigg) \\[-10pt]
&=\Sc{N}^{(1)}_{0}[\om ]_n\Big| _0^t+\int _0^t\Sc{N}^{(1)}_{1}[\om ]_n,
}
where
\eqq{
&\Sc{N}^{(1)}_0[\om ]_n:=\sum _{\Sc{T}\in \mathfrak{T}(1)}\sum _{\mat{\mathbf{n}\in \mathfrak{N}_n(\Sc{T})\\ |\phi ^1|\text{\,:\,large}}}\frac{e^{it\phi ^1}}{i\phi ^1}m^1\prod _{a\in \Sc{T}_\I}\om _{n_a},\\
&\Sc{N}^{(1)}_1[\om ]_n:=-\sum _{\Sc{T}\in \mathfrak{T}(1)}\sum _{\mat{\mathbf{n}\in \mathfrak{N}_n(\Sc{T})\\ |\phi ^1|\text{\,:\,large}}}\frac{e^{it\phi ^1}}{i\phi ^1}m^1\sum _{a\in \Sc{T}_\I}\Big[ \prod _{\mat{b\in \Sc{T}_\I \\b\neq a}}\om _{n_b}\Big] \p _t\om _{n_a}.
}
Substituting the original equation \eqref{Abs}, we have $\mathcal{N}^{(1)}_1[\om ]=\mathcal{R}^{(1)}[\om ]+\mathcal{N}^{(2)}[\om ]$,
\begin{align}
\Sc{R}^{(1)}[\om ]_n&:=-\sum _{\Sc{T}\in \mathfrak{T}(1)}\sum _{\mat{\mathbf{n}\in \mathfrak{N}_n(\Sc{T})\\ |\phi ^1|\text{\,:\,large}}}\frac{e^{it\phi ^1}}{i\phi ^1}m^1\sum _{a\in \Sc{T}_\I}\Big[ \prod _{\mat{b\in \Sc{T}_\I \\b\neq a}}\om _{n_b}\Big] \Sc{R}[\om ]_{n_a},\notag \\
\begin{split}
\Sc{N}^{(2)}[\om ]_n&:=-\sum _{\Sc{T}\in \mathfrak{T}(1)}\sum _{\mat{\mathbf{n}\in \mathfrak{N}_n(\Sc{T})\\ |\phi ^1|\text{\,:\,large}}}\frac{e^{it\phi ^1}}{i\phi ^1}m^1\sum _{a\in \Sc{T}_\I} \Big[ \prod _{\mat{b\in \Sc{T}_\I \\b\neq a}}\om _{n_b}\Big] \\
&\hx\hx \times \Big[ \sum _{n_a=n_{a^1}+\dots +n_{a^p}}e^{it\phi (n_a,n_{a^1},\dots n_{a^p})}m(n_a,n_{a^1},\dots n_{a^p})\om_{n_{a^1}}\cdots \om_{n_{a^p}}\Big] 
\end{split}\label{exp2-1}\\
&\;=-\sum _{\Sc{T}\in \mathfrak{T}(2)}\sum _{\mat{\mathbf{n}\in \mathfrak{N}_n(\Sc{T})\\ |\phi ^1|\text{\,:\,large}}}\frac{e^{it(\phi ^1+\phi ^2)}}{i\phi ^1}m^1m^2\prod _{a\in \Sc{T}_\I}\om _{n_a}.\label{exp2-2}
\end{align}
We have thus obtained the \emph{equation of the second generation}:
\eq{eq:2nd}{\om _n\Big| _0^t=\Sc{N}^{(1)}_{0}[\om ]_n\Big| _0^t+\int _0^t\,\Big( \Sc{N}^{(1)}_R[\om ]_n+\Sc{R}[\om ]_n+\Sc{R}^{(1)}[\om ]_n+\Sc{N}^{(2)}[\om ]_n\Big) ,\quad t\in [0,T],~n\in \Bo{Z}^d.}
Observe that $\Sc{T}\in \mathfrak{T}(1)$ in the intermediate expression \eqref{exp2-1} of $\Sc{N}^{(2)}[\om ]_n$ gives exactly $p$ trees of $\mathfrak{T}(2)$ in the final expression \eqref{exp2-2} by developing one of $p$ leaves $a\in \Sc{T}_\I$ into the (second) node and its $p$ children.
NFR means the above reduction procedure including decomposition into resonant/non-resonant terms, application of an integration by parts to the non-resonant part, and substitution of the original equation. 
Note that the last term $\Sc{N}^{(2)}[\om ]_n$ is of order $2(p-1)+1$ in $\om$, which is higher than the others.

As mentioned before, formal calculations in \eqref{eq:just} must be justified for a general solution $\om \in C_T\ell ^2_s$.
The absolute convergence \eqref{conv:N1} and Fubini's theorem verify the first equality in \eqref{eq:just}.
Also, the estimate $(A3)$ (or $(B3)$) implies
\[ \norm{\sum _{\mathbf{n}\in \mathfrak{N}_n(\Sc{T})}e^{it'\phi ^1}m^1\Big( \prod _{a\in \Sc{T}_\I}\om _{n_a}(t')-\prod _{a\in \Sc{T}_\I}\om _{n_a}(t)\Big)}{X}\lesssim \| \om \|_{C_T\ell^2_s}^{p-1}\| \om (t')-\om (t)\|_{\ell ^2_s} ~\to ~0\quad (t'\to t). \]
This and \eqref{conv:N1}, together with the dominated convergence theorem, show that $\mathcal{N}^{(1)}[\om ]_n\in C([0,T])$, and then, by the integral equation, $\om _n(\cdot ) \in C^1([0,T])$ for each $n$.%
\footnote{For any $\om \in C_T\ell ^2_s$, it holds that $\| \chi _{|n|>L}\om \|_{C_T\ell^2_s}\to 0$ ($L\to \infty$).
From this property, we see that the series in $\mathcal{N}^{(1)}[\om ]_n$ converges absolutely and uniformly in $t$, from which continuity of $\mathcal{N}^{(1)}[\om ]_n$ follows.
(In fact, we can show a stronger claim that $\mathcal{N}^{(1)}[\om ]\in C_TX$ for $\om \in C_T\ell ^2_s$; see the argument in Appendix~\ref{appendix:R}, proof of Theorem~\ref{thm:abstract-R} in the case (ii).)
Similarly, convergence of the series in $\mathcal{N}^{(1)}_1[\om ]_n$ is shown to be uniform in $t$, which allows termwise differentiation of the series in $\mathcal{N}^{(1)}_0[\om ]_n$ and hence justifies \eqref{eq:just}.
We point out that this argument does not work if the space $\ell^2_s$ is replaced with an $\ell^\infty$-type space $Y$, for which $\| \chi _{|n|>L}\om \|_{Y}\to 0$ fails to hold.} 
This verifies the second equality in \eqref{eq:just}, i.e., integration by parts in $t$ and application of the product rule for each $n$ and $\mathbf{n}\in \mathfrak{N}_n(\mathcal{T})$.
Similarly to the first one, the third equality is verified once we have the absolute convergence:
\[ \sup_{t\in [0,T]}\sum _{\mat{\mathbf{n}\in \mathfrak{N}_n(\Sc{T})\\ |\phi ^1|\text{\,:\,large}}}\frac{|m^1|}{|\phi ^1|}|(\p _t\om _{n_a})(t)|\prod _{\mat{b\in \Sc{T}_\I \\b\neq a}}|\om _{n_b}(t)| ~<~\infty \qquad (\mathcal{T}\in \mathfrak{T}(1),~a\in \mathcal{T}_\infty ),\]
which in turn follows from 
\eqs{
\sup_{t\in [0,T]}\sum _{\mat{\mathbf{n}\in \mathfrak{N}_n(\Sc{T})\\ |\phi ^1|\text{\,:\,large}}}\frac{|m^1|}{|\phi ^1|}|\mathcal{R}[\om ]_{n_a}(t)|\prod _{\mat{b\in \Sc{T}_\I \\b\neq a}}|\om _{n_b}(t)| ~<~\infty \qquad (\mathcal{T}\in \mathfrak{T}(1),~a\in \mathcal{T}_\infty ),\\
\sup_{t\in [0,T]}\sum _{\mat{\mathbf{n}\in \mathfrak{N}_n(\Sc{T})\\ |\phi ^1|\text{\,:\,large}}}\frac{|m^1||m^2|}{|\phi ^1|}\prod _{a\in \Sc{T}_\I}|\om _{n_a}(t)| ~<~\infty \qquad (\mathcal{T}\in \mathfrak{T}(2)).
}
These two estimates are consequences of the estimates for $\mathcal{R}^{(1)}[\om]$ and $\mathcal{N}^{(2)}[\om]$ in Proposition~\ref{prop:abstract} below.
They also show that the new terms $\mathcal{N}^{(1)}_1[\om ]_n$, $\mathcal{R}^{(1)}[\om ]_n$, $\mathcal{N}^{(2)}[\om ]_n$ are all well-defined, and verify rearrangement of the series from \eqref{exp2-1} to \eqref{exp2-2}.

Recall that we already have a closed estimate $(R)$ in $\ell ^2_s$ for $\Sc{R}[\om ]$.
Furthermore, since the summation in $\Sc{N}^{(1)}_R[\om ]$ is restricted and there is a large denominator in $\Sc{N}^{(1)}_0[\om ]$ and $\Sc{R}^{(1)}[\om ]$, one can expect that these terms also have closed $\ell ^2_s$ estimates.
The problem is then how to control the higher-order term $\Sc{N}^{(2)}[\om ]$.
In general, this term requires more regularity and does not admit a closed $\ell ^2_s$ estimate for the same $s$, and one has to repeat NFR for this term.
(In some equations, however, the structure of resonance is good enough and one has a closed $\ell ^2_s$ estimate also for $\Sc{N}^{(2)}[\om ]$. This is the case, \mbox{e.g.}, for the KdV equation and $s>\frac{1}{2}$; see \cite{BIT11}.)

After the second NFR, we get the \emph{equation of the third generation} as
\[ \om _n\Big| _0^t=\sum _{j=1}^2\Sc{N}^{(j)}_{0}[\om ]_n\Big| _0^t+\int _0^t\,\Big( \sum _{j=1}^2\Sc{N}^{(j)}_R[\om ]_n+\sum _{j=0}^2\Sc{R}^{(j)}[\om ]_n+\Sc{N}^{(3)}[\om ]_n\Big) ,\quad t\in [0,T],~~n\in \mathbb{Z}^d,\]
where $\Sc{R}^{(0)}[\om ]:=\Sc{R}[\om ]$ and
\eqq{
\Sc{N}^{(2)}_R[\om ]_n&:=-\sum _{\Sc{T}\in \mathfrak{T}(2)}\sum _{\mat{\mathbf{n}\in \mathfrak{N}_n(\Sc{T})\\ |\phi ^1|\text{\,:\,large},\,|\phi ^1+\phi ^2|\text{\,:\,small}}}\frac{e^{it(\phi ^1+\phi ^2)}}{i\phi ^1}m^1m^2\prod _{a\in \Sc{T}_\I}\om _{n_a},\\
\Sc{N}^{(2)}_0[\om ]_n&:=-\sum _{\Sc{T}\in \mathfrak{T}(2)}\sum _{\mat{\mathbf{n}\in \mathfrak{N}_n(\Sc{T})\\ |\phi ^1|\text{\,:\,large},\,|\phi ^1+\phi ^2|\text{\,:\,large}}}\frac{e^{it(\phi ^1+\phi ^2)}}{i\phi ^1i(\phi ^1+\phi ^2)}m^1m^2\prod _{a\in \Sc{T}_\I}\om _{n_a},\\
\Sc{R}^{(2)}[\om ]_n&:=\sum _{\Sc{T}\in \mathfrak{T}(2)}\sum _{\mat{\mathbf{n}\in \mathfrak{N}_n(\Sc{T})\\ |\phi ^1|\text{\,:\,large},\,|\phi ^1+\phi ^2|\text{\,:\,large}}}\frac{e^{it(\phi ^1+\phi ^2)}}{i\phi ^1i(\phi ^1+\phi ^2)}m^1m^2\sum _{a\in \Sc{T}_\I}\Big[ \prod _{\mat{b\in \Sc{T}_\I \\b\neq a}}\om _{n_b}\Big] \Sc{R}[\om ]_{n_a},\\
\Sc{N}^{(3)}[\om ]_n&:=\sum _{\Sc{T}\in \mathfrak{T}(3)}\sum _{\mat{\mathbf{n}\in \mathfrak{N}_n(\Sc{T})\\ |\phi ^1|\text{\,:\,large},\,|\phi ^1+\phi ^2|\text{\,:\,large}}}\frac{e^{it(\phi ^1+\phi ^2+\phi ^3)}}{i\phi ^1i(\phi ^1+\phi ^2)}m^1m^2m^3\prod _{a\in \Sc{T}_\I}\om _{n_a}.
}
Similarly, after the $(J-1)$-th NFR, we get the \emph{equation of the $J$-th generation} as
\eq{eq:Jth}{\om _n\Big| _0^t=\sum _{j=1}^{J-1}\Sc{N}^{(j)}_{0}[\om ]_n\Big| _0^t+\int _0^t\,\Big( \sum _{j=1}^{J-1}\Sc{N}^{(j)}_R[\om ]_n+\sum _{j=0}^{J-1}\Sc{R}^{(j)}[\om ]_n+\Sc{N}^{(J)}[\om ]_n\Big) ,\quad t\in [0,T],~~n\in \mathbb{Z}^d,}
where
\eqq{
\Sc{N}^{(j)}_R[\om ]_n&:=(-1)^{j-1}\sum _{\Sc{T}\in \mathfrak{T}(j)}\sum _{\mat{\mathbf{n}\in \mathfrak{N}_n(\Sc{T})\\ (\phi ^k)_{k=1}^{j}\in \Phi _R^{j}}}\frac{e^{it\ti{\phi}^j}}{\prod\limits _{k=1}^{j-1}i\ti{\phi}^k}\Big[ \prod _{k=1}^jm^k\Big] \prod _{a\in \Sc{T}_\I}\om _{n_a},\\
\Sc{N}^{(j)}_0[\om ]_n&:=(-1)^{j-1}\sum _{\Sc{T}\in \mathfrak{T}(j)}\sum _{\mat{\mathbf{n}\in \mathfrak{N}_n(\Sc{T})\\ (\phi ^k)_{k=1}^{j}\in \Phi _{N\!R}^{j}}}\frac{e^{it\ti{\phi}^j}}{\prod\limits _{k=1}^{j}i\ti{\phi}^k}\Big[ \prod _{k=1}^jm^k\Big] \prod _{a\in \Sc{T}_\I}\om _{n_a},\\
\Sc{R}^{(j)}[\om ]_n&:=(-1)^j\sum _{\Sc{T}\in \mathfrak{T}(j)}\sum _{\mat{\mathbf{n}\in \mathfrak{N}_n(\Sc{T})\\ (\phi ^k)_{k=1}^{j}\in \Phi _{N\!R}^{j}}}\frac{e^{it\ti{\phi}^j}}{\prod\limits _{k=1}^{j}i\ti{\phi}^k}\Big[ \prod _{k=1}^jm^k\Big] \sum _{a\in \Sc{T}_\I}\Big[ \prod _{\mat{b\in \Sc{T}_\I \\b\neq a}}\om _{n_b}\Big] \Sc{R}[\om ]_{n_a},\\
\Sc{N}^{(J)}[\om ]_n&:=(-1)^{J-1}\sum _{\Sc{T}\in \mathfrak{T}(J)}\sum _{\mat{\mathbf{n}\in \mathfrak{N}_n(\Sc{T})\\ (\phi ^k)_{k=1}^{J}\in \Phi _R^{J}\cup \Phi _{N\!R}^{J}}}\frac{e^{it\ti{\phi}^J}}{\prod\limits _{k=1}^{J-1}i\ti{\phi}^k}\Big[ \prod _{k=1}^Jm^k\Big] \prod _{a\in \Sc{T}_\I}\om _{n_a},
}
and we have introduced the notation $\ti{\phi}^k:=\phi ^1+\phi ^2+\cdots +\phi ^k$,
\eqq{
\Phi _R^{j}&:=\Shugo{(\phi ^k)_{k=1}^j\in \R ^j}{\text{$|\ti{\phi}^k|$ : large for $1\le k\le j-1$, and $|\ti{\phi}^j|$ : small}},\\
\Phi _{N\!R}^{j}&:=\Shugo{(\phi ^k)_{k=1}^j\in \R ^j}{\text{$|\ti{\phi}^k|$ : large for $1\le k\le j$}}.
}
Observe that $\Sc{N}^{(J)}[\om ]_n$ has been obtained by replacing one of $\omega _{n_a}$ in $\Sc{N}^{(J-1)}_0[\om ]_n$ with $-\Sc{N}^{(1)}[\om ]_{n_a}$ and rearranging the series; each three $\Sc{T}\in \mathfrak{T}(J-1)$ in the expression of $\Sc{N}^{(J-1)}_0[\om ]_n$ gives exactly $\# \Sc{T}_\I$ trees of $\mathfrak{T}(J)$ in $\Sc{N}^{(J)}[\om ]_n$ by developing each of the leaves into the $J$-th node and its $p$ children.
The precise definition of $\Phi _R^{j}$ and $\Phi _{N\!R}^{j}$ will be given in the proof of Proposition~\ref{prop:abstract} below; it depends on the size of solutions and also on which of $[A]$ and $[B]$ we assume in the theorem. 
Finally, every formal calculation in deriving \eqref{eq:Jth} can be justified for a general solution $\om \in C_T\ell ^2_s$ of \eqref{Abs} in a similar manner to the case \eqref{eq:2nd}, based on the fact that $\om _n(\cdot )\in C^1([0,T])$ for each $n$, and that the series in $\Sc{N}^{(J)}_R[\om ]_n$, $\Sc{N}^{(J)}_0[\om ]_n$, $\Sc{R}^{(J)}[\om ]_n$, $\Sc{N}^{(J)}[\om ]_n$ are absolutely convergent (and bounded in $t$), which will also be shown in Proposition~\ref{prop:abstract} below.

\subsection{Proof of the main theorem}

Now, we are interested in the situation where the infimum of the regularity $s$ for which $\Sc{N}^{(j)}[\om ]$ has a closed $\ell ^2_s$ estimate is \emph{not} improved as generation $j$ proceeds.
For instance, in the case of the one-dimensional cubic Schr\"odinger equation treated in \cite{GKO13}, $\Sc{N}^{(j)}[\om ]$ always requires $s>\frac{1}{2}$ for closed $\ell ^2_s$ estimates, while all the other terms can be controlled for $s\ge 0$.
At first glance there seems no hope to obtain any a priori estimate on the solutions for lower regularities by the NFR method.

The idea in \cite{GKO13} to overcome this difficulty is that one can eliminate the bad term $\Sc{N}^{(J)}[\om ]$ by repeating NFR infinitely many times.
Specifically, $\Sc{N}^{(J)}[\om ]$ cannot be estimated in $\ell^2_s$ but can be controlled and shown to vanish in a ``weaker'' topology $X$.%
\footnote{This strategy was not explicitly written in the original work \cite{GKO13}, though the required multilinear estimates in a weaker norm were given there, which hinted such handling of the bad term. 
It was then made rigorous in subsequent works \cite{K-announcement,CGKO17,P19,KOY20}.}
To make it rigorous, we deduce the following nonlinear estimates from the fundamental $p$-linear estimates assumed in Theorem~\ref{thm:abstract}.

\begin{prop}\label{prop:abstract}
Let $s\in \R$ and assume the hypotheses in Theorem~\ref{thm:abstract}.
Define $\de :=\frac{s-s_1}{s_2-s_1}>0$ if we assume $[B]$ in Theorem~\ref{thm:abstract}.

Then, for any $J\in \Bo{N}$ and $\om \in \ell ^2_s$, the series in $\mathbf{n}$ in $\Sc{N}^{(J)}_R[\om ]_n$, $\Sc{N}^{(J)}_0[\om ]_n$, $\Sc{R}^{(J)}[\om ]_n$, $\Sc{N}^{(J)}[\om ]_n$ converge absolutely for each $n\in \Bo{Z}^d$.
Moreover, for any $M\ge 1$, (with $\Phi _R^{j}$ and $\Phi _{N\!R}^{j}$ suitably defined depending on $M$) we have
\eqq{
\norm{\Sc{N}^{(J)}_R[\om ]-\Sc{N}^{(J)}_R[\ti{\om}]}{\ell ^2_s}&\le CM\Big[ CM^{-\de}\big( \tnorm{\om}{\ell ^2_s}+\tnorm{\ti{\om}}{\ell ^2_s}\big) ^{p-1}\Big] ^J\norm{\om -\ti{\om}}{\ell ^2_s},\\
\norm{\Sc{N}^{(J)}_0[\om ]-\Sc{N}^{(J)}_0[\ti{\om}]}{\ell ^2_s}&\le C\Big[ CM^{-\de}\big( \tnorm{\om}{\ell ^2_s}+\tnorm{\ti{\om}}{\ell ^2_s}\big) ^{p-1}\Big] ^J\norm{\om -\ti{\om}}{\ell ^2_s},\\
\norm{\Sc{R}^{(J)}[\om ]-\Sc{R}^{(J)}[\ti{\om}]}{\ell ^2_s}&\le C\Big[ CM^{-\de}\big( \tnorm{\om}{\ell ^2_s}+\tnorm{\ti{\om}}{\ell ^2_s}\big) ^{p-1}\Big] ^JC'\big( \tnorm{\om}{\ell ^2_s},\tnorm{\ti{\om}}{\ell ^2_s}\big) \norm{\om -\ti{\om}}{\ell ^2_s},\\
\norm{\Sc{N}^{(J)}[\om ]-\Sc{N}^{(J)}[\ti{\om}]}{X}&\le C\Big[ CM^{-\de}\big( \tnorm{\om}{\ell ^2_s}+\tnorm{\ti{\om}}{\ell ^2_s}\big) ^{p-1}\Big] ^{J-1}\big( \tnorm{\om}{\ell ^2_s}+\tnorm{\ti{\om}}{\ell ^2_s}\big) ^{p-1}\norm{\om -\ti{\om}}{\ell ^2_s}
}
for any $\om ,\ti{\om}\in \ell ^2_s$, where $C,C'(\cdot ,\cdot )>0$ are independent of $J$ and $M$.
\end{prop}

\begin{proof}[Proof of Theorem~\ref{thm:abstract} assuming Proposition~\ref{prop:abstract}]
Let $\om ,\ti{\om}\in C_T\ell ^2_s$ be two solutions of \eqref{Abs}.
It follows from Proposition~\ref{prop:abstract} that for any $\eta \in (0,1)$ there exist $M\ge 1$ and $T'\in (0,T]$ depending on $\eta$, $\delta$, and the $C_T\ell ^2_s$ norm of $\om$ and $\ti{\om}$, such that we have
\eq{est:Jth}{&\sup _{t\in [0,T']}\Big( \norm{\Sc{N}^{(J)}_{0}[\om ](t)-\Sc{N}^{(J)}_{0}[\ti{\om}](t)}{\ell ^2_s}+\norm{\int _0^t\Big[ \Sc{N}^{(J)}_R[\om ]-\Sc{N}^{(J)}_R[\ti{\om}]\Big]}{\ell ^2_s}\\
&\qquad\qquad +\norm{\int _0^t\Big[ \Sc{R}^{(J-1)}[\om ]-\Sc{R}^{(J-1)}[\ti{\om}]\Big]}{\ell ^2_s}+\norm{\int _0^t\Big[ \Sc{N}^{(J)}[\om ]-\Sc{N}^{(J)}[\ti{\om}]\Big]}{X}\Big) \\
&\quad \le C\eta ^J\norm{\om -\ti{\om}}{C_{T'}\ell ^2_s}}
for any $J\ge 1$.
Then, we can take the limit $J\to \I$ in the hierarchy, where the sums in $j$ for $\Sc{N}^{(j)}_0[\om ]$, $\Sc{N}^{(j)}_R[\om ]$, $\Sc{R}^{(j)}[\om ]$ all converge absolutely in $\ell ^2_s$ and the bad term $\int _0^t\Sc{N}^{(J)}[\om]$ vanishes in the $X$ norm.
As a result, we get the limit equation:
\eq{eq:limit}{\om _n\Big| _0^t=\sum _{j=1}^{\infty}\Sc{N}^{(j)}_{0}[\om ]_n\Big| _0^t+\int _0^t\,\Big( \sum _{j=1}^{\I}\Sc{N}^{(j)}_R[\om ]_n+\sum _{j=1}^{\I}\Sc{R}^{(j-1)}[\om ]_n\Big) ,\qquad t\in [0,T'],~~n\in \Bo{Z}^d.}
By \eqref{eq:limit} and \eqref{est:Jth}, we can easily show that if the two solutions share the same initial datum, then
\[ \norm{\om -\ti{\om}}{C_{T'}\ell ^2_s}\le C\norm{\om -\ti{\om}}{C_{T'}\ell ^2_s}\sum _{j=1}^\I \eta ^j,\]
which implies that $\om (t)\equiv \ti{\om}(t)$ for $t\in [0,T']$.
Repeating this argument, we have the coincidence on the whole interval where both of two solutions are defined.
This completes the proof of Theorem~\ref{thm:abstract}.
\end{proof}

\subsection{Proof of multilinear estimates}

All we have to do is to prove Proposition~\ref{prop:abstract}.

We want to establish the $[(p-1)J+1]$-linear estimates for all $J$.
One may expect that these estimates follow from $J$ times iteration of the $p$-linear estimates for terms in the equation of the first generation, such as
\[ \norm{\sum _{\mat{n=n_1+\dots +n_p\\ |\phi |\text{\,:\,small}}}|m|\om ^{(1)}_{n_1}\cdots \om^{(p)}_{n_p}}{\ell ^2_s}+\norm{\sum _{\mat{n=n_1+\dots +n_p\\ |\phi |\text{\,:\,large}}}\frac{|m|}{|\phi|}\om ^{(1)}_{n_1}\cdots \om^{(p)}_{n_p}}{\ell ^2_s}\le C\prod _{j=1}^p\norm{\om ^{(j)}}{\ell ^2_s}.\] 
However, such reduction seems impossible.
The reason is that the structure of resonance gets more complicated as the generation proceeds; namely, the phase function in the equation of the $J$-th generation is not $\phi ^J$ but $\ti{\phi}^J=\phi ^1+\phi ^2+\dots +\phi ^J$, in which all variables $n_a$ appearing before are involved.
Therefore, to get nonlinear estimates for every generation by an induction on $J$, we have to prepare fundamental $p$-linear estimates which are \emph{stronger} than just required for the first generation as above.
Actually, the sets of $p$-linear estimates assumed in Theorem~\ref{thm:abstract} are examples of such fundamental estimates, as we see below.

\begin{proof}[Proof of Proposition~\ref{prop:abstract}, Case {$[A]$}]
We begin with giving a precise definition of the resonant/non-resonant decomposition.
Let $M\ge 1$.
We define the sets $\Phi _R^{J}$, $\Phi _{N\!R}^{J}$ as
\eqq{
\Phi _R^{1}&:=\Shugo{\phi ^1}{|\phi ^1|\le 16M},\qquad \Phi _{N\!R}^{1}:=\Shugo{\phi ^1}{|\phi ^1|>16M},\\
\Phi _R^{J}&:=\Shugo{(\phi ^j)_{j=1}^J}{|\phi ^1|>16M,\,|\ti{\phi}^{j}|>16|\ti{\phi}^{j-1}|~(2\le j\le J-1),\, |\ti{\phi}^{J}|\le 16|\ti{\phi}^{J-1}|},\\
\Phi _{N\!R}^{J}&:=\Shugo{(\phi ^j)_{j=1}^J}{|\phi ^1|>16M,\,|\ti{\phi}^{j}|>16|\ti{\phi}^{j-1}|~(2\le j\le J)}\qquad (J\ge 2).
}
It is easily verified that
\eqq{|\ti{\phi}^{j}|>16|\ti{\phi}^{j-1}|\qquad &\Longrightarrow \qquad |\ti{\phi}^j|\sim |\phi ^j|,\\
|\ti{\phi}^{j}|\le 16|\ti{\phi}^{j-1}|\qquad &\Longrightarrow \qquad |\phi ^j|\lec |\ti{\phi}^{j-1}|,}
which implies that for $\phi ^1\in \Phi ^1_R$
\[ 1\gec M^{-1/2}\LR{\phi ^1}^{1/2},\]
for $(\phi ^j)_{j=1}^J\in \Phi _R^J$ ($J\ge 2$)
\eqq{\prod _{j=1}^{J-1}|\ti{\phi}^j|\gec \prod _{j=1}^{J-2}|\ti{\phi}^j|^{1/2}\cdot \prod _{j=1}^{J}\LR{\phi ^j}^{1/2}&\ge 16^{\sum _{j=1}^{J-2}j/2}M^{(J-2)/2}\prod _{j=1}^{J}\LR{\phi ^j}^{1/2}\\
&=2^{(J-1)(J-2)}M^{(J-2)/2}\prod _{j=1}^{J}\LR{\phi ^j}^{1/2},}
and for $(\phi ^j)_{j=1}^J\in \Phi _{N\!R}^J$
\eqq{\prod _{j=1}^{J}|\ti{\phi}^j|\sim &\prod _{j=1}^{J}|\ti{\phi}^j|^{1/2}\cdot \prod _{j=1}^{J}\LR{\phi ^j}^{1/2}\ge 16^{\sum _{j=1}^{J}j/2}M^{J/2}\prod _{j=1}^{J}\LR{\phi ^j}^{1/2}=2^{J(J+1)}M^{J/2}\prod _{j=1}^{J}\LR{\phi ^j}^{1/2},\\
\prod _{j=1}^{J}|\ti{\phi}^j|\sim &\prod _{j=1}^{J}|\ti{\phi}^j|^{\de}\cdot \prod _{j=1}^{J}\LR{\phi ^j}^{1-\de}\ge 16^{\sum _{j=1}^{J}\de j}M^{\de J}\prod _{j=1}^{J}\LR{\phi ^j}^{1-\de}=4^{\de J(J+1)}M^{\de J}\prod _{j=1}^{J}\LR{\phi ^j}^{1-\de}.
}
These inequalities then yield that
\eqq{
|\Sc{N}^{(J)}_R[\om ]_n|&\lec 2^{-(J-1)(J-2)}M^{1-J/2}\sum _{\Sc{T}\in \mathfrak{T}(J)}\sum _{\mathbf{n}\in \mathfrak{N}_n(\Sc{T})}\frac{\prod _{j=1}^J|m^j|}{\prod _{j=1}^{J}\LR{\phi ^j}^{1/2}}\prod _{a\in \Sc{T}_\I}|\om _{n_a}|,\\
|\Sc{N}^{(J)}_0[\om ]_n|&\lec 2^{-J(J+1)}M^{-J/2}\sum _{\Sc{T}\in \mathfrak{T}(J)}\sum _{\mathbf{n}\in \mathfrak{N}_n(\Sc{T})}\frac{\prod _{j=1}^J|m^j|}{\prod _{j=1}^{J}\LR{\phi ^j}^{1/2}}\prod _{a\in \Sc{T}_\I}|\om _{n_a}|,\\
|\Sc{R}^{(J)}[\om ]_n|&\lec 2^{-J(J+1)}M^{-J/2}\sum _{\Sc{T}\in \mathfrak{T}(J)}\sum _{a\in \Sc{T}_\I}\sum _{\mathbf{n}\in \mathfrak{N}_n(\Sc{T})}\frac{\prod _{j=1}^J|m^j|}{\prod _{j=1}^{J}\LR{\phi ^j}^{1/2}}|\Sc{R}[\om ]_{n_a}|\prod _{\mat{b\in \Sc{T}_\I \\b\neq a}}|\om _{n_b}|,\\
|\Sc{N}^{(J+1)}[\om ]_n|&\lec 4^{-\de J(J+1)}M^{-\de J}\sum _{\Sc{T}\in \mathfrak{T}(J)}\sum _{a\in \Sc{T}_\I}\sum _{\mathbf{n}\in \mathfrak{N}_n(\Sc{T})}\frac{\prod _{j=1}^J|m^j|}{\prod _{j=1}^{J}\LR{\phi ^j}^{1-\de}}|\Sc{N}^{(1)}[\om ]_{n_a}|\prod _{\mat{b\in \Sc{T}_\I \\b\neq a}}|\om _{n_b}|.
}

Now that there are only $\phi ^j$ and no $\ti{\phi}^j$ appearing in these expressions, we can iterate the $p$-linear estimates to obtain $[(p-1)J+1]$-linear estimates for every $J$.
To be more precise, we recall how these terms have been obtained from the terms in the equation of the preceding generation:
For each $\Sc{T}\in \mathfrak{T}(J)$ there exist a tree $\Sc{T}'\in \mathfrak{T}(J-1)$ and $a_*\in \Sc{T}'_\I$ such that $\Sc{T}$ is obtained from $\Sc{T}'$ by developing the leaf $a_*$ into the $J$-th node and its $p$ children (denoted by $a_*^1,\dots ,a_*^p$), which originates from the substitution of a $p$-linear form $\Sc{N}^{(1)}[\om ]_{n_{a_*}}$ for $\om _{n_{a_*}}$.
Noticing this process, we see that
\eqq{
&\sum _{\mathbf{n}\in \mathfrak{N}_n(\Sc{T})}\frac{\prod _{j=1}^J|m^j|}{\prod _{j=1}^{J}\LR{\phi ^j}^{\al}}\prod _{a\in \Sc{T}_\I}|\om ^{(a)}_{n_a}|\\
&\hx =\sum _{\mathbf{n}\in \mathfrak{N}_n(\Sc{T}')}\frac{\prod _{j=1}^{J-1}|m^j|}{\prod _{j=1}^{J-1}\LR{\phi ^j}^{\al}}\prod _{\mat{a\in \Sc{T}'_\I \\ a\neq a_*}}|\om ^{(a)}_{n_a}|\cdot \sum _{n_{a_*}=n_{a_*^1}+\dots +n_{a_*^p}}\frac{|m(n_{a_*},n_{a_*^1},\dots ,n_{a_*^p})|}{\LR{\phi (n_{a_*},n_{a_*^1},\dots ,n_{a_*^p})}^\al}\prod _{l=1}^p|\om ^{(a_*^l)}_{n_{a_*^l}}|.
}
Hence, by an induction on $J$ we can easily deduce the following statement from the assumptions $(A1)$, $(A2)$:
There exists $C>0$ such that for any $J\in \Bo{N}$, we have
\begin{align}
\norm{\sum _{\mathbf{n}\in \mathfrak{N}_n(\Sc{T})}\frac{\prod _{j=1}^J|m^j|}{\prod _{j=1}^{J}\LR{\phi ^j}^{1/2}}\prod _{a\in \Sc{T}_\I}|\om ^{(a)}_{n_a}|}{\ell ^2_s}&\le C^J\prod _{a\in \Sc{T}_\I}\tnorm{\om ^{(a)}}{\ell ^2_s}, \label{est:red1} \\
\norm{\sum _{\mathbf{n}\in \mathfrak{N}_n(\Sc{T})}\frac{\prod _{j=1}^J|m^j|}{\prod _{j=1}^{J}\LR{\phi ^j}^{1-\de}}\prod _{a\in \Sc{T}_\I}|\om ^{(a)}_{n_a}|}{X}&\le C^J\min _{a\in \Sc{T}_\I}\Big[ \tnorm{\om ^{(a)}}{X}\prod _{\mat{b\in \Sc{T}_\I \\ b\neq a}}\tnorm{\om ^{(b)}}{\ell ^2_s}\Big]  \label{est:red2}
\end{align}
for any $\Sc{T}\in \mathfrak{T}(J)$.
From \eqref{est:red1} we have
\eqq{
&\norm{\Sc{N}^{(J)}_R[\om ]-\Sc{N}^{(J)}_R[\ti{\om}]}{\ell ^2_s}\\
&\quad\lec \big( (p-1)J+1\big) 2^{-(J-1)(J-2)}\# \mathfrak{T}(J)M\Big[ CM^{-1/2}\big( \tnorm{\om}{\ell ^2_s}+\tnorm{\ti{\om}}{\ell ^2_s}\big) ^{p-1}\Big] ^J\norm{\om -\ti{\om}}{\ell ^2_s},\\
&\norm{\Sc{N}^{(J)}_0[\om ]-\Sc{N}^{(J)}_0[\ti{\om}]}{\ell ^2_s}\\
&\quad\lec \big( (p-1)J+1\big) 2^{-J(J+1)}\# \mathfrak{T}(J)\Big[ CM^{-1/2}\big( \tnorm{\om}{\ell ^2_s}+\tnorm{\ti{\om}}{\ell ^2_s}\big) ^{p-1}\Big] ^J\norm{\om -\ti{\om}}{\ell ^2_s},
}
while combining \eqref{est:red1} with the assumption $(R)$ implies
\eqq{\norm{\Sc{R}^{(J)}[\om ]-\Sc{R}^{(J)}[\ti{\om}]}{\ell ^2_s}&\lec \big( (p-1)J+1\big) ^22^{-J(J+1)}\# \mathfrak{T}(J)\\
&\quad \times \Big[ CM^{-1/2}\big( \tnorm{\om}{\ell ^2_s}+\tnorm{\ti{\om}}{\ell ^2_s}\big) ^{p-1}\Big] ^JC'\big( \tnorm{\om}{\ell ^2_s},\tnorm{\ti{\om}}{\ell ^2_s}\big) \norm{\om -\ti{\om}}{\ell ^2_s},
}
and from \eqref{est:red2} with $(A3)$ that
\eqq{\norm{\Sc{N}^{(J+1)}[\om ]-\Sc{N}^{(J+1)}[\ti{\om}]}{X}&\lec \big( pJ+1\big) \big( (p-1)J+1\big) 4^{-\de J(J+1)}\# \mathfrak{T}(J)\\
&\quad \times \Big[ CM^{-\de}\big( \tnorm{\om}{\ell ^2_s}+\tnorm{\ti{\om}}{\ell ^2_s}\big) ^{p-1}\Big] ^J\big( \tnorm{\om}{\ell ^2_s}+\tnorm{\ti{\om}}{\ell ^2_s}\big) ^{p-1}\norm{\om -\ti{\om}}{\ell ^2_s},
}
where $C,C'>0$ and the implicit constants are independent of $J$ and $M$.

Note that $\big( (p-1)J+1\big) ^22^{-(J-1)(J-2)}\# \mathfrak{T}(J)$, $\big( pJ+1\big) \big( (p-1)J+1\big) 4^{-\de J(J+1)}\# \mathfrak{T}(J)$ are bounded in $J$.
Therefore, we obtain the desired estimates.
\end{proof}

\begin{proof}[Proof of Proposition~\ref{prop:abstract}, Case {$[B]$}]
In this case, we set
\eqq{
\Phi _R^{J}&:=\Shugo{(\phi ^j)_{j=1}^J}{|\ti{\phi}^{j}|>2^jM~(1\le j\le J-1),\, |\ti{\phi}^{J}|\le 2^JM},\\
\Phi _{N\!R}^{J}&:=\Shugo{(\phi ^j)_{j=1}^J}{|\ti{\phi}^{j}|>2^jM~(1\le j\le J)}
}
for $M\ge 1$.
For given $\bfmu =(\mu ^j)_{j=1}^J\in \Bo{Z}^J$, we have
\eqq{(\phi ^j)_{j=1}^J\in \Phi _R^{J}\cap \big( \bfmu +[0,1)^J\big) \quad &\Longrightarrow \quad \prod _{j=1}^{J-1}|\ti{\phi}^j|\ge \prod _{j=1}^{J-1}\max \{ |\ti{\mu}^j|-j,\,2^jM\} ,\\
(\phi ^j)_{j=1}^J\in \Phi _{N\!R}^{J}\cap \big( \bfmu +[0,1)^J\big) \quad &\Longrightarrow \quad \prod _{j=1}^{J}|\ti{\phi}^j|\ge \prod _{j=1}^{J}\max \{ |\ti{\mu}^j|-j,\,2^jM\} ,
}
where $\ti{\mu}^j:=\mu ^1+\cdots +\mu ^j$.
Hence, we may consider estimating terms of the form
\[ T(f;(\om ^{(a)})_{a\in \mathcal{T}_\infty})~:=~\sum _{\bfmu \in \Bo{Z}^J}f(\bfmu )\!\!\sum _{\mat{\mathbf{n}\in \mathfrak{N}_n(\Sc{T})\\ (\phi ^j)_{j=1}^J\in \bfmu +[0,1)^J}}\prod _{j=1}^J|m^j|\prod _{a\in \Sc{T}_\I}\om^{(a)}_{n_a}\]
for $\mathcal{T}\in \mathfrak{T}(J)$ and an appropriate $f:\Bo{Z}^J\to \R$.
In fact, the function $f$ will be taken as one of 
\[ f_R(\bfmu )=\frac{\chi _{|\ti{\mu}^J|\le 2^JM+J}}{\prod\limits _{j=1}^{J-1}\max \{ |\ti{\mu}^j|-j,\,2^jM\} },\qquad f_{N\!R}(\bfmu)=\frac{1}{\prod\limits _{j=1}^{J}\max \{ |\ti{\mu}^j|-j,\,2^jM\} }.\]
Again, (for each fixed $\bfmu$) the series in $\mathbf{n}$ has such a form that an induction on $J$ can be applied.

Note that the above functions $f_R$, $f_{N\!R}$ do not belong to $L^1(\Bo{Z}^J)$, but are in $L^p(\Bo{Z}^J)$ for any $p>1$ and 
\[ \| f_R\| _{\ell ^p(\Bo{Z}^J)}\lec _p 2^JM\cdot \big[ 2^{-\frac{1}{2}J(J+1)}M^{-J}\big] ^{1-\frac{1}{p}},\qquad \| f_{N\!R}\| _{\ell ^p(\Bo{Z}^J)}\lec _p  \big[ 2^{-\frac{1}{2}J(J+1)}M^{-J}\big] ^{1-\frac{1}{p}}. \]
For the moment, we consider general functions $f:\Bo{Z}^J\to \R$.
Iterating $(B1)$ or $(B1)'$ $J$ times (as we did in Case $[A]$ above), we have the estimates
\begin{equation}\label{est:B-1}
\begin{split}
\sup _{\bfmu \in \mathbb{Z}^J}\norm{\sum _{\mat{\mathbf{n}\in \mathfrak{N}_n(\Sc{T})\\ (\phi ^j)_{j=1}^J\in \bfmu +[0,1)^J}}\prod _{j=1}^J|m^j|\prod _{a\in \Sc{T}_\I}|\om^{(a)}_{n_a}|}{\ell ^2_{s_1}}&\le C^J\prod _{a\in \mathcal{T}_\infty} \norm{\om^{(a)}}{\ell ^2_{s_1}},\\
\norm{\sum _{\mathbf{n}\in \mathfrak{N}_n(\Sc{T})}\prod _{j=1}^J|m^j|\prod _{a\in \Sc{T}_\I}|\om^{(a)}_{n_a}|}{\ell ^2_{s_2}}&\le C^J\prod _{a\in \mathcal{T}_\infty} \norm{\om^{(a)}}{\ell ^2_{s_2}},
\end{split}
\end{equation}
and hence,
\eqq{
\norm{T(f;(\om ^{(a)})_{a\in \mathcal{T}_\infty})}{\ell ^2_{s_1}}&\leq \| f\| _{L^1(\mathbb{Z}^J)}C^J\prod _{a\in \mathcal{T}_\infty} \norm{\om^{(a)}}{\ell ^2_{s_1}},\\
\norm{T(f;(\om ^{(a)})_{a\in \mathcal{T}_\infty})}{\ell ^2_{s_2}}&\leq \| f\| _{L^\infty (\mathbb{Z}^J)}C^J\prod _{a\in \mathcal{T}_\infty} \norm{\om^{(a)}}{\ell ^2_{s_2}}.}
By multilinear interpolation, we obtain
\eq{est:B-2}{
\norm{T(f;(\om ^{(a)})_{a\in \mathcal{T}_\infty})}{\ell ^2_s}&\leq \| f\| _{L^p(\mathbb{Z}^J)}C^J\prod _{a\in \mathcal{T}_\infty} \norm{\om^{(a)}}{\ell ^2_s},
}
where $p\in (1,\I )$ satisfies $s=\frac{1}{p}s_1+(1-\frac{1}{p})s_2$ and $1-\frac{1}{p}=\frac{s-s_1}{s_2-s_1}=\de$.
Now, taking $f=f_R$ and $f_{N\!R}$, we have
\eqq{&\norm{\Sc{N}^{(J)}_R[\om ]}{\ell ^2_s}\le C2^JM2^{-\frac{\de}{2}J(J+1)}\# \FR{T}(J)\big[ CM^{-\de}\norm{\om}{\ell ^2_{s}}^{p-1}\big] ^J\norm{\om}{\ell ^2_{s}},\\
&\norm{\Sc{N}^{(J)}_0[\om ]}{\ell ^2_s}\le C2^{-\frac{\de}{2}J(J+1)}\# \FR{T}(J)\big[ CM^{-\de}\norm{\om}{\ell ^2_{s}}^{p-1}\big] ^J\norm{\om}{\ell ^2_{s}},}
which implies the desired estimates for $\Sc{N}^{(J)}_R[\om ]$ and $\Sc{N}^{(J)}_0[\om ]$ when $\ti{\om}=0$.
The same argument applies to the difference estimates.
The estimate of $\Sc{R}^{(J)}[\om ]$ follows from that of  $\Sc{N}^{(J)}_0[\om ]$ and $(R)$.
Finally, $\Sc{N}^{(J+1)}[\om ]$ can be estimated in $X$ by using $(B2)$ and $(B2)'$ instead, together with $(B3)$.
\end{proof}

\begin{rem}\label{rem:Bgeneral}
Clearly, the above argument for $[B]$ is still valid if the spaces $\ell ^2_s$, $\ell ^2_{s_1}$, and $\ell ^2_{s_2}$ in the assumption $[B]$ are generalized to $Y$, $Y_1$, and $Y_2$, Banach spaces with the property \eqref{absolutenorm}, such that $Y$ is obtained by complex interpolation between $Y_1$ and $Y_2$.

If $Y$ is a Banach space satisfying \eqref{absolutenorm} and $Y_1$, $Y_2$ are defined by the norms $\| \LR{n}^{-\nu_1}\om \|_{Y}$ and $\| \LR{n}^{\nu _2}\om\|_Y$ for some $\nu_1,\nu_2>0$, we can deduce the estimate corresponding to \eqref{est:B-2} (with a larger constant $C$) from those corresponding to \eqref{est:B-1} by a direct calculation without using multilinear interpolation.
To see this, we first observe that the property \eqref{absolutenorm} of the $Y$ norm implies
\eqq{
\norm{|\om |^{1-\delta}|\ti{\om}|^{\delta}}{Y}&\leq \bigg\| |\om |^{1-\delta}|\ti{\om}|^{\delta}\chi \bigg[ \frac{|\om _n|}{\| \om \|_Y}\!\leq\!\frac{|\ti{\om}_n|}{\| \ti{\om}\|_Y}\bigg] \bigg\| _{Y}+\bigg\| |\om |^{1-\delta}|\ti{\om}|^{\delta}\chi \bigg[ \frac{|\om _n|}{\| \om \|_Y}\!>\!\frac{|\ti{\om}_n|}{\| \ti{\om}\|_Y}\bigg] \bigg\| _{Y}\\
&\leq C\bigg\| \bigg[ \frac{\| \om\|_Y}{\| \ti{\om}\|_Y}\bigg] ^{1-\delta}\ti{\om}\bigg\| _{Y}+C\bigg\| \bigg[ \frac{\| \ti{\om}\|_Y}{\| \om\|_Y}\bigg] ^{\delta}\om \bigg\|_{Y}=2C\| \om \|_{Y}^{1-\delta}\| \ti{\om}\|_{Y}^{\delta}
}
for any $\delta \in (0,1)$ and $\om ,\ti{\om}\in Y$.
On the other hand, by the H\"older inequality we have (with $\delta =\frac{\nu _1}{\nu _1+\nu _2}$) 
\eqq{
\big| T(f;(\om ^{(a)})_{a\in \mathcal{T}_\infty})\big| 
&\leq \Big[ \LR{n}^{-\nu_1}\sum _{\bfmu}|f(\bfmu )|^{\frac{1}{1-\delta}}\sum _{\mat{\mathbf{n}\in \mathfrak{N}_n(\mathcal{T})\\ (\phi ^j)\in \bfmu +[0,1)^J}}\prod _j|m^j|\prod_a\LR{n_a}^{\nu_1}|\om ^{(a)}_{n_a}|\Big] ^{1-\delta}\\
&\quad \times \Big[ \LR{n}^{\nu_2}\sum _{\mathbf{n}\in \mathfrak{N}_n(\mathcal{T})} \prod _j|m^j|\prod_a\LR{n_a}^{-\nu_2}|\om ^{(a)}_{n_a}|\Big] ^{\delta},
}
Therefore, assuming the estimates corresponding to \eqref{est:B-1}, we obtain
\eqq{
\norm{T(f;(\om ^{(a)})_{a\in \mathcal{T}_\infty})}{Y}&\lesssim \norm{\sum _{\bfmu}|f(\bfmu )|^{\frac{1}{1-\delta}}\sum _{\mat{\mathbf{n}\in \mathfrak{N}_n(\mathcal{T})\\ (\phi ^j)\in \bfmu +[0,1)^J}}\prod _j|m^j|\prod_a\LR{n_a}^{\nu_1}|\om ^{(a)}_{n_a}|}{Y_1}^{1-\delta}\\
&\quad \times \norm{\sum _{\mathbf{n}\in \mathfrak{N}_n(\mathcal{T})} \prod _j|m^j|\prod_a\LR{n_a}^{-\nu_2}|\om ^{(a)}_{n_a}|}{Y_2}^{\delta}\\
&\lesssim \norm{f}{L^{\frac{1}{1-\delta}}(\mathbb{Z}^J)}\Big[ C^J\prod _a\norm{\LR{n}^{\nu_1}\om ^{(a)}}{Y_1}\Big] ^{1-\delta}\Big[ C^J\prod _a\norm{\LR{n}^{-\nu _2}\om ^{(a)}}{Y_2}\Big] ^{\delta} \\
&=\| f\|_{L^p(\mathbb{Z}^J)}C^J\prod _a\norm{\om ^{(a)}}{Y},
}
as desired.
\end{rem}

\subsection{Comments on the case of a system and multiple nonlinear terms}
\label{subsec:multiple}

It is straightforward to adapt the proof of Theorem~\ref{thm:abstract} presented above to the case of a system and multiple principal nonlinear terms.

As an example, let us consider the following system of two equations:
\[
\left\{ \begin{aligned}
\p _t\om _n&=\sum _{i=1}^{I_1}\bigg[ \sum _{n=n_1+\dots +n_{p_{1,i}+q_{1,i}}}e^{it\phi_{1,i}}m_{1,i}\om _{n_1}\cdots \om _{n_{p_{1,i}}}\Om _{n_{p_{1,i}+1}}\cdots \Om _{n_{p_{1,i}+q_{1,i}}} \bigg] +\Sc{R}_1[\om ,\Om ]_n, \\
\p _t\Om _n&=\sum _{i=1}^{I_2}\bigg[ \sum _{n=n_1+\dots +n_{p_{2,i}+q_{2,i}}}e^{it\phi_{2,i}}m_{2,i}\om _{n_1}\cdots \om _{n_{p_{2,i}}}\Om _{n_{p_{2,i}+1}}\cdots \Om _{n_{p_{2,i}+q_{2,i}}} \bigg] +\Sc{R}_2[\om ,\Om ]_n,
\end{aligned}\right.
\]
where $I_1,I_2$ are the numbers of nonlinear terms in the equations;  non-negative integers $p_{*,i}$ and $q_{*,i}$ denote the degree of each term in $\om$ and $\Om$ satisfying $P_{*,i}:=p_{*,i}+q_{*,i}\geq 2$
; $\phi _{*,i}=\phi _{*,i}(n,n_1,\dots ,n_{p_{*,i}+q_{*,i}})\in \R$ and $m_{*,i}=m_{*,i}(n,n_1,\dots ,n_{p_{*,i}+q_{*,i}})\in \Bo{C}$ are the phase and multiplier parts of each term; $\Sc{R}_*[\om ,\Om ]$ is the remainder part of each equation.
We set $I:=\max \{ I_1,I_2\}$ and $P:=\max _{*,i}P_{*,i}$, which is the highest degree of (the principal part of) the nonlinearity.

Then, the fundamental multilinear estimates for unconditional uniqueness of solutions $(\om ,\Om )$ in $C_T(\ell ^2_s\times \ell ^2_r)$ are stated as follows:
\eqs{
\norm{\Sc{R}_1[\om ,\Om ]}{C_T\ell ^2_s}+\norm{\Sc{R}_2[\om ,\Om ]}{C_T\ell ^2_r}\le C\big( \tnorm{\om}{C_T\ell ^2_s},\tnorm{\Om}{C_T\ell ^2_r}\big),\\
\norm{\Sc{R}_1[\om ,\Om ]-\Sc{R}_1[\ti{\om} ,\ti{\Om} ]}{C_T\ell ^2_s}+\norm{\Sc{R}_2[\om ,\Om ]-\Sc{R}_2[\ti{\om} ,\ti{\Om} ]}{C_T\ell ^2_r}\hspace*{50pt}\\
\hspace*{50pt}\le C\big( \tnorm{\om ,\ti{\om}}{C_T\ell ^2_s},\tnorm{\Om ,\ti{\Om}}{C_T\ell ^2_r}\big) \big( \tnorm{\om -\ti{\om}}{C_T\ell ^2_s}+\tnorm{\Om -\ti{\Om}}{C_T\ell ^2_r}\big) 
}
instead of $(R)$, 
\eqq{
&\norm{\sum _{n=n_1+\dots +n_{P_{1,i}}}\frac{|m_{1,i}|}{\LR{\phi_{1,i}}^{1/2}}\prod _{j=1}^{p_{1,i}}\om ^{(j)}_{n_j}\prod _{k=1}^{q_{1,i}}\Om ^{(k)}_{n_{p_{1,i}+k}}}{\ell ^2_s}\lec \prod _{j=1}^{p_{1,i}}\tnorm{\om ^{(j)}}{\ell ^2_s}\prod _{k=1}^{q_{1,i}}\tnorm{\Om ^{(k)}}{\ell ^2_r}\qquad (1\le i\le I_1),\\
&\norm{\sum _{n=n_1+\dots +n_{P_{2,i}}}\frac{|m_{2,i}|}{\LR{\phi_{2,i}}^{1/2}}\prod _{j=1}^{p_{2,i}}\om ^{(j)}_{n_j}\prod _{k=1}^{q_{2,i}}\Om ^{(k)}_{n_{p_{2,i}+k}}}{\ell ^2_r}\lec \prod _{j=1}^{p_{2,i}}\tnorm{\om ^{(j)}}{\ell ^2_s}\prod _{k=1}^{q_{2,i}}\tnorm{\Om ^{(k)}}{\ell ^2_r}\qquad (1\le i\le I_2)
}
instead of $(A1)$,
\eqq{
&\norm{\sum _{n=n_1+\dots +n_{P_{1,i}}}\frac{|m_{1,i}|}{\LR{\phi_{1,i}}^{1-\de}}\prod _{j=1}^{p_{1,i}}\om ^{(j)}_{n_j}\prod _{k=1}^{q_{1,i}}\Om ^{(k)}_{n_{p_{1,i}+k}}}{X_1}\\
&\lec \min \bigg\{ \min _{1\le j\le p_{1,i}}\tnorm{\om ^{(j)}}{X_1}\prod _{\mat{l=1 \\ l\neq j}}^{p_{1,i}}\tnorm{\om ^{(l)}}{\ell ^2_s}\prod _{k=1}^{q_{1,i}}\tnorm{\Om ^{(k)}}{\ell ^2_r},~\min _{1\le k\le q_{1,i}}\tnorm{\Om ^{(k)}}{X_2}\prod _{\mat{l=1 \\ l\neq k}}^{q_{1,i}}\tnorm{\Om ^{(l)}}{\ell ^2_r}\prod _{j=1}^{p_{1,i}}\tnorm{\om ^{(j)}}{\ell ^2_s}\bigg\} \\[-5pt]
&\hspace{370pt} (1\le i\le I_1), \\[-5pt]
&\norm{\sum _{n=n_1+\dots +n_{P_{2,i}}}\frac{|m_{2,i}|}{\LR{\phi_{2,i}}^{1-\de}}\prod _{j=1}^{p_{2,i}}\om ^{(j)}_{n_j}\prod _{k=1}^{q_{2,i}}\Om ^{(k)}_{n_{p_{2,i}+k}}}{X_2}\\
&\lec \min \bigg\{ \min _{1\le j\le p_{2,i}}\tnorm{\om ^{(j)}}{X_1}\prod _{\mat{l=1 \\ l\neq j}}^{p_{2,i}}\tnorm{\om ^{(l)}}{\ell ^2_s}\prod _{k=1}^{q_{2,i}}\tnorm{\Om ^{(k)}}{\ell ^2_r},~\min _{1\le k\le q_{2,i}}\tnorm{\Om ^{(k)}}{X_2}\prod _{\mat{l=1 \\ l\neq k}}^{q_{2,i}}\tnorm{\Om ^{(l)}}{\ell ^2_r}\prod _{j=1}^{p_{2,i}}\tnorm{\om ^{(j)}}{\ell ^2_s}\bigg\} \\[-5pt]
&\hspace{370pt}(1\le i\le I_2)
}
(for suitable Banach spaces $X_1,X_2$ and some $\de $) instead of $(A2)$, and
\eqq{
&\norm{\sum _{n=n_1+\dots +n_{P_{1,i}}}|m_{1,i}|\prod _{j=1}^{p_{1,i}}\om ^{(j)}_{n_j}\prod _{k=1}^{q_{1,i}}\Om ^{(k)}_{n_{p_{1,i}+k}}}{X_1}\lec \prod _{j=1}^{p_{1,i}}\tnorm{\om ^{(j)}}{\ell ^2_s}\prod _{k=1}^{q_{1,i}}\tnorm{\Om ^{(k)}}{\ell ^2_r}\qquad (1\le i\le I_1),\\
&\norm{\sum _{n=n_1+\dots +n_{P_{2,i}}}|m_{2,i}|\prod _{j=1}^{p_{2,i}}\om ^{(j)}_{n_j}\prod _{k=1}^{q_{2,i}}\Om ^{(k)}_{n_{p_{2,i}+k}}}{X_2}\lec \prod _{j=1}^{p_{2,i}}\tnorm{\om ^{(j)}}{\ell ^2_s}\prod _{k=1}^{q_{2,i}}\tnorm{\Om ^{(k)}}{\ell ^2_r}\qquad (1\le i\le I_2)
}
instead of $(A3)$.
The estimates $(B1)$--$(B2)'$ are extended similarly; we omit here to write them down explicitly.

In what follows, we focus on the case $[A]$ as an example.
For the proof of uniqueness, we first modify the notation of ordered tree:
For a partially ordered set $\Sc{T}$, denote the subset of all maximal elements by $\Sc{T}_\I$ (i.e., $a\in \Sc{T}_\I$ means there is no $b\in \Sc{T}\setminus \{ a\}$ satisfying $a\preceq b$), and $\Sc{T}_0:=\Sc{T}\setminus \Sc{T}_\I$.
An element of $\Sc{T}_0$ or $\Sc{T}_\I$ will be called a node or a leaf.
Consider the triplet $(\Sc{T}, \iota ,\kappa )$ of a finite partially ordered set $\Sc{T}$, a map $\iota :\Sc{T}\to \{ 1,2\}$, and a map $\kappa :\Sc{T}_0\to \{ 1,2,\dots ,I\}$.
Then, $\mathfrak{T}_1(J)$ [resp.~$\mathfrak{T}_2(J)$] is defined as the set of all $(\Sc{T}, \iota ,\kappa )$ satisfying the following properties:
\begin{enumerate}
\item $\Sc{T}$ has the least element $r$ called the \emph{root}.
\item For each $a\in \Sc{T}\setminus \{ r\}$, there exists a unique element $b\in \Sc{T}$ called the \emph{parent} of $a$ such that $b\neq a$, $b\preceq a$, and that $b\preceq c\preceq a$ implies $c=a$ or $c=b$.
We say $a$ is a \emph{child} of $b$ if $b$ is the parent of $a$.
Note that $a\in \Sc{T}$ has a child if and only if $a\in \Sc{T}_0$.
\item 
$\# \Sc{T}_0=J$, and these nodes are numbered from $1$ to $J$ so that $a^{j_1}\preceq a^{j_2}$ implies $j_1\le j_2$, where we write the $j$-th node of $\Sc{T}$ as $a^j$.
\item For each $a\in \Sc{T}_0$, $\kappa (a) \in \{ 1,2,\dots ,I_{\iota (a)}\}$ and $a$ has exactly $P_{\iota (a),\kappa (a)}$ children, denoted by $a_1,\dots , a_{P_{\iota (a),\kappa (a)}}$.
\item The map $\iota$ satisfies that
\begin{itemize}
\item $\iota (r)=1$ [resp.~$\iota (r)=2$], and
\item for each $a\in \Sc{T}_0$, $\iota (a_j)=\begin{cases}1& (1\le j\le p_{\iota (a),\kappa (a)}), \\ 2 &(p_{\iota (a),\kappa (a)}+1\le j\le p_{\iota (a),\kappa (a)}+q_{\iota (a),\kappa (a)}).\end{cases}$
\end{itemize}
\end{enumerate}
We observe that an ordered tree $(\Sc{T},\iota ,\kappa )$ can create several ordered trees of the next generation by developing one of its leaves into the new node and its children.
By the property (iv), for each choice of a leaf $a\in \Sc{T}_0$ there are $I_{\iota (a)}$ choices for the value of $\kappa$ at the new node $a$, while the values of $\iota$ at the $P_{\iota (a),\kappa (a)}$ new children of $a$ are uniquely determined by the property (v).
Hence, for each $(\Sc{T},\iota ,\kappa )\in \mathfrak{T}_*(J)$ the number of generated ordered trees in $\mathfrak{T}_*(J+1)$ is equal to $\sum _{a\in \Sc{T}_0}I_{\iota (a)}$ and bounded by $ \big( (P-1)J+1\big) I$.
Conversely, any ordered tree is obtained from a specific ordered tree of the previous generation by this procedure.
We therefore see that
\[ \# \mathfrak{T}_*(J)\leq \prod _{j=0}^{J-1}\big[ \big( (P-1)j+1\big) I\big] \leq (P-1)^JI^JJ! \qquad (*\in \{ 1,2\} ).\]
The following notation is also naturally adapted to the present setting:
\begin{itemize}
\item Let $J\in \Bo{N}$, $*\in \{ 1,2\}$ and $(\Sc{T},\iota ,\kappa )\in \FR{T}_*(J)$.
We write $\FR{N}(\Sc{T},\iota ,\kappa )$ to denote the set of all maps $\mathbf{n}=\{ n_a\} _{a\in \Sc{T}}:\Sc{T}\to \Bo{Z}^d$ satisfying the following property:
For each $a\in \Sc{T}_0$, with its children being denoted by $a_1,a_2,\dots ,a_{P_{\iota (a),\kappa (a)}}$, it holds that 
\[ n_a=n_{a_1}+n_{a_2}+\cdots +n_{a_{P_{\iota (a),\kappa (a)}}}.\]
For each $n$, we define $\mathfrak{N}_n(\mathcal{T},\iota,\kappa):=\{ \mathbf{n}\in \mathfrak{N}(\mathcal{T},\iota,\kappa)\,|\,n_{\mathrm{root}}=n\}$.
\item Given $(\Sc{T},\iota ,\kappa )\in \FR{T}_*(J)$ and $\mathbf{n}\in \FR{N}(\Sc{T},\iota ,\kappa )$, we write 
\eqs{
\phi ^j:=\phi _{\iota (a^j),\kappa (a^j)}\big( n_{a^j}, n_{a^j_1},\dots ,n_{a^j_{P_{\iota (a^j),\kappa (a^j)}}}\big) ,\\
m^j:=m_{\iota (a^j),\kappa (a^j)}\big( n_{a^j}, n_{a^j_1},\dots ,n_{a^j_{P_{\iota (a^j),\kappa (a^j)}}}\big)
}
for $1\le j\le J$, where $a^j$ is the $j$-th node of $\Sc{T}$ and $a^j_1,a^j_2,\dots ,a^j_{P_{\iota (a^j),\kappa (a^j)}}$ are its children.
\end{itemize}
With these extended notions, the system can be written as 
\[
\left\{ \begin{aligned}
\p _t\om _n&=\sum _{(\Sc{T},\iota ,\kappa )\in\mathfrak{T}_1(1)}\bigg[ \sum _{\mathbf{n}\in \mathfrak{N}_n(\Sc{T},\iota ,\kappa )}e^{it\phi ^1}m ^1\prod _{\mat{a\in \Sc{T}_\I \\ \iota (a)=1}}\om _{n_a}\prod _{\mat{b\in \Sc{T}_\I \\ \iota (b)=2}}\Om_{n_b} \bigg] ~+\Sc{R}_1[\om ,\Om ]_n, \\
\p _t\Om _n&=\sum _{(\Sc{T},\iota ,\kappa )\in\mathfrak{T}_2(1)}\bigg[ \sum _{\mathbf{n}\in \mathfrak{N}_n(\Sc{T},\iota ,\kappa )}e^{it\phi ^1}m ^1\prod _{\mat{a\in \Sc{T}_\I \\ \iota (a)=1}}\om _{n_a}\prod _{\mat{b\in \Sc{T}_\I \\ \iota (b)=2}}\Om_{n_b} \bigg] ~+\Sc{R}_2[\om ,\Om ]_n.
\end{aligned}\right.
\]

We now proceed to the iteration of NFR with the same rule for resonant/non-resonant decomposition as before (i.e., we use the same sets $\Phi ^J_{R},\Phi ^J_{N\!R}\subset \R ^J$).
Then, after the $(J-1)$-th NFR, we get the system of the $J$-th generation as
\[
\left\{ 
\begin{aligned}
&\om _n\Big| _0^t=\sum _{j=1}^{J-1}\Sc{N}^{(j)}_{1,0}[\om ,\Om ]_n\Big| _0^t+\int _0^t\,\Big( \sum _{j=1}^{J-1}\Sc{N}^{(j)}_{1,R}[\om ,\Om ]_n+\sum _{j=0}^{J-1}\Sc{R}^{(j)}_1[\om ,\Om]_n+\Sc{N}^{(J)}_1[\om ,\Om ]_n\Big) ,\\
&\Om _n\Big| _0^t=\sum _{j=1}^{J-1}\Sc{N}^{(j)}_{2,0}[\om ,\Om ]_n\Big| _0^t+\int _0^t\,\Big( \sum _{j=1}^{J-1}\Sc{N}^{(j)}_{2,R}[\om ,\Om ]_n+\sum _{j=0}^{J-1}\Sc{R}^{(j)}_2[\om ,\Om]_n+\Sc{N}^{(J)}_2[\om ,\Om ]_n\Big) ,
\end{aligned}\right.
\]
where 
\eqq{
\Sc{N}^{(J)}_{*,R}[\om ,\Om ]_n:=(-1&)^{J-1}\sum _{(\Sc{T},\iota ,\kappa )\in \mathfrak{T}_*(J)}\sum _{\mat{\mathbf{n}\in \mathfrak{N}_n(\Sc{T},\iota ,\kappa )\\ (\phi ^j)_{j=1}^{J}\in \Phi _{R}^{J}}}\frac{e^{it\ti{\phi}^J}\prod\limits _{j=1}^Jm^j}{\prod\limits _{j=1}^{J-1}i\ti{\phi}^j}\prod _{\mat{a\in \Sc{T}_\I \\ \iota (a)=1}}\om _{n_a}\prod _{\mat{b\in \Sc{T}_\I \\ \iota (b)=2}}\Om_{n_b},\\
\Sc{R}^{(J)}_{*}[\om ,\Om ]_n:=(-1&)^J\sum _{(\Sc{T},\iota ,\kappa )\in \mathfrak{T}_*(J)}\sum _{\mat{\mathbf{n}\in \mathfrak{N}_n(\Sc{T},\iota ,\kappa )\\ (\phi ^j)_{j=1}^{J}\in \Phi _{N\!R}^{J}}}\frac{e^{it\ti{\phi}^J}\prod\limits _{j=1}^Jm^j}{\prod\limits _{j=1}^{J}i\ti{\phi}^j}\\
&\qquad \times \bigg\{ \sum _{\mat{a\in \Sc{T}_\I \\ \iota (a)=1}}\Sc{R}_1[\om ,\Om ]_{n_a}\prod _{\mat{a'\in \Sc{T}_\I \\ \iota (a')=1,\,a'\neq a}}\om _{n_{a'}}\prod _{\mat{b\in \Sc{T}_\I \\ \iota (b)=2}}\Om_{n_b}
\\
&\qquad\qquad\qquad +\sum _{\mat{b\in \Sc{T}_\I \\ \iota (b)=2}}\Sc{R}_2[\om, \Om ]_{n_b}\prod _{\mat{a\in \Sc{T}_\I \\ \iota (a)=1}}\om _{n_{a}}\prod _{\mat{b'\in \Sc{T}_\I \\ \iota (b')=2,\,b'\neq b}}\Om_{n_{b'}} \bigg\} ,
}
and so on (with the same notation $\ti{\phi}^j:=\phi ^1+\dots +\phi ^j$).
The required estimate of $\Sc{N}_{1,R}^{(J)}[\om ,\Om ]$, for instance, is reduced by the same argument as before to evaluating
\eq{term-s}{
\sum _{\mathbf{n}\in \mathfrak{N}_n(\Sc{T},\iota ,\kappa )}\frac{\prod _{j=1}^J|m^j|}{\prod _{j=1}^{J}\LR{\phi ^j}^{1/2}}\prod _{\mat{a\in \Sc{T}_\I \\ \iota (a)=1}}\om ^{(a)}_{n_a}\prod _{\mat{b\in \Sc{T}_\I \\ \iota (b)=2}}\Om ^{(b)}_{n_b}
}
in $\ell ^2_s$ for each $(\Sc{T},\iota ,\kappa )\in \mathfrak{T}_1(J)$.
Here, we recall that there are unique tree $(\Sc{T}',\iota',\kappa' )\in \mathfrak{T}_1(J-1)$, $a^*\in \Sc{T}'_\I$ and $\kappa (a^*)\in \{ 1,2,\dots ,I_{\iota' (a^*)}\}$ such that $(\Sc{T},\iota ,\kappa )$ is obtained from $(\Sc{T}',\iota',\kappa' )$ by changing the leaf $a^*$ into the $J$-th node, extending the map $\kappa'$ to $\kappa$ with the value at $a^*$, adding the $P_{\iota' (a^*),\kappa (a^*)}$ children of $a^*$, and extending the map $\iota'$ to $\iota$ with the values at the children of $a^*$ determined by the property (v).
In particular, we have
\eqq{
\eqref{term-s}&=\sum _{\mathbf{n}\in \mathfrak{N}_n(\Sc{T},\iota ,\kappa )}\frac{\prod _{j=1}^{J-1}|m^j|}{\prod _{j=1}^{J-1}\LR{\phi ^j}^{1/2}}\prod _{\mat{a\in \Sc{T}'_\I \\ \iota' (a)=1,\,a\neq a^*}}\om ^{(a)}_{n_a}\prod _{\mat{b\in \Sc{T}'_\I \\ \iota' (b)=2}}\Om ^{(b)}_{n_b}\\
&\qquad \times \sum _{n_{a^*}=n_{a^*_1}+\dots +n_{a^*_{P_{1,\kappa (a^*)}}}}\frac{\big| m_{1,\kappa (a^*)}\big|}{\LR{\phi _{1,\kappa (a^*)}}^{1/2}}\prod _{j=1}^{p_{1,\kappa (a^*)}}\om ^{(a^*_j)}_{n_{a^*_j}}\prod _{k=1}^{q_{1,\kappa (a^*)}}\Om ^{(a^*_{p_{1,\kappa (a^*)}+k})}_{n_{a^*_{p_{1,\kappa (a^*)}+k}}}
}
if $\iota'(a^*)=1$, for instance.
In view of this representation, we can verify by the fundamental multilinear estimates and an induction on $J$ that
\[ \tnorm{\eqref{term-s}}{\ell ^2_s}\le C^J\prod _{\mat{a\in \Sc{T}_\I \\ \iota (a)=1}}\tnorm{\om ^{(a)}}{\ell ^2_s}\prod _{\mat{b\in \Sc{T}_\I \\ \iota (b)=2}}\tnorm{\Om ^{(b)}}{\ell ^2_r}\]
for any $J\in \Bo{N}$ and $(\Sc{T},\iota ,\kappa )\in \mathfrak{T}_1(J)$.
As a consequence, we obtain
\[ \norm{\Sc{N}_{1,R}^{(J)}[\om ,\Om ]}{\ell ^2_s}\lec M\Big[ CM^{-1/2}\Big\{ (\tnorm{\om}{\ell ^2_s}+\tnorm{\Om}{\ell ^2_r})+(\tnorm{\om}{\ell ^2_s}+\tnorm{\Om}{\ell ^2_r})^{P-1}\Big\} \Big] ^{J}(\tnorm{\om}{\ell ^2_s}+\tnorm{\Om}{\ell ^2_r}),\]
and similarly for the estimates on other terms and the corresponding difference estimates.

Once the counterpart of Proposition~\ref{prop:abstract} is ready, the rest of the proof of unconditional uniqueness is analogous to the single equation case.

%%%%%%%%%%%%%%%%%%%%%%%%%%%%%%%%%%%%%%%%%%%%%%
% cNLS
%%%%%%%%%%%%%%%%%%%%%%%%%%%%%%%%%%%%%%%%%%%%%%

%\bigskip
\section{Application to the cubic NLS in higher dimension} 
\label{section:cNLS}

This and the subsequent three sections will be devoted to providing applications of the abstract framework given in Theorem~\ref{thm:abstract} to various problems.
 
As the first application, let us consider the cubic NLS on $\T^2$:
\eq{cNLS}{
\left\{ \begin{alignedat}{2}
~i\p _tu+\Delta u&=\pm |u|^2u,&\qquad &(t,x)\in [0,T]\times \T ^2,\\
u(0,x)&=u_0(x), & &x\in \mathbb{T}^2.
\end{alignedat}\right.
}

In the context of UU by the NFR approach, only the one-dimensional and cubic case has been studied before; see \cite{GKO13,P19,CHKP19a,KOY20,CHKP19b,OWp}.
For higher dimensions or higher degree of nonlinearities, complicated structure of resonance makes it substantially more involved to estimate the multilinear terms of arbitrarily high degrees arising in the infinite NFR machinery.
Nevertheless, our abstract framework yields UU in a wide range of regularity; see \cite{K-NLSp} for the full result.
Here, we focus on the two-dimensional cubic case to illustrate how easily our framework can be applied.

\begin{thm}\label{thm:cNLS}
Unconditional uniqueness of solutions to the Cauchy problem \eqref{cNLS} holds in $H^s(\T^2)$ for $s\ge 2/3$.
\end{thm} 

Recently, Herr and Sohinger \cite{HS19} used a different method to prove UU for the cubic NLS on arbitrary rectangular torus in general dimensions.%
\footnote{The regularity assumption in Theorem~\ref{thm:cNLS} is more restrictive than that in \cite{HS19}, where in the two-dimensional case UU in $H^s$ for $s>7/12$ was shown for the torus with \emph{arbitrary} aspect ratio.
Note that we can refine the argument in the proof of Theorem~\ref{thm:cNLS} to improve the regularity range to $s>2/5$ in the case of \emph{rational} torus; see \cite{K-NLSp}.}
See \cite{CH19} for a related result on the quintic NLS in dimension three.

\subsection{Reduction to the fundamental trilinear estimate}

Let $u(t)\in C_TH^{2/3}(\T^2)$ be a (distributional) solution of \eqref{cNLS}.
By the Sobolev embedding, the cubic nonlinearity is well-defined as an $L^2$ function.
Set $\hat{v}_n(t):=\F [e^{-it\Delta}u(t)] =e^{it|n|^2}(\F u(t))_n$, then $\hat{v}(t)$ satisfies
\eq{eq_v1}{\p _t\hat{v}_n(t)=c\sum _{\mat{n_1,n_2,n_3\in \Bo{Z}^2\\ n=n_1+n_2+n_3}}e^{it\Phi}\hat{v}_{n_1}(t)\bbar{\hat{v}_{-n_2}(t)}\hat{v}_{n_3}(t),\quad n\in \Bo{Z}^2,}
where $c$ is a suitable constant and
\[ \Phi =\Phi (n,n_1,n_2,n_3):=\,|n|^2-|n_1|^2+|n_2|^2-|n_3|^2.\]
To apply Theorem~\ref{thm:abstract}, we consider the system for $\om _n(t):=\hhat{v}_n(t)$ and $\psi _n(t):=\bbar{\hhat{v}_{-n}(t)}$:
\eq{cNLS'}{
\left\{ ~
\begin{aligned}
\p _t\om_n&=c\sum _{\mat{n_1,n_2,n_3\in \Bo{Z}^2\\ n=n_1+n_2+n_3}}e^{it\Phi}\om _{n_1}\psi _{n_2}\om _{n_3},\\
\p _t\psi _n&=\bbar{c}\sum _{\mat{n_1,n_2,n_3\in \Bo{Z}^2\\ n=n_1+n_2+n_3}}e^{-it\Phi}\psi _{n_1}\om _{n_2}\psi _{n_3},
\end{aligned}
\right.
\qquad n\in \Bo{Z}^2.
}
It then suffices to prove unconditional uniqueness for \eqref{cNLS'} in $\ell ^2_{2/3}(\Bo{Z}^2)\times \ell ^2_{2/3}(\Bo{Z}^2)$.

We shall apply Theorem~\ref{thm:abstract} [B] with $\Sc{R}\equiv 0$ and $X=\ell ^2\times \ell ^2$.
Note that the equations for $\om _n$ and $\psi _n$ have the common phase function $\Phi$ up to signs. 
By this symmetry, it is sufficient to show the fundamental trilinear estimates $(B1)$--$(B3)$ only for one component of $(\om _n,\psi _n)$. 
If we choose $s_2>1$, then the estimates $(B1)'$, $(B2)'$, and $(B3)$ are easy consequences of the Sobolev embedding.
Since $(B1)$ follows from $(B2)$ as long as $0\le s_1<s=2/3$, the proof of Theorem~\ref{thm:cNLS} is reduced to only showing 
\eq{est:cNLS}{
\sup _{\mu \in \Bo{Z}}\Big\| \sum _{\mat{n_1,n_2,n_3\in \Bo{Z}^2\\ n=n_1-n_2+n_3,\, \Phi =\mu}}\om ^{(1)} _{n_1}\om ^{(2)}_{n_2}\om ^{(3)}_{n_3}\Big\| _{\ell ^2}\lec \min _{1\le j\le 3}\norm{\om ^{(j)}}{\ell ^2}\prod _{\mat{l=1 \\l\neq j}}^3\norm{\om ^{(l)}}{\ell ^2_{s_1}}
}
for some $0\le s_1<2/3$.
(We have replaced $n_2$ with $-n_2$ for convenience.)

\subsection{Proof of the trilinear estimate}

We shall verify the estimate \eqref{est:cNLS} for any $s_1>0$.
The proof is based on the following combinatorial tool.
\begin{lem}\label{lem:numbertheory}
For any $\eta >0$ there exists $C>0$ such that
\[ \# \Shugo{n\in \Bo{Z}^2}{|n-n_0|^2=\mu ,\,n\in B_R}\le CR^{\eta}\]
for any $n_0\in \Bo{Z}^2$, $\mu \ge 0$, and any ball $B_R\subset \R^2$ of radius $R>1$.
\end{lem}

\begin{proof}
Although this bound is well-known, we give an outline of proof.
When $\mu ^{1/2}\lec R^3$, we recall the estimate on the number of integer points on a circle.
When $\mu ^{1/2}\gg R^3$, the estimate follows from the fact that there are at most two integer points on a (connected) arc of radius $R$ and length $r$ if $R\gg r^3$ (see, e.g., Lemma~4.4 in \cite{DPST07}).
\end{proof}

The estimate \eqref{est:cNLS} will be deduced from the following estimates for localized functions:
\begin{lem}\label{lem:trilinear-}
Let $\e >0$.
Then, there exists $C>0$ such that we have
\[ \norm{P_N\!\!\sum _{\mat{n_1,n_2,n_3\in \Bo{Z}^2\\ n=n_1-n_2+n_3,\, \Phi =\mu}}\prod _{j=1}^3P_{N_j}\om ^{(j)}_{n_l}}{\ell ^2}\le C (N_{\med}N_{\min})^{\e}\prod _{j=1}^3\norm{P_{N_j}\om ^{(j)}}{\ell ^2}\]
for any dyadic $N, N_j\ge 1$ ($1\le j\le 3$) and $\mu \in \Bo{Z}$, where $N_{\max},N_{\med},N_{\min}$ denote the maximum, median, minimum of $\shugo{N_1,N_2,N_3}$, respectively, and $(P_N\om )_n:=\chi _{N\le \LR{n}<2N}\om _n$.
\end{lem}

\begin{proof}
Fix $\mu \in \Bo{Z}$.
We write ``$(*)$'' to denote the condition
\[ n=n_1-n_2+n_3,\qquad \Phi =\mu ,\qquad N_j\le \LR{n_j}<2N_j\quad (1\le j\le 3).\]
Note that
\begin{align}
\Phi =\mu \quad \Longleftrightarrow \quad &\big| 2n_1-(n+n_2)\big| ^2=\big| 2n_3-(n+n_2)\big| ^2=| n-n_2| ^2-2\mu \label{A}\\
\Longleftrightarrow \quad &\big| 2n_2-(n_1+n_3)\big| ^2=| n_1-n_3| ^2+2\mu \label{B}
\end{align}
under the condition $n=n_1-n_2+n_3$.

Applying the Cauchy-Schwarz inequality several times, we have
\begin{align*}
\Big( \sum _{n}& \Big| \sum _{\mat{n_1,n_2,n_3\\ (*)}}\prod _{j=1}^3\om ^{(j)}_{n_j}\Big|^2\Big) ^{1/2}\\
&\le \Big( \sum _{n}\Big( \sum _{n_2}|P_{N_2}\om ^{(2)}_{n_2}|^2\Big) \Big( \sum _{n_2}\Big| \sum _{\mat{n_1,n_3\\ (*)}}\om ^{(1)}_{n_1}\om ^{(3)}_{n_3}\Big| ^2\Big) \Big) ^{1/2}\\
&\le \norm{P_{N_2}\om ^{(2)}}{\ell ^2}\Big( \sum _{n,n_2}A_{\mu}(n,n_2)\sum _{\mat{n_1,n_3\\ (*)}}|\om ^{(1)}_{n_1}\om ^{(3)}_{n_3}|^2\Big) ^{1/2},
\intertext{where $A_{\mu}(n,n_2):=\# \Shugo{(n_1,n_3)\in (\Bo{Z}^2)^2}{(*)}$, and}
&\le \norm{P_{N_2}\om ^{(2)}}{\ell ^2}\sup _{n,n_2}A_\mu (n,n_2)^{1/2}\Big( \sum _{n_1,n_3}|P_{N_1}\om ^{(1)}_{n_1}P_{N_3}\om ^{(3)}_{n_3}|^2B_\mu (n_1,n_3)\Big) ^{1/2},
\intertext{where $B_{\mu}(n_1,n_3):=\# \Shugo{(n,n_2)\in (\Bo{Z}^2)^2}{(*)}$, and}
&\le \sup _{n,n_2}A_\mu (n,n_2)^{1/2}\sup _{n_1,n_3}B_\mu (n_1,n_3)^{1/2}\prod _{j=1}^3\norm{P_{N_j}\om ^{(j)}}{\ell ^2}.
\end{align*}
It then suffices to show 
\eq{claim1}{\sup _{n,n_2}A_\mu (n,n_2)\cdot \sup _{n_1,n_3}B_\mu (n_1,n_3)\lec (N_{\med}N_{\min})^{2\e}.}

{\bf Case 1: $N_2\lec N_{\med}$}.
Lemma~\ref{lem:numbertheory} (with $\eta =2\e$) and \eqref{A} imply that
\[ A_\mu (n,n_2)\lec \begin{cases} N_{\min}^{2\e} &\text{if $N_2\ge N_{\med}$,}\\ N_{\med}^{2\e} &\text{if $N_2=N_{\min}$,}\end{cases}\]
uniformly in $n,n_2$, while using Lemma~\ref{lem:numbertheory} and \eqref{B} we have
\[ B_\mu (n_1,n_3)\lec \begin{cases} N_{\med}^{2\e} &\text{if $N_2\ge N_{\med}$,}\\ N_{\min}^{2\e} &\text{if $N_2=N_{\min}$,}\end{cases}\]
uniformly in $n_1,n_3$.
This implies \eqref{claim1}.

{\bf Case 2: $N_2=N_{\max}\gg N_{\med}$}.
This time we always have $A_\mu (n,n_2) \lec N_{\min}^{2\e}$ by Lemma~\ref{lem:numbertheory} and \eqref{A}.
However, \eqref{B} yields only $B_\mu (n_1,n_3) \lec N_{\max}^{2\e}$.
Here, we exploit the almost orthogonality and restrict $n,n_2$ onto cubes of side length $\sim N_{\med}$ at the beginning of the estimate.
Then, we can obtain the bound $N_{\med}^{2\e}$ for $B_\mu (n_1,n_3)$, which implies \eqref{claim1}.
\end{proof}

\begin{proof}[Proof of \eqref{est:cNLS} for $s_1>0$]
Applying Lemma~\ref{lem:trilinear-}, the left-hand side of \eqref{est:cNLS} is bounded by
\eqq{
&\sup _\mu \bigg[\sum _{N\ge 1}\bigg( \sum _{\mat{N_1,N_2,N_3\ge 1\\ N_{\max}\gec N}}\norm{P_N\!\!\sum _{\mat{n_1,n_2,n_3\in \Bo{Z}^2\\ n=n_1-n_2+n_3,\, \Phi =\mu}}\prod _{j=1}^3P_{N_j}\om ^{(j)}_{n_l}}{\ell ^2}\bigg) ^2\bigg] ^{1/2}\\
&\quad \lec \Big[ \sum _{N\ge 1}\Big( \sum _{\mat{N_1,N_2,N_3\ge 1\\ N_{\max}\gec N}}(N_{\med}N_{\min})^{\e}\prod _{j=1}^3\norm{P_{N_j}\om ^{(j)}}{\ell ^2}\Big) ^2\Big] ^{1/2}.
}
If $N_{\max}\sim N$, we take $\e =s_1/2>0$ and apply the Cauchy-Schwarz inequality in $N_{\min}$ and $N_{\med}$ to show \eqref{est:cNLS}.
If $N_{\max}\gg N$, then we have $N_{\max}\sim N_{\med}$ and thus $(N_{\med}N_{\min})^{\e}\lec N^{-\e }N_{\max}^{-\e}N_{\med}^{3\e}N_{\min}^\e$, which enables us to undo the dyadic decompositions and obtain \eqref{est:cNLS} for $s_1=4\e$.
\end{proof}

%%%%%%%%%%%%%%%%%%%%%%%%%%%%%%%%%%%%%%%%%%%%%%
% cFNLS
%%%%%%%%%%%%%%%%%%%%%%%%%%%%%%%%%%%%%%%%%%%%%%

\bigskip
\section{Application to the cubic fractional NLS} \label{section:FNLS}

In this section, we consider the cubic nonlinear Schr\"odinger equation with fractional Laplacian on $\T$:
\eq{FNLS}
{\left\{ \begin{alignedat}{2}
~i\p _tu+(-\p _x^2)^\al u&=\pm |u|^2u,&\qquad &(t,x)\in [0,T]\times \T ,\\
u(0,x)&=u_0(x), & &x\in \T .
\end{alignedat}
\right.
}

Well-posedness of the above Cauchy problem in Sobolev spaces $H^s(\T )$ was addressed by Cho \mbox{et al.}~\cite{CHKL15} (see also~\cite{DET16}).
They showed that if $\frac{1}{2}<\al <1$, then \eqref{FNLS} is locally well-posed in $H^s$ for $s\ge \frac{1-\al}{2}$ and ill-posed in Sobolev spaces of smaller indices in the sense that the data-to solution map fails to be locally uniformly continuous.
Note that the local solution given in \cite{CHKL15} was constructed via the iteration with the Fourier restriction spaces, and UU has been open.
(However, this is trivial if $s>\frac{1}{2}$.)
In order to discuss UU, the solution should belong to $L^3_{\mathrm{loc}}$ in space so that the cubic nonlinear term can make sense.
In view of the embedding $H^{1/6}\hookrightarrow L^3$, we are naturally led to restrict the regularity to $s\ge \frac{1}{6}$.

We shall prove the following almost optimal result:
\begin{thm}\label{thm:uu}
Let $\frac{1}{2}<\al <1$.
Assume $s\ge \frac{1}{6}$ and $s>\frac{1-\al}{2}$.
Then, for any $T>0$ there is at most one solution (in the sense of distribution) to \eqref{FNLS} in $C_TH^s(\T )$.
\end{thm}

\subsection{Reduction to the fundamental trilinear estimates}

First of all, we introduce an equivalent problem in the Fourier side.
Let $u(t)\in C_TH^s(\T )$ be a solution of \eqref{FNLS}, and define $\om (t)\in C_T\ell ^2_s$ as
\[ \om _n(t):=\big[ \F e^{-it(-\p _x^2)^\al}u(t,\cdot )\big] _n=\frac{1}{2\pi}e^{-it|n|^{2\al}}\int _{\T}u(t,x)e^{-inx}\,dx,\quad (t,n)\in [0,T]\times \Bo{Z}.\]
Then, the Cauchy problem \eqref{FNLS} is transformed to
\begin{gather}
\p _t\om _n=\mp i\sum _{\mat{n_1,n_2,n_3\in \Bo{Z}\\ n_1-n_2+n_3=n}}e^{-it\Phi}\om _{n_1}\bbar{\om _{n_2}}\om _{n_3},\qquad (t,n)\in [0,T]\times \Bo{Z} ,\label{eq:om}\\
\Phi =\Phi (n,n_1,n_2,n_3):=|n|^{2\al}-|n_1|^{2\al}+|n_2|^{2\al}-|n_3|^{2\al},\label{def:Phi}
\end{gather}
with the initial condition
\eq{IC:om}{\om _n(0)=\big[ \F u_0\big] _n,\qquad n\in \Bo{Z} .}
We divide the nonlinear part of \eqref{eq:om} into two parts,
\[ \Sc{N}^{(1)}[\om ]_n+\Sc{R} [\om ]_n:=\mp i\bigg( \sum _{(n_1,n_2,n_3)\in \Gamma _n^N}+\sum _{(n_1,n_2,n_3)\in \Gamma _n^R}\bigg) e^{-it\Phi}\om _{n_1}\bbar{\om _{n_2}}\om _{n_3},\]
where
\eqq{
\Gamma _n&:=\Shugo{(n_1,n_2,n_3)\in \Bo{Z}^3}{n_1-n_2+n_3=n},\\
\Gamma _n^N&:=\Shugo{(n_1,n_2,n_3)\in \Gamma _n}{|n_2-n_1|>|n_2|^{1-\al}~~\text{and}~~|n_2-n_3|>|n_2|^{1-\al}},\\
\Gamma _n^R&:=\Shugo{(n_1,n_2,n_3)\in \Gamma _n}{|n_2-n_1|\le |n_2|^{1-\al}~~\text{or}~~|n_2-n_3|\le |n_2|^{1-\al}}.}
This particular decomposition is important to prove an almost optimal result.

By Theorem~\ref{thm:abstract} [B] with $X=\ell ^\I$, Theorem~\ref{thm:uu} is obtained if the following trilinear estimates are established.
Similarly as in Section~\ref{section:cNLS}, we consider the system for $(\om _n,\bbar{\om_{n}})$ and (exploiting symmetry) prove these trilinear estimates only on the first component.
\begin{prop}\label{prop:trilinear0}
Let $\frac{1}{2}<\al <1$ and $s\ge \frac{1}{6}$ be such that $s>\frac{1-\al}{2}$.
Then, there exist $s_1<s$ and $s_2>s$ such that the following estimates hold.
\begin{align}
\sup _{\mu \in \Bo{Z}}\norm{\sum _{\mat{(n_1,n_2,n_3)\in \Gamma _n^N\\ |\Phi -\mu |<1}}\om ^{(1)}_{n_1}\om ^{(2)}_{n_2}\om ^{(3)}_{n_3}}{\ell ^2_{s_1}}&\lec \prod _{j=1}^3\norm{\om ^{(j)}}{\ell ^2_{s_1}},\label{est:B1}\\
\norm{\sum _{(n_1,n_2,n_3)\in \Gamma _n^N}\om ^{(1)}_{n_1}\om ^{(2)}_{n_2}\om ^{(3)}_{n_3}}{\ell ^2_{s_2}}&\lec \prod _{j=1}^3\norm{\om ^{(j)}}{\ell ^2_{s_2}},\label{est:B1'}\\
\sup _{\mu \in \Bo{Z}}\norm{\sum _{\mat{(n_1,n_2,n_3)\in \Gamma _n^N\\ |\Phi -\mu |<1}}\om ^{(1)}_{n_1}\om ^{(2)}_{n_2}\om ^{(3)}_{n_3}}{\ell ^\I}&\lec \min _{1\le j\le 3}\norm{\om ^{(j)}}{\ell ^\I}\prod _{\mat{l=1\\ l\neq j}}^3\norm{\om ^{(l)}}{\ell ^2_{s_1}},\label{est:B2}\\
\norm{\sum _{(n_1,n_2,n_3)\in \Gamma _n^N}\om ^{(1)}_{n_1}\om ^{(2)}_{n_2}\om ^{(3)}_{n_3}}{\ell ^\I}&\lec \min _{1\le j\le 3}\norm{\om ^{(j)}}{\ell ^\I}\prod _{\mat{l=1\\ l\neq j}}^3\norm{\om ^{(l)}}{\ell ^2_{s_2}},\label{est:B2'}\\
\norm{\sum _{(n_1,n_2,n_3)\in \Gamma _n^N}\om ^{(1)}_{n_1}\om ^{(2)}_{n_2}\om ^{(3)}_{n_3}}{\ell ^\I}&\lec \prod _{j=1}^3\norm{\om ^{(j)}}{\ell ^2_{s}},\label{est:B3}\\
\norm{\sum _{(n_1,n_2,n_3)\in \Gamma _n^R}\om ^{(1)}_{n_1}\om ^{(2)}_{n_2}\om ^{(3)}_{n_3}}{\ell ^2_s}&\lec \prod _{j=1}^3\norm{\om ^{(j)}}{\ell ^2_s}.\label{est:R}
\end{align}
\end{prop}

\subsection{Proof of the trilinear estimates}

The key lemma to prove Proposition~\ref{prop:trilinear0} is the following estimates on the number of frequencies.
\begin{lem}\label{lem:number}
Let $K\ge 1$ and $\mu _*,k_*\in \Bo{Z}$.
Then, we have
\eq{est:number+}{\# \Shugo{(k,l)\in \Bo{Z}^2}{k+l=k_*,\,|k|^{2\al}+|l|^{2\al}=\mu _*+O(1),\,|k|\le |l|,\,|k|\le K}\lec K^{1-\al}.}
If we further assume that $|k_*|\gec K^{1-\al}$, then we have
\eq{est:number-}{\# \Shugo{(k,l)\in \Bo{Z}^2}{k-l=k_*,\,|k|^{2\al}-|l|^{2\al}=\mu _*+O(1),\,|k|\le |l|,\,|k|\le K}\lec K^{1-\al}.}
Here, the implicit constants are independent of $K$, $\mu _*$, and $k_*$.
\end{lem}

\begin{proof}
Define $f_\pm (x):=|x|^{2\al}\pm |x-k_*|^{2\al}$ for $x\in \R$ and 
\[ I_\pm :=\Shugo{x}{|f_{\pm}(x)|=\mu _*+O(1),\;|x|\le |x-k_*|,\;|x|\le K}.\]
It then suffices to show that $|I_+|\lec K^{1-\al}$, and $|I_-|\lec K^{1-\al}$ provided $|k_*|\gec K^{1-\al}$.

Since $\al >\frac{1}{2}$, $f\in C^1(\R )$ and 
\[ f'_\pm (x)=2\al \big( |x|^{2\al -2}x\pm |x-k_*|^{2\al -2}(x-k_*)\big) .\]
If $|x|\le \frac{1}{2}|x-k_*|$ or $\pm x(x-k_*)\ge 0$, then $|f'_\pm (x)|\gec |x|^{2\al -1}$.
This implies that
\[ \Big| I_\pm \cap \Shugo{x}{|x|\le 1~~\text{or}~~|x|\le \tfrac{1}{2}|x-k_*|~~\text{or}~~\pm x(x-k_*)\ge 0}\Big| \lec 1.\]
(Note that the set in the left-hand side consists of finite number of intervals.)

From now on, we assume
\[ x\in J_\pm :=\Shugo{x}{|x|>1,\;\tfrac{1}{2}|x-k_*|<|x|\le |x-k_*|,\;|x|\le K,\;\pm x(x-k_*)<0}.\]
$J_\pm$ also consists of finite number of intervals.
Since $x$ and $\mp (x-k_*)$ has the same sign on $J_\pm$, by the mean value theorem, we see that
\[ |f'_\pm (x)|\sim |x|^{2\al -2}|x\pm (x-k_*)|\ge K^{2\al -2}|x\pm (x-k_*)|\quad \text{on $J_\pm$}.\]
Hence, for the $+$ case we have
\[ |f'_+(x)|\gec K^{\al -1}\quad \text{on $J_+\cap \Shugo{x}{|2x-k_*|\ge K^{1-\al}}$}\]
and 
\[ \Big| \Shugo{x}{|2x-k_*|\le K^{1-\al}}\Big| =K^{1-\al},\]
while for the $-$ case 
\[ |f'_-(x)|\gec K^{\al -1}\quad \text{on $J_-$}\]
under the assumption that $|k_*|\gec K^{1-\al}$.
Therefore, in both cases, we deduce that $|I_\pm \cap J_\pm |\lec K^{1-\al}$, which concludes the proof.
\end{proof}

\begin{proof}[Proof of Proposition~\ref{prop:trilinear0}]
It is easy to see that the estimates \eqref{est:B1'} and \eqref{est:B2'} hold for any $s_2>\frac{1}{2}$ by the embedding $\ell ^2_{s_2}\hookrightarrow \ell ^1$.
The estimate \eqref{est:B3} is also easily verified whenever $s\ge \frac{1}{6}$ by the Sobolev embedding $H^{s}(\T )\hookrightarrow L^3(\T )$.
(For these estimates, restriction onto frequencies in $\Gamma_n^N$ is not necessary.)

Before proving the remaining trilinear estimates, we introduce some notation.
The operator $P_N$ for a dyadic number $N\ge 1$ is the same as in Section~\ref{section:cNLS}; i.e., $[P_N\om ]_n:=\chi _{N\le \LR{n}<2N}\om _n$. 
Given quadruplets of dyadic numbers $\{ N_0,\dots ,N_3\}$, we write $N_{(j)}$ ($0\le j\le 3$) to denote the $j$-th largest one among them. 

Let us prove \eqref{est:B1}.
By an easy argument with the Cauchy-Schwarz inequality and duality, it suffices to show that
\eqq{I&:=\Big( \frac{N_0}{N_1N_2N_3}\Big) ^{s_1}\bigg| \sum _{n\in \Bo{Z}}\sum _{\mat{(n_1,n_2,n_3)\in \Gamma _n^N\\ \Phi =\mu +O(1)}}P_{N_1}\om ^{(1)}_{n_1}P_{N_2}\om ^{(2)}_{n_2}P_{N_3}\om ^{(3)}_{n_3}P_{N_0}\om ^{(0)}_n\bigg| \\
&\lec \chi _{N_{(1)}\sim N_{(2)}}N_{(3)}^{-\de}\prod _{j=0}^3\norm{P_{N_j}\om ^{(j)}}{\ell ^2}}
for some $\de >0$.
Note that there is no contribution if $N_{(1)}\gg N_{(2)}$.
If $N_2\lec N_{(3)}$, which implies $N_0\lec \max \shugo{N_1,N_3}$, the Cauchy-Schwarz inequality followed by two applications of \eqref{est:number+} implies that 
\eqq{I&\lec \Big( \frac{1}{N_2\min \{ N_1,N_3\}}\Big) ^{s_1}\bigg[ \sum _{n,n_2\in \Bo{Z}}\big| P_{N_0}\om ^{(0)}_nP_{N_2}\om ^{(2)}_{n_2}\big| ^2\sum _{\mat{n_1,n_3\in \Bo{Z}^2\\ (n_j)_{j=1}^3\in \Gamma _n^N,\,\Phi =\mu +O(1)}}1\bigg] ^{1/2}\\
&\hspace{100pt}\times \bigg[ \sum _{n_1,n_3\in \Bo{Z}}\big| P_{N_1}\om ^{(1)}_{n_1}P_{N_3}\om ^{(3)}_{n_3}\big| ^2\sum _{\mat{n,n_2\in \Bo{Z}^2\\ (n_j)_{j=1}^3\in \Gamma _n^N,\,\Phi =\mu +O(1)}}1\bigg] ^{1/2}\\
&\lec \Big( \frac{1}{N_2\min \{ N_1,N_3\}}\Big) ^{s_1-\frac{1-\al}{2}}\prod _{j=0}^3\norm{\om ^{(j)}}{\ell ^2}.}
This implies the desired estimate if $s_1>\frac{1-\al}{2}$.
If $N_2\gg N_{(3)}$, then $N_0\lec N_2$ and we have
\eqq{I&\lec \Big( \frac{1}{N_1N_3}\Big) ^{s_1}\bigg[ \sum _{n,n_1\in \Bo{Z}}\big| P_{N_0}\om ^{(0)}_nP_{N_1}\om ^{(1)}_{n_1}\big| ^2\sum _{\mat{n_2,n_3\in \Bo{Z}^2\\ (n_j)_{j=1}^3\in \Gamma _n^N,\,\Phi =\mu +O(1)}}1\bigg] ^{1/2}\\
&\hspace{100pt}\times \bigg[ \sum _{n_2,n_3\in \Bo{Z}}\big| P_{N_2}\om ^{(2)}_{n_2}P_{N_3}\om ^{(3)}_{n_3}\big| ^2\sum _{\mat{n,n_1\in \Bo{Z}^2\\ (n_j)_{j=1}^3\in \Gamma _n^N,\,\Phi =\mu +O(1)}}1\bigg] ^{1/2}\\
&\lec \Big( \frac{1}{N_1N_3}\Big) ^{s_1-\frac{1-\al}{2}}\prod _{j=0}^3\norm{\om ^{(j)}}{\ell ^2},}
where we have used \eqref{est:number-} twice.
This is again a proper bound if $s_1>\frac{1-\al}{2}$.

Proof of \eqref{est:B2} proceeds in a similar manner.
Let us show, for instance, the estimate
\[ \norm{\sum _{\mat{(n_1,n_2,n_3)\in \Gamma _n^N\\ |\Phi -\mu |<1}}\om ^{(1)}_{n_1}\om ^{(2)}_{n_2}\om ^{(3)}_{n_3}}{\ell ^\I}\lec \norm{\om ^{(1)}}{\ell ^\I}\norm{\om ^{(2)}}{\ell ^2_{s_1}}\norm{\om ^{(3)}}{\ell ^2_{s_1}},\]
which is reduced to showing, for fixed $n,\mu \in \Bo{Z}$, that
\eqq{II&:=\Big( \frac{1}{N_2N_3}\Big) ^{s_1}\bigg| \sum _{\mat{(n_1,n_2,n_3)\in \Gamma _n^N\\ \Phi =\mu +O(1)}}\om ^{(1)}_{n_1}P_{N_2}\om ^{(2)}_{n_2}P_{N_3}\om ^{(3)}_{n_3}\bigg| \\
&\lec (N_2N_3)^{-\de}\norm{\om ^{(1)}}{\ell ^\I}\norm{P_{N_2}\om ^{(2)}}{\ell ^2}\norm{P_{N_3}\om ^{(3)}}{\ell ^2}}
for some $\de >0$.
By the Cauchy-Schwarz, \eqref{est:number+} and \eqref{est:number-}, we have
\eqq{II&\lec \Big( \frac{1}{N_2N_3}\Big) ^{s_1}\norm{\om ^{(1)}}{\ell ^\I}\bigg[ \sum _{n_2\in \Bo{Z}}\big| P_{N_2}\om ^{(2)}_{n_2}\big| ^2\sum _{\mat{n_1,n_3\in \Bo{Z}\\(n_j)_{j=1}^3\in \Gamma _n^N,\,\Phi =\mu +O(1)}}1\bigg] ^{1/2}\\
&\hspace{100pt}\times \bigg[ \sum _{n_3\in \Bo{Z}}\big| P_{N_3}\om ^{(3)}_{n_3}\big| ^2\sum _{\mat{n_1,n_2\in \Bo{Z}\\(n_j)_{j=1}^3\in \Gamma _n^N,\,\Phi =\mu +O(1)}}1\bigg] ^{1/2}\\
&\lec \Big( \frac{1}{N_2N_3}\Big) ^{s_1-\frac{1-\al}{2}}\norm{\om ^{(1)}}{\ell ^\I}\norm{P_{N_2}\om ^{(2)}}{\ell ^2}\norm{P_{N_3}\om ^{(3)}}{\ell ^2},}
which yields the desired estimate if $s_1>\frac{1-\al}{2}$.

Finally, \eqref{est:R} is shown once the following block estimate is verified:
\[ III:=\Big( \frac{N_0}{N_1N_2N_3}\Big) ^s\bigg| \sum _{\mat{n_0,\dots ,n_3\in \Bo{Z}\\ n_1-n_2=n_0-n_3\\|n_1-n_2|\le |n_2|^{1-\al}}}P_{N_0}\om ^{(0)}_{n_0}\cdots P_{N_3}\om ^{(3)}_{n_3}\bigg| \lec \chi _{N_{(1)}\sim N_{(2)}}N_{(3)}^{-\de}\prod _{j=0}^3\norm{P_{N_j}\om ^{(j)}}{\ell ^2}\]
for some $\de >0$.
Note that $|n_1-n_2|\le |n_2|^{1-\al}$ implies $\LR{n_1}\sim \LR{n_2}$.
If $N_0\gg N_3$, namely $|n_0|\gg |n_3|$, then $|n_0|\sim |n_0-n_3|=|n_1-n_2|\le |n_2|^{1-\al}$, which allows us to assume $N_3\ll N_0\lec N_1^{1-\al}\sim N_2^{1-\al}$.
By the Young inequality, it follows that
\eqq{III&\le \Big( \frac{N_0}{N_1N_2N_3}\Big) ^s\norm{P_{N_1}\om ^{(1)}}{\ell ^2}\norm{P_{N_2}\om ^{(2)}}{\ell ^2}\norm{P_{N_3}\om ^{(3)}}{\ell ^1}\norm{P_{N_0}\om ^{(0)}}{\ell ^1}\\
&\lec \frac{N_0^{\frac{1}{2}+s}N_3^{\frac{1}{2}-s}}{N_1^{2s}}\prod _{j=0}^3\norm{P_{N_j}\om ^{(j)}}{\ell ^2}\\
&\lec 
\begin{cases}
N_1^{1-\al -2s}\prod _{j=0}^3\norm{P_{N_j}\om ^{(j)}}{\ell ^2},&\quad \text{if $0<s\le \frac{1}{2}$},\\
N_1^{\frac{1}{2}-s}\prod _{j=0}^3\norm{P_{N_j}\om ^{(j)}}{\ell ^2},&\quad \text{if $s>\frac{1}{2}$},
\end{cases}}
which is proper if $s>\frac{1-\al}{2}$.
Therefore, let us focus on the case $N_0\lec N_3$.
Then, applying the Cauchy-Schwarz inequality,
\eqq{III&\lec \frac{1}{N_1^{2s}}\norm{P_{N_3}\om ^{(3)}}{\ell ^2}\norm{P_{N_0}\om ^{(0)}}{\ell ^2}\sum _{\mat{n_1,n_2\in \Bo{Z}\\ |n_1-n_2|\lec N_1^{1-\al}\sim N_2^{1-\al}}}\big| P_{N_1}\om ^{(1)}_{n_1}P_{N_2}\om ^{(2)}_{n_2}\big| \\
&\lec \frac{1}{N_1^{2s}}\norm{P_{N_3}\om ^{(3)}}{\ell ^2}\norm{P_{N_0}\om ^{(0)}}{\ell ^2}N_1^{\frac{1-\al}{2}}\norm{P_{N_1}\om ^{(1)}}{\ell ^2}N_1^{\frac{1-\al}{2}}\norm{P_{N_2}\om ^{(2)}}{\ell ^2},}
which is again proper if $s>\frac{1-\al}{2}$.
The proof is completed.
\end{proof}

%%%%%%%%%%%%%%%%%%%%%%%%%%%%%%%%%%%%%%%%%%%%%%
% DNLS
%%%%%%%%%%%%%%%%%%%%%%%%%%%%%%%%%%%%%%%%%%%%%%

%\bigskip
\section{Application to the cubic derivative NLS} \label{section:DNLS}

In this section, we consider the one-dimensional cubic derivative NLS (DNLS):
\eq{DNLS}{
\left\{ 
\begin{alignedat}{2}
\p _tu&=i\p _x^2u\pm \p _x(|u|^2u),&\qquad &(t,x)\in [0,T]\times \T ,\\
u(0,x)&=u_0(x), & &x\in \T .
\end{alignedat}\right.
}

The Cauchy problem \eqref{DNLS} was shown to be locally well-posed in $H^{1/2}$ in the non-periodic case by Takaoka \cite{T99} and in the periodic case by Herr \cite{H06}.
Both of them used the Fourier restriction norm method to prove well-posedness for an equivalent Cauchy problem obtained via a gauge transform which converts the derivative nonlinearity $\p _x(|u|^2u)$ into milder one $u^2\p _x\bbar{u}$.
In the non-periodic case it was also shown by Biagioni and Linares \cite{BL01} that the regularity $s\ge 1/2$ is sharp in the sense that the flow map loses the uniform continuity in the $H^s$ topology for $s<1/2$, while the critical Sobolev space with respect to scaling is $L^2$.

UU of solutions to \eqref{DNLS} was proved by Win \cite{S08} in the energy space $H^1$ for the non-periodic case.
Note that any distributional solution of the (gauge-equivalent) DNLS in $C_TH^1$ lies in Bourgain space $X^{1/2,1/2}$, because $u\in C_TH^1\hookrightarrow L^2_TH^1=X^{1,0}_T$ and $u\approx \LR{i\p _t-\p _x^2}^{-1}N(u)\in X^{0,1}_T$ whenever the nonlinear terms $N(u)$ belong to $L^2_TL^2$.
Then, the problem is reduced to showing uniqueness of solutions to the gauge-equivalent DNLS in $X^{1/2,1/2}$, which was proved in \cite{S08} by slightly modifying the multilinear estimates obtained in \cite{T99}.
The same strategy can be applied to the periodic problem, since we already have enough multilinear estimates; a slight modification of Corollary~4.6 in \cite{H06} is sufficient for the fixed point argument in $X^{1/2,1/2}$.

Here, we prove UU of solutions to DNLS on $\mathbb{T}$ in weaker spaces than $H^1$ via the abstract theory.
Our result reads as follows:
\begin{thm}\label{thm:DNLS}
For the Cauchy problem \eqref{DNLS}, unconditional uniqueness of solutions holds in $C_TH^s$ for $s>1/2$.
\end{thm}

Our result is optimal in the sense that the derivative term $u^2\p _x\bbar{u}$ in the gauge-equivalent DNLS does not make sense for $u\in H^s$ in the framework of distribution if $s\le 1/2$.
However, the original DNLS makes sense if $u\in L^3\hookrightarrow H^{1/6}$.
Unconditional uniqueness for $1/6\le s\le 1/2$, as well as (conditional) well-posedness for $0\le s<1/2$, is a challenging open problem.%
\footnote{Recently, the non-periodic problem was shown to be globally well-posed in $H^{1/6}(\mathbb{R})$ using the integrable structure of the equation. 
For the recent progress in the well-posedness theory of \eqref{DNLS} based on its complete integrability, we refer to \cite{HGKVp} and references therein.}

We apply the abstract framework after transforming the equation into an equivalent but milder one by a  gauge transform, as in the previous results mentioned above.
Recently, Mosincat and Yoon \cite{MYp} followed our ideas to establish UU for \eqref{DNLS} in the same regularity range in the non-periodic setting.
A similar approach can be used to show UU for the modified Benjamin-Ono equation; see \cite{K-mBOp}.

\subsection{Gauge transform}

In this section, we (continue to) use the following definition of the Fourier coefficients of a function on $\T =\R /2\pi \Bo{Z}$:
\[ (\F f)_n:=\frac{1}{2\pi} \int _{\T}f(x)e^{-inx}\,dx,\qquad n\in \Bo{Z} ,\]
and use in addition the following operators:
\eqs{P_{\leq N}f:=\F ^{-1}\chf{|n|\leq N}\F f,\quad P_{>N}f:=f-P_{\leq N}f\qquad (\text{for $N>0$}),\\
P_cf:=(\F f)_0=\frac{1}{2\pi}\int _{\T}f(x)\,dx,\quad (P_{\neq c}f)(x):=f(x)-P_cf =\sum _{n\in \Bo{Z}\setminus \{ 0\}}(\F f)_ne^{inx},\\
(\p _x^{-1}f)(x):=\F ^{-1}\Big[\frac{(\F f)_n}{in}\Big] (x)=\frac{1}{2\pi}\int _{\T} \int _\th ^xf(y)\,dy\,d\th \qquad (\text{for $f$ s.t. $P_cf=0$}).
}
Note that $\p _x^{-1}f$ is also $2\pi$-periodic.

We focus on the case of the $+$ sign in \eqref{DNLS}, while the sign is not important in our argument.
Let $u\in C_TH^s$ ($s\ge 0$) be a solution of DNLS \eqref{DNLS} in the sense of distribution (the nonlinear term $\p _x(|u|^2u)$ makes sense if $s\ge 1/6$).
For $N>0$, the function $P_{\le N}u(t,x)$ belongs to $C^1_TH^\I$;
In fact, we have
\[ \p _tP_{\le N}u=iP_{\le N}\p _x^2u+P_{\le N}\p _x(|u|^2u)~\in C_TH^\I ,\]
and this equality is interpreted in the classical sense.

The gauge transform for the periodic DNLS was introduced in \cite{H06}: 
\[ u(t)\quad \mapsto \quad v(t):=e^{-iJ(u(t))}u(t),\qquad J(f):=\p _x^{-1}P_{\neq c}(|f|^2).\]
If $u$ has sufficient smoothness, a formal calculation shows that $v$ satisfies the equation
\eq{GDNLS'}{\p _tv=i\p _x^2v-\Big( v^2\p _x\bbar{v}-2P_c\big( v\p _x\bbar{v}\big) v\Big) +2\mu \p _xv+\frac{i}{2}|v|^4v-i\mu |v|^2v+i\Big( \mu ^2-\frac{1}{2}P_c(|v|^4)\Big) v,}
where $\mu (t):=P_c(|v(t)|^2)=P_c(|u(t)|^2)$.
Set
\[ w(t,x):=(\tau _{2\mu}v)(t,x):=v\Big( t,\,x-2\int _0^t\mu (t')\,dt'\Big) \]
to eliminate the linear term $2\mu \p _xv$, then $w$ satisfies
\eq{GDNLS}{\p _tw=i\p _x^2w-\Big( w^2\p _x\bbar{w}-2P_c\big( w\p _x\bbar{w}\big) w\Big) +\frac{i}{2}|w|^4w-i\mu |w|^2w+i\Big( \mu ^2-\frac{1}{2}P_c(|w|^4)\Big) w.}
Note that $\mu (t)=P_c(|w(t)|^2)$, and hence we have a closed equation for $w$ (in contrast to the case of the modified Benjamin-Ono equation~\cite{K-mBOp}, where one has a system of equations for the original unknown function and its gauge transform).

The equation \eqref{GDNLS} is the one for which in the next subsection we will apply the abstract theory and prove UU in $H^s$, $s>1/2$.
Let us see here that this actually implies UU for \eqref{DNLS}, which is ensured by the next two lemmas:
\begin{lem}[{\cite[Lemma~2.3]{H06}}]
For $s\ge 0$ the map $u\mapsto \tau _{2\mu}(e^{-iJ(u)}u)$ is a homeomorphism on $C_TH^s$.
\end{lem}
\begin{rem}
This was proved in \cite{H06} assuming the $L^2$ conservation (i.e., $P_c(|u(t)|^2)=P_c(|u(0)|^2)$ for any $t$), but the same argument works; we only need the fact that $u_k\to u$ in $C_TH^s$ ($s\geq 0$) implies $\int _0^tP_c(|u_k(t')|^2)\,dt'\to \int _0^tP_c(|u(t')|^2)\,dt'$ uniformly in $t\in [0,T]$.
Conservation of the $L^2$ norm for general solutions (in the sense of distribution) to \eqref{DNLS} is a non-trivial problem, and we can show that at least for solutions in $C_TH^s$ with $s>1/2$; see the argument in \cite[Lemma~2.5]{KT18}, for instance.
However, we will not rely on the $L^2$ conservation law in the proof of Theorem~\ref{thm:DNLS}.
\end{rem}

\begin{lem}\label{lem:3}
Let $s>1/2$, and let $u\in C_TH^s(\T )$ satisfy the equation \eqref{DNLS} in the sense of distribution.
Then, $w=\tau _{2\mu}(e^{-iJ(u)}u)$ satisfies the equation \eqref{GDNLS} in the sense of distribution.
\end{lem}

\begin{proof}
Since we have
\[ \p _t(\tau _{2\mu}v)=\tau _{2\mu}(\p _tv)-2\mu \tau _{2\mu}(\p _xv)\qquad \text{in $\Sc{D}'((0,T)\times \T )$}\]
for $v\in C_TH^s$, it suffices to show that $v=e^{-iJ(u)}u$ satisfies the equation \eqref{GDNLS'} in the sense of distribution.

We first set $u_N:=P_{\le N}u$, $v_N:=e^{-iJ(u_N)}u_N$ for $N>0$ and derive an equation for $v_N$.
For the gauge part we have
\[ \p _xJ(u_N)=P_{\neq c}(|u_N|^2)=P_{\neq c}(|v_N|^2),\qquad \p _x^2J(u_N)=\p _x(|v_N|^2),\]
and thus
\eqq{\p _xv_N&=e^{-iJ(u_N)}\p _xu_N-iP_{\neq c}(|v_N|^2)v_N,\\
\p _x^2v_N&=e^{-iJ(u_N)}\p _x^2u_N-iP_{\neq c}(|v_N|^2)\big[ e^{-iJ(v_N)}\p _xu_N+\p _xv_N\big] -i\p _x(|v_N|^2)v_N\\
&=e^{-iJ(u_N)}\p _x^2u_N-2iP_{\neq c}(|v_N|^2)\p _xv_N+\big( P_{\neq c}(|v_N|^2)\big) ^2v_N-i\p _x(|v_N|^2)v_N\\
&=e^{-iJ(u_N)}\p _x^2u_N-3i|v_N|^2\p _xv_N-iv_N^2\p _x\bbar{v_N}+2i\mu _N\p _xv_N+\big( P_{\neq c}(|v_N|^2)\big) ^2v_N,\\
\p _x(|v_N|^2v_N)&=e^{-iJ(u_N)}\p _x(|u_N|^2u_N)-iP_{\neq c}(|v_N|^2)|v_N|^2v_N,
}
where $\mu _N(t):=P_c(|v_N(t)|^2)=P_c(|u_N(t)|^2)$.
Similarly, we see that
\eqq{
\p _tJ(u_N)&=\p _x^{-1}P_{\neq c}2\Re \big[ \bbar{u_N}\big( i\p _x^2u_N+P_{\leq N}\p _x(|u|^2u)\big) \big] \\
&=2P_{\neq c}\Re \big[ i\bbar{u_N}\p _xu_N\big] +2\p _x^{-1}P_{\neq c}\Re \big[ \bbar{u_N}P_{\leq N}\p _x(|u|^2u)\big] .
}
Now, since $2P_c\Im [i\bbar{v_N}\p _xv_N]=P_c\p _x(|v_N|^2)=0$, we have
\eqq{2P_{\neq c}\Re \big[ i\bbar{u_N}\p _xu_N\big] &=2P_{\neq c}\Big\{ \Re \big[ i\bbar{v_N}\p _xv_N\big] -P_{\neq c}(|v_N|^2)|v_N|^2\Big\} \\
&=i\bbar{v_N}\p _xv_N-iv_N\p _x\bbar{v_N}+2iP_c(v_N\p _x\bbar{v_N})-2P_{\neq c}\big( P_{\neq c}(|v_N|^2)|v_N|^2\big) ,}
and also
\eqq{
&2\p _x^{-1}P_{\neq c}\Re \big[ \bbar{u_N}P_{\leq N}\p _x(|u|^2u)\big] \\
&=2\p _x^{-1}P_{\neq c}\Re \big[ \bbar{u_N}\big\{ P_{\leq N}\p _x(|u|^2u)-\p _x(|u_N|^2u_N)\big\} \big] +\frac{3}{2}P_{\neq c}(|v_N|^4). 
}
With these identities, we have
\begin{align}
\p _tv_N-i\p _x^2v_N
&=e^{-iJ(u_N)}\big( \p _tu_N-i\p _x^2u_N\big) \notag \\
&\quad -iv_N\Big\{ i\bbar{v_N}\p _xv_N-iv_N\p _x\bbar{v_N}+2iP_c(v_N\p _x\bbar{v_N})-2P_{\neq c}\big( P_{\neq c}(|v_N|^2)|v_N|^2\big) \notag \\
&\qquad +2\p _x^{-1}P_{\neq c}\Re \big[ \bbar{u_N}\big\{ P_{\leq N}\p _x(|u|^2u)-\p _x(|u_N|^2u_N)\big\} \big] +\frac{3}{2}P_{\neq c}(|v_N|^4)\Big\} \notag \\
&\quad -i\Big\{ -3i|v_N|^2\p _xv_N-iv_N^2\p _x\bbar{v_N}+2i\mu _N\p _xv_N+\big( P_{\neq c}(|v_N|^2)\big) ^2v_N\Big\} \notag \\
&=e^{-iJ(u_N)}\big( P_{\leq N}\p _x (|u|^2u)-\p _x(|u_N|^2u_N)\big) \label{eq:N1} \\
&\quad -2iv_N\p _x^{-1}P_{\neq c}\Re \big[ \bbar{u_N}\big\{ P_{\leq N}\p _x(|u|^2u)-\p _x(|u_N|^2u_N)\big\} \big] \label{eq:N2}\\
&\quad -v_N^2\p _x\bbar{v_N}+2P_c(v_N\p _x\bbar{v_N})v_N+2\mu _N\p _xv_N \label{eq:N3}\\
&\quad +2iP_{\neq c}\big( P_{\neq c}(|v_N|^2)|v_N|^2\big) v_N-\frac{3}{2}iP_{\neq c}(|v_N|^4)v_N+i\mu _NP_{\neq c}(|v_N|^2)v_N. \label{eq:N4}
\end{align}
Note that, since $u_N,v_N\in C^1_TH^\I$ and $P_{\leq N}\p _x(|u|^2u)\in C_TH^\I$, all the equalities above can be verified in the classical sense.

Now, we take the limit $N\to \I$.
Recall the following estimate (\cite[(A.1)]{H06}):
For all $s\ge 0$ there exists $C>0$ such that for $f,g,h\in H^s(\T )$ we have
\[ \norm{\big( e^{-iJ(f)}-e^{-iJ(g)}\big) h}{H^s}\le Ce^{C(\tnorm{f}{H^s}^2+\tnorm{g}{H^s}^2)}\big( \tnorm{f}{H^s}+\tnorm{g}{H^s}\big) \tnorm{f-g}{H^s}\norm{h}{H^s}.\]
In particular, we have (with $f=u_N$, $g=u$, and $h=1$)
\[ \norm{e^{-iJ(u_N)}-e^{-iJ(u)}}{C_TH^s}\le C(\tnorm{u}{C_TH^s})\tnorm{u_N-u}{C_TH^s}\to 0\qquad (N\to \I )\]
if $u\in C_TH^s$.
This implies that
\[ \tnorm{v_N-v}{C_TH^s}\le C(\tnorm{u}{C_TH^s})\tnorm{u_N-u}{C_TH^s}\to 0\qquad (N\to \I )\]
if $s>1/2$ and $u\in C_TH^s$.
From this and the fact that $\mu _N(t)\to \mu (t)$ uniformly in $t$, we have
\[ \p _tv_N-i\p _x^2v_N-2\mu _N\p _xv_N\quad \to \quad \p _tv-i\p _x^2v-2\mu \p _xv\qquad \text{in $\Sc{D}'((0,T)\times \T )$}.\]
The terms \eqref{eq:N4} are easily dealt with by the Sobolev inequality:
\eqq{\eqref{eq:N4}\quad \to \quad &2iP_{\neq c}\big( P_{\neq c}(|v|^2)|v|^2\big) v-\frac{3}{2}iP_{\neq c}(|v|^4)v+i\mu P_{\neq c}(|v|^2)v\\
&=\frac{i}{2}|v|^4v-i\mu |v|^2v+i\Big( \mu ^2-\frac{1}{2}P_c(|v|^4)\Big) v\qquad \text{in $C_TH^s$},}
and similarly,
\[ 2P_c(v_N\p _x\bbar{v_N})v_N\quad \to \quad 2P_c(v\p _x\bbar{v})v\qquad \text{in $C_TH^s$}.\]
For the remaining terms, we will exploit the product estimate
\eq{est:Sob-prod}{
\tnorm{fg}{H^{s-1}}\lec \tnorm{f}{H^{s}}\tnorm{g}{H^{s-1}},
}
which is valid when $s>1/2$.
This verifies
\[ -v_N^2\p _x\bbar{v_N}\quad \to \quad -v^2\p _x\bbar{v}\qquad \text{in $C_TH^{s-1}$}\]
and
\eqq{\norm{\eqref{eq:N1}}{C_TH^{s-1}}
&\lec \norm{e^{-iJ(u_N)}}{C_TH^s}\Big( \norm{P_{>N}\p _x (|u|^2u)}{C_TH^{s-1}}+\norm{\p _x(|u|^2u-|u_N|^2u_N)}{C_TH^{s-1}}\Big) \\
&\leq C(\tnorm{u}{C_TH^s})\tnorm{P_{>N/3}u}{C_TH^s}\quad \to 0\qquad (N\to \I ),}
by which \eqref{eq:N3} and \eqref{eq:N1} can be treated.
Similarly, we have
\eqq{\norm{\eqref{eq:N2}}{C_TH^s}&\lec \tnorm{v_N}{C_TH^s}\norm{\p _x^{-1}P_{\neq c}\Re \big[ \bbar{u_N}\big\{ P_{\leq N}\p _x(|u|^2u)-\p _x(|u_N|^2u_N)\big\} \big]}{C_TH^s}\\
&\lec \tnorm{v_N}{C_TH^s}\tnorm{u_N}{C_TH^s}\norm{P_{\leq N}\p _x(|u|^2u)-\p _x(|u_N|^2u_N)}{C_TH^{s-1}}\\
&\leq C(\tnorm{u}{C_TH^s})\tnorm{P_{>N/3}u}{C_TH^s}\quad \to 0\qquad (N\to \I ).
}
As a result, we obtain the equation \eqref{GDNLS'} for $v$, which is satisfied in the sense of distribution.
\end{proof}

\subsection{Main trilinear estimates and proof}

In what follows, we consider solutions of \eqref{GDNLS} in the sense of distribution.
We restate the equation as
\eqs{\p _tw=i\p _x^2w-\Big( w^2\p _x\bbar{w}-2P_c\big( w\p _x\bbar{w}\big) w \Big) +N[w],\\
N[w]:=\frac{i}{2}|w|^4w-iP_c(|w|^2)|w|^2w+i\Big( \big( P_c(|w|^2)\big) ^2-\frac{1}{2}P_c(|w|^4)\Big) w.}
Let $w\in C_TH^s$ be such a solution and define $\om \in C_T\ell ^2_s$ as 
\[ \om _n(t):=e^{itn^2}(\F w(t))_n.\]
We see that $\om$ satisfies the equation of the form \eqref{Abs}:
\begin{gather}
\p _t\om _n(t)=\sum _{\mat{n=n_1-n_2+n_3\\ n_2\neq n_1,n_3}}e^{it\phi (n,n_1,n_2,n_3)}in_2\om _{n_1}(t)\bbar{\om _{n_2}(t)}\om _{n_3}(t)+\Sc{R}[\om (t)]_n,\label{eq_om}\\
\phi (n,n_1,n_2,n_3):=n^2-n_1^2+n_2^2-n_3^2=2(n_2-n_1)(n_2-n_3)\in \Bo{Z}\setminus \shugo{0}\quad \text{in the summation},\notag \\
\Sc{R}[\om (t)]_n:=-in|\om _{n}(t)|^2\om _{n}(t)+e^{itn^2}(\F N[e^{it\p _x^2}\F ^{-1}\om (t)])_n.\notag
\end{gather}
More precisely, we consider the system of equations for $\om _n(t)$ and $\bbar{\om _{n}(t)}$, but again there is no substantial difference (see the argument in Section~\ref{section:cNLS}).
Here, we set the multiplier $m(n,n_1,n_2,n_3)=in_2\chi _{n_2\neq n_1,n_3}(n_1,n_2,n_3)$.

We estimate harmless terms $\Sc{R}[\om ]$ as the first step:
\begin{lem}\label{lem:R}
Let $s>1/2$.
Then, we have
\eqq{\norm{\Sc{R}[\om ]}{\ell ^2_s}&\lec \Big( 1+\tnorm{\om}{\ell ^2_s}^4\Big) \tnorm{\om}{\ell ^2_s},\\
\norm{\Sc{R}[\om ]-\Sc{R}[\ti{\om}]}{\ell ^2_s}&\lec \Big( 1+\tnorm{\om}{\ell ^2_s}^4+\tnorm{\ti{\om}}{\ell ^2_s}^4\Big) \tnorm{\om -\ti{\om}}{\ell ^2_s}}
for any $\om ,\ti{\om}\in \ell ^2_s$.
\end{lem}

\begin{proof}
We only consider the first estimate, because the second one can be shown by a similar argument.
By the embedding $\ell ^2\hookrightarrow \ell ^6$ we see that
\[ \norm{n|\om |^2\om}{\ell ^2_s}\le \norm{\LR{n}^{(1+s)/3}\om}{\ell ^6}^3\le \tnorm{\om}{\ell ^2_{(1+s)/3}}^3,\]
which evaluates the first term in $\Sc{R}[\om ]$ since $(1+s)/3\le s$ if $s\ge 1/2$.
The other terms can be easily handled with the Sobolev inequality.
\end{proof}

In view of Theorem~\ref{thm:abstract} [A] (with $X=\ell ^2_{s-1}$), UU for \eqref{eq_om} in $\ell ^2_s$, $s>1/2$ (and thus Theorem~\ref{thm:DNLS}, by the argument in the preceding subsection) follows once we establish the following:
\begin{lem}
Let $1/2<s<1$.
There exists $\de >0$ such that we have
\begin{align}
\norm{\sum _{\mat{n=n_1-n_2+n_3\\ n_2\neq n_1,n_3}}\frac{n_2}{\LR{\phi}^{1/2}}\om _{n_1}^{(1)}\om _{n_2}^{(2)}\om _{n_3}^{(3)}}{\ell ^2_s}&\lec \norm{\om ^{(1)}}{\ell ^2_s}\norm{\om ^{(2)}}{\ell ^2_s}\norm{\om ^{(3)}}{\ell ^2_s},\label{lem:trilinear}\\
\norm{\sum _{\mat{n=n_1-n_2+n_3\\ n_2\neq n_1,n_3}}\frac{n_2}{\LR{\phi}^{1-\de}}\om _{n_1}^{(1)}\om _{n_2}^{(2)}\om _{n_3}^{(3)}}{\ell ^2_{s-1}}&\lec \norm{\om ^{(k)}}{\ell ^2_{s-1}}\prod _{l\in \shugo{1,2,3}\setminus \shugo{k}}\norm{\om ^{(l)}}{\ell ^2_s}\label{lem:trilinear2}
\end{align}
for any $k=1,2,3$.
\end{lem}
Note that the assumption $(A3)$ in Theorem~\ref{thm:abstract} is easily deduced from the estimate \eqref{est:Sob-prod}.

\begin{proof}[Proof of \eqref{lem:trilinear}]
By the Cauchy-Schwarz inequality it suffices to prove
\[ A_n:=\sum _{\mat{n=n_1-n_2+n_3\\ n_2\neq n_1,n_3}}\frac{|n_2|^2\LR{n}^{2s}}{|n_2-n_1||n_2-n_3|\LR{n_1}^{2s}\LR{n_2}^{2s}\LR{n_3}^{2s}}\lec 1\]
uniformly for $n\in \Bo{Z}$.
By symmetry we may assume that $|n_1|\ge |n_3|$.
We consider several cases separately.

Case 1: $|n_2|\gg |n_1|$.
We have $|n_2-n_1|\sim |n_2-n_3|\sim |n_2|\gec |n|$, so
\[ A_n\lec \sum _{n_1,n_3}\frac{1}{\LR{n_1}^{2s}\LR{n_3}^{2s}}\lec 1.\]

Case 2: $|n_2|\sim |n_1|\sim |n_3|$.
In this case we have
\[ A_n\lec \sum _{n_1,n_3}\frac{1}{|n-n_3||n-n_1|\LR{n_1}^{2s-1}\LR{n_3}^{2s-1}}\lec 1.\]

Case 3: $|n_2|\sim |n_1|\gec |n|\gg |n_3|$.
In this case we have $|n_2-n_1|=|n-n_3|\sim |n|$ and $|n_2-n_3|\sim |n_2|$, which imply $|n_2|^2\LR{n}^{2s}\lec |n_2-n_1||n_2-n_3|\LR{n_2}^{2s}$ and
\[ A_n\lec \sum _{n_1,n_3}\frac{1}{\LR{n_1}^{2s}\LR{n_3}^{2s}}\lec 1.\]

Case 4: $|n_2|\sim |n_1|\gg |n_3|\gec |n|$.
We have
\[ A_n\lec \sum _{n_1}\frac{1}{\LR{n_1}^{2s}}\sum _{n_2}\frac{1}{|n_2-n_1|\LR{n_2}^{2s-1}}\lec 1.\]

Case 5: $|n_2|\ll |n_1|\sim |n_3|$.
It holds that $|n_2-n_1|\sim |n_2-n_3|\sim |n_1|$, which implies $|n_2|^2\LR{n}^{2s}\lec |n_2-n_1||n_2-n_3|\LR{n_2}^{2s}$ and
\[ A_n\lec \sum _{n_1,n_3}\frac{1}{\LR{n_1}^{2s}\LR{n_3}^{2s}}\lec 1.\]

Case 6: $|n_2|\lec |n_3|\ll |n_1|$.
In this case $|n_2|^2\lec \LR{n_2}\LR{n_3}$ implies
\[ A_n\lec \sum _{n_3}\frac{1}{|n-n_3|\LR{n_3}^{2s-1}}\sum _{n_2}\frac{1}{|n_2-n_3|\LR{n_2}^{2s-1}}\lec 1.\]

Case 7: $|n_3|\ll |n_2|\ll |n_1|$.
It holds $|n_2|^2\lec |n_2-n_1||n_2-n_3|$ in this case.
Therefore, we have
\[ A_n\lec \sum _{n_2,n_3}\frac{1}{\LR{n_2}^{2s}\LR{n_3}^{2s}}\lec 1.\]
Now, we have seen all the cases.
\end{proof}

\begin{proof}[Proof of \eqref{lem:trilinear2}]
In the proof, $\e$ denotes various positive small constants.
Similarly to the proof of \eqref{lem:trilinear} it suffices to show that
\[ B_n:=\sum _{\mat{n=n_1-n_2+n_3\\ n_2\neq n_1,n_3}}\frac{|n_2|^2\LR{n}^{2s}}{|n_2-n_1|^{2-\e}|n_2-n_3|^{2-\e}\LR{n_1}^{2s}\LR{n_2}^{2s}\LR{n_3}^{2s}}\frac{\LR{n_{\max}}^2}{\LR{n}^2}\lec 1\]
uniformly for $n\in \Bo{Z}$, where $|n_{\max}|:=\max \shugo{|n_1|,|n_2|,|n_3|}$.
Alternatively, we take $\e >0$ such that $s':=s-\e>1/2$, then we have
\eqq{B_n &\lec \sum _{\mat{n=n_1-n_2+n_3\\ n_2\neq n_1,n_3}}\frac{|\mu |^{\e}}{\LR{n_1}^{2\e}\LR{n_2}^{2\e}\LR{n_3}^{2\e}}\frac{|n_2|^2\LR{n}^{2s'}}{|n_2-n_1|^2|n_2-n_3|^2\LR{n_1}^{2s'}\LR{n_2}^{2s'}\LR{n_3}^{2s'}}\frac{\LR{n_{\max}}^2}{\LR{n}^{2-2\e}}\\
&\lec \sum _{\mat{n=n_1-n_2+n_3\\ n_2\neq n_1,n_3}}\frac{|n_2|^2\LR{n}^{2s'}}{|n_2-n_1|^2|n_2-n_3|^2\LR{n_1}^{2s'}\LR{n_2}^{2s'}\LR{n_3}^{2s'}}\frac{\LR{n_{\max}}^2}{\LR{n}^{2-2\e}},}
since it holds $|\mu|\lec |n_{\max}|^2$.
Therefore, it also suffices to show
\[ C_n:=\sum _{\mat{n=n_1-n_2+n_3\\ n_2\neq n_1,n_3}}\frac{|n_2|^2\LR{n}^{2s}}{|n_2-n_1|^2|n_2-n_3|^2\LR{n_1}^{2s}\LR{n_2}^{2s}\LR{n_3}^{2s}}\frac{\LR{n_{\max}}^2}{\LR{n}^{2-2\e}}\lec 1\]
uniformly for $n\in \Bo{Z}$.
By symmetry we may assume that $|n_1|\ge |n_3|$.
We will take a similar decomposition of analysis.

Case 0a: $|n_{\max}|\sim |n|$.
In this case $B_n\lec A_n$ and the proof is reduced to \eqref{lem:trilinear}.

Case 0b: $|\mu|\sim |n_{\max}|^2$.
Since $|n_2-n_1||n_2-n_3|\sim |n_{\max}|^2$ we have $C_n\lec A_n$, from which this is also reduced to \eqref{lem:trilinear}.

Case 1: $|n_2|\gg |n_1|$.
This case is reduced to Case 0a, since $|n|\sim |n_2|=|n_{\max}|$.

Case 2: $|n_2|\sim |n_1|\sim |n_3|$.
If $|\mu |\sim |n-n_1||n-n_3|\sim |n_{\max}|^2$, the proof is reduced to Case 0b.
Otherwise, we have either $|n|\sim |n_1|$ or $|n|\sim |n_3|$, and the proof is reduced to Case 0a.

Case 3\&4: $|n_2|\sim |n_1|\gg |n_3|$.
In this case both $B_n$ and $C_n$ may be unbounded and we have to modify the proof.
We may assume $|n|\ll |n_1|$, otherwise the proof is reduced to Case 0a.
We restrict $n$ and $n_l$ ($l=1,2,3$) into dyadic intervals $R:=\shugo{N\le \LR{n}<2N}$ and $R_l:=\shugo{N_l\le \LR{n_l}<2N_l}$, and furthermore restrict $n_1,n_2$ onto intervals $Q_1,Q_2$ of length $\sim N_*:=\max \shugo{N,N_3}$.
Since $|n_2-n_3|\sim N_2$, we have
\eqq{C_n(R,Q_1,Q_2,R_3):=&\;\sum _{\mat{n=n_1-n_2+n_3;\,n_2\neq n_1,n_3\\ n\in R,n_1\in Q_1,n_2\in Q_2,n_3\in R_3}}\frac{|n_2|^2\LR{n}^{2s}}{|n_2-n_1|^2|n_2-n_3|^2\LR{n_1}^{2s}\LR{n_2}^{2s}\LR{n_3}^{2s}}\frac{\LR{n_{\max}}^2}{\LR{n}^{2-2\e}}\\
\lec &\;\frac{N_1^4N^{2s}}{N_1^{2+4s}N_3^{2s}N^{2-2\e}}\sum _{n_1\in Q_1,\,n_3}\frac{1}{|n-n_3|^2}\lec \frac{N_*}{N_1^{4s-2}N_3^{2s}N^{2-2s-2\e}}\lec N_1^{-2\e}}
if $s-2\e \ge 1/2$, which implies
\eqq{&\norm{\sum _{\mat{n=n_1-n_2+n_3\\ n_2\neq n_1,n_3;\, |\mu |>M}}\frac{n_2}{\mu}[\om _1\chi _{Q_1}](n_1)[\om _2\chi _{Q_2}](n_2)[\om _3\chi _{R_3}](n_3)}{\ell ^2_{s-1}(R)}\\
&\hx \lec M^{-\e /2}N_1^{-\e}\norm{\om _1}{\ell ^2_{s-1}(Q_1)}\norm{\om _2}{\ell ^2_s(Q_2)}\norm{\om _3}{\ell ^2_s(R_3)}.}
We note that $Q_2$ is determined almost uniquely for given $Q_1$ if the contribution is non-zero.
Therefore, after summing up over $Q_1,Q_2$ with the Cauchy-Schwarz inequality, we obtain
\eqq{&\norm{\sum _{\mat{n=n_1-n_2+n_3\\ n_2\neq n_1,n_3;\, |\mu |>M}}\frac{n_2}{\mu}[\om _1\chi _{R_1}](n_1)[\om _2\chi _{R_2}](n_2)[\om _3\chi _{R_3}](n_3)}{\ell ^2_{s-1}(R)}\\
&\hx \lec M^{-\e /2}N_1^{-\e}\norm{\om _1}{\ell ^2_{s-1}(R_1)}\norm{\om _2}{\ell ^2_s(R_2)}\norm{\om _3}{\ell ^2_s(R_3)}.}
The factor $N_1^{-\e}$ allows summation over $N,N_1,N_2,N_3$ such that $N_1\sim N_2\gg N,N_3$, which implies the desired estimate.

Case 5: $|n_2|\ll |n_1|\sim |n_3|$.
This is reduced to Case 0b, because $|n_2-n_1|\sim |n_2-n_3|\sim |n_{\max}|$.

Case 6\&7: $|n_2|,|n_3|\ll |n_1|$.
This is reduced to Case 0a.
\end{proof}

%%%%%%%%%%%%%%%%%%%%%%%%%%%%%%%%%%%%%%%%%%%%%%
% Zakharov
%%%%%%%%%%%%%%%%%%%%%%%%%%%%%%%%%%%%%%%%%%%%%%

\bigskip
\section{Application to the Zakharov system} \label{section:Z}

As an application of the abstract theory to a system, we consider the Zakharov system on $\T$:
\begin{equation}\label{Z}
\left\{ \begin{aligned}
&i\p _tu+\p _x^2u=Nu,\quad \p _t^2N-\p_x^2N=\p _x^2 (|u|^2),\qquad (u,N)(t,x):[0,T]\times \T \to \Bo{C}\times \R , \\
&(u,N,\p _tN) \big| _{t=0}=(u_0,N_0,N_1)\quad \in \Sc{H}^{s,l}(\T ):=H^s(\T ;\Bo{C})\times H^l(\T ;\R )\times H^{l-1}(\T ;\R ).
\end{aligned}\right.
\end{equation}
For \eqref{Z}, we will prove uniqueness of solution $(u,N)$ in the class
\[ C_T\Sc{H}^{s,l}(\T ) :=C([0,T];H^s(\T ;\Bo{C}))\times \big[ C([0,T];H^l(\T ;\R ))\cap C^1([0,T];H^{l-1}(\T ;\R ))\big] . \]
Concerning the local well-posedness in the periodic setting, Takaoka \cite{T99-2} obtained a sharp result in the case of $d=1$, and the author treated higher dimensional cases in \cite{K13}.
For UU, Masmoudi and Nakanishi \cite{MN09} obtained the results in the energy space for $\R ^d$ with $d=1,2,3$, while there is no result in the periodic case.

We shall prove the following:
\begin{thm}\label{thm1d}
The solution (in the sense of distribution) to the Cauchy problem \eqref{Z} is unique in $C_T\Sc{H}^{1/2,0}(\T )$.
\end{thm}

\begin{rem}
A similar argument with nonlinear estimates established in \cite{K13} yields some UU results in two and higher dimensions; see \cite{K-Zp}.
We remark that the energy space $(s,l)=(1,0)$ is included in one and two dimensional cases with arbitrary spatial period and aspect ratio.
\end{rem}

\subsection{Reduction to a first order system}

Let $(u,N)\in C_T\Sc{H}^{s,l}$ be a solution (in the sense of distribution) to \eqref{Z}.
We introduce a new complex-valued function 
\[ w(t):=N(t)+i\LR{\p_x}^{-1}\p _tN(t) \in C_TH^l(\T ;\Bo{C}).\]
It is easily checked that $(u,w)\in C_T\big( H^s(\T ;\Bo{C})\times H^l(\T ;\Bo{C})\big) $ is a solution (in the sense of distribution) to 
\begin{equation}\label{Z'}
\left\{ \begin{aligned}
&i\p _tu+\p_x^2 u=\frac{1}{2}(w+\bbar{w})u,\quad i\p _tw-\LR{\p_x}w=-\frac{\p_x^2}{\LR{\p_x}}(|u|^2)-\frac{w+\bbar{w}}{2\LR{\p_x}},\quad (t,x)\in [0,T]\times \T ,\\
&(u,w)\big|_{t=0}=(u_0,w_0)\quad \in H^s(\T ;\Bo{C})\times H^l(\T ;\Bo{C}),
\end{aligned}\right.
\end{equation}
where $w_0:=N_0+i\LR{\p_x}^{-1}N_1$.

Hereafter, we assume $s,l\ge 0$.
Under this assumption, all the nonlinear terms in \eqref{Z'} can be considered as continuous functions in $t$ with values in some Sobolev spaces with respect to $x$.
In particular, 
$\hat{u}_n:=(\F u)_n$ and $\hat{w}_n:=(\F w)_n$ belong to $C^1([0,T])$ for each $n\in \Bo{Z}$ and they solve the following system in the classical sense:
\begin{equation}\label{Z2}
\left\{
\begin{alignedat}{2}
\p _t\hat{u}_n&=-in^2\hat{u}_n-\frac{i}{2}\sum _{n'\in \Bo{Z}}\big( \hat{w}_{n'}\hat{u}_{n-n'}+\bbar{\hat{w}_{-n'}}\hat{u}_{n-n'}\big) ,& &(t,n)\in [0,T]\times \Bo{Z},\\
\p _t\hat{w}_n&=-i\LR{n}\hat{w}_n-\frac{i n^2}{\LR{n}}\sum _{n'\in \Bo{Z}}\hat{u}_{n'}\bbar{\hat{u}_{n'-n}}+\frac{i}{2\LR{n}}\big( \hat{w}_n+\bbar{\hat{w}_{-n}}\big) ,&\qquad &(t,n)\in [0,T]\times \Bo{Z},\\
(\hat{u}_n,~&\hat{w}_n) \big| _{t=0}=\big( (\F u_0)_n,~(\F w_0)_n\big) ,& &n\in \mathbb{Z}.
\end{alignedat}
\right.
\end{equation}
Note that the sum on the right-hand side of each equation converges absolutely.
Set new functions
\eqs{\psi ^+_n(t):=e^{in^2t}\hat{u}_n(t),\qquad \psi ^-_n(t):=e^{-in^2t}\bbar{\hat{u}_{-n}(t)},\\[-5pt]
\om ^+_n(t):=e^{i\LR{n}t}\hat{w}_n(t),\qquad \om ^-_n(t):=e^{-i\LR{n}t}\bbar{\hat{w}_{-n}(t)}.}
Then, these function lie in the space $\psi ^{\pm}\in C_T\ell ^2_s(\Bo{Z})$, $\om ^\pm \in C_T\ell ^2_l(\Bo{Z})$, continuously differentiable on $[0,T]$ for each $n$, and satisfy the system
\begin{equation}\label{Z3}
\left\{
\begin{alignedat}{2}
\p _t\psi ^\pm _{n_1}&=\mp \frac{i}{2}\sum _{\mat{n_0,n_2\in \Bo{Z}\\ n_1=n_0+n_2}}\big( e^{\pm i\Phi_-t}\om ^\pm _{n_0}\psi ^\pm _{n_2}+e^{\pm i\Phi_+t}\om ^\mp _{n_0}\psi ^\pm _{n_2}\big) ,& &(t,n_1)\in [0,T]\times \Bo{Z},\\
\p _t\om ^\pm _{n_0}&=\mp \frac{i n_0^2}{\LR{n_0}}\!\!\!\sum _{\mat{n_1,n_2\in \Bo{Z}\\ n_0=n_1+n_2}}\!\!\!e^{\pm i\Phi_+t}\psi ^\mp _{n_1}\psi ^\pm _{n_2}~\pm \frac{i(\om ^\pm _{n_0}+e^{\pm 2i\LR{n_0}t}\om ^\mp _{n_0})}{2\LR{n_0}},&\quad &(t,n_0)\in [0,T]\times \Bo{Z},\\
( \psi ^+_n,~&\psi ^-_n, \om ^+_n, \om ^-_n) \big| _{t=0}=\big( (\F u_0)_n,(\F \bbar{u_0})_n,(\F w_0)_n,(\F \bbar{w_0})_n\big) ,& &n\in \Bo{Z},
\end{alignedat}
\right.
\end{equation}
where $\Phi_\pm :=n_1^2-n_2^2\pm \LR{n_0}$.

\subsection{Main bilinear estimates and proof}

To prove unconditional uniqueness for the system \eqref{Z3} in $(\ell ^2_s)^2\times (\ell ^2_l)^2$, we apply Theorem~\ref{thm:abstract} [A] with $X=(\ell ^2_{s-1})^2\times (\ell ^2_{l-1})^2$ and 
\[ \Sc{R} [(\psi ^+,\psi ^-,\om ^+ ,\om ^-)]_n=\Big( 0,\,0,\,\frac{i(\om ^+_{n}+e^{2i\LR{n}t}\om ^-_{n})}{2\LR{n}},\,-\frac{i(\om ^-_{n}+e^{-2i\LR{n}t}\om ^+_{n})}{2\LR{n}}\Big) .\]
The assumption $(R)$ trivially holds, and $(A1)$--$(A3)$ can be interpreted into the following estimates with some $\e >0$:
\begin{gather*}
\norm{\sum _{n_1=n_0+n_2}\frac{f_{n_0}h_{n_2}}{\LR{\Phi_\pm}^{1/2}}}{\ell ^2_s(\Bo{Z}_{n_1})}\lec \tnorm{f}{\ell ^2_l}\tnorm{h}{\ell ^2_s},\\
\norm{|n_0|\sum _{n_0=n_1-n_2}\frac{g_{n_1}h_{n_2}}{\LR{\Phi _\pm}^{1/2}}}{\ell ^2_{l}(\Bo{Z}_{n_0})}\lec \tnorm{g}{\ell ^2_s}\tnorm{h}{\ell ^2_s},\\
\norm{\sum _{n_1=n_0+n_2}\frac{\LR{n_0}+\LR{n_2}}{\LR{n_1}}\frac{f_{n_0}h_{n_2}}{\LR{\Phi_\pm}^{1-\e}}}{\ell ^2_s(\Bo{Z}_{n_1})}\lec \tnorm{f}{\ell ^2_l}\tnorm{h}{\ell ^2_s},\\
\norm{|n_0|\sum _{n_0=n_1-n_2}\frac{\LR{n_1}+\LR{n_2}}{\LR{n_0}}\frac{g_{n_1}h_{n_2}}{\LR{\Phi _\pm}^{1-\e}}}{\ell ^2_{l}(\Bo{Z}_{n_0})}\lec \tnorm{g}{\ell ^2_s}\tnorm{h}{\ell ^2_s},\\
\tnorm{f*h}{\ell ^2_{s-1}}\lec \tnorm{f}{\ell ^2_l}\tnorm{h}{\ell ^2_s},\qquad \tnorm{g*h}{\ell ^2_l}\lec \tnorm{g}{\ell ^2_s}\tnorm{h}{\ell ^2_s}
\end{gather*}
for any real-valued non-negative sequences $f\in \ell ^2_l(\Bo{Z})$, $g,h\in \ell ^2_s(\Bo{Z})$, where $*$ denotes the convolution (and in the second and the fourth inequalities we have replaced $n_2$ with $-n_2$ for convenience).

The Sobolev inequality shows the last two estimates if $(s,l)=(\frac{1}{2},0)$.
The first four estimates are equivalent by duality to the trilinear estimates:
\eq{est:trilinear}{\sum _{\mat{n_0,n_1,n_2\in \Bo{Z}\\ n_0=n_1-n_2}}W_j(n_0,n_1,n_2)f_{n_0}g_{n_1}h_{n_2}\lec \tnorm{f}{\ell ^2}\tnorm{g}{\ell ^2}\tnorm{h}{\ell ^2},\qquad j=1,\dots ,4}
for real-valued non-negative sequences $f,g,h\in \ell ^2(\Bo{Z})$, where
\eqs{W_1=\frac{\LR{n_1}^s}{\LR{\Phi _\pm}^{1/2}\LR{n_0}^{l}\LR{n_2}^s},\qquad W_2=\frac{\LR{n_0}^{l}|n_0|}{\LR{\Phi _\pm}^{1/2}\LR{n_1}^s\LR{n_2}^s},\\
W_3=\frac{\LR{n_1}^{s-1}\big( \LR{n_0}+\LR{n_2}\big)}{\LR{\Phi _\pm}^{1-\e}\LR{n_0}^{l}\LR{n_2}^s},\qquad W_4=\frac{\LR{n_0}^{l-1}|n_0|\big( \LR{n_1}+\LR{n_2}\big)}{\LR{\Phi _\pm}^{1-\e}\LR{n_1}^s\LR{n_2}^s}.}

For our purpose, it suffices to show the following:
\begin{prop}\label{prop:trilinear}
The estimate \eqref{est:trilinear} holds if $(s,l)=(\frac{1}{2},0)$ and $\e =\frac{1}{2}$.
\end{prop}

\begin{proof}
Since $W_j\lec 1$ if $n_0=0$, \eqref{est:trilinear} holds when the left-hand side is restricted to $\{ n_0=0\}$.
Assuming $n_0\neq 0$, we observe that $\LR{\Phi_\pm}\sim \LR{n_1^2-n_2^2\pm |n_0|}=\LR{n_0(n_1+n_2\pm \mathrm{sgn}(n_0))}$ under the relation $n_0=n_1-n_2$.
Although it is still possible that $\Phi_\pm=0$, we have
\eq{factorization}{&\LR{\Phi _\pm}\sim \LR{n_0}\LR{n_1+n_2\pm \mathrm{sgn}(n_0)}\sim \LR{n_0}\LR{n_1+n_2}\\
&\quad \text{whenever}~ n_0\neq 0~\text{and}~n_1+n_2\pm \mathrm{sgn}(n_0)\neq 0.}

In the case where the factorization \eqref{factorization} is valid, a simple case-by-case analysis yields that
\eqq{
W_1+W_2\lec W_3+W_4&\sim \frac{\LR{n_0}+\LR{n_1}+\LR{n_2}}{\LR{n_0}^{1/2}\LR{n_1+n_2}^{1/2}\LR{n_1}^{1/2}\LR{n_2}^{1/2}}\\
&\lec \frac{\chf{|n_1|\gg |n_2|}}{\LR{n_0}^{1/2}\LR{n_2}^{1/2}}+\frac{\chf{|n_1|\sim |n_2|}}{\LR{n_0}^{1/2}\LR{n_1+n_2}^{1/2}}+\frac{\chf{|n_1|\ll |n_2|}}{\LR{n_0}^{1/2}\LR{n_1}^{1/2}}.
}
under the assumption $(s,l)=(\frac{1}{2},0)$, $\e =\frac{1}{2}$.
This estimate and the H\"older inequality easily show \eqref{est:trilinear}; for instance, when the middle term is dominant, we can argue as
\eqq{
&\sum _{n_0,n_2\in \Bo{Z}}\frac{f_{n_0}g_{n_0+n_2}h_{n_2}}{\LR{n_0}^{1/2}\LR{n_0+2n_2}^{1/2}}\leq \sum _{n_2}h_{n_2}\sum _{n_0}\frac{f_{n_0}g_{n_0+n_2}}{\LR{n_0}}+\sum _{n_0}f_{n_0}\sum _{n_2}\frac{g_{n_0+n_2}h_{n_2}}{\LR{n_0+2n_2}}\\
&\quad \lec \sum _{n_2}h_{n_2}\Big( \sum _{n_0}f_{n_0}^2g_{n_0+n_2}^2\Big) ^{1/2}+\sum _{n_0}f_{n_0}\Big( \sum _{n_2}g_{n_0+n_2}^2h_{n_2}^2\Big) ^{1/2} \lec \tnorm{f}{\ell ^2}\tnorm{g}{\ell ^2}\tnorm{h}{\ell ^2}.
}

It remains to consider the case $n_1+n_2\pm \mathrm{sgn}(n_0)=0$.
Since there are at most two possible pairs of $(n_0,n_2)$ for each $n_1$, it suffices to show $W_j\lec 1$ under the assumption $(s,l)=(\frac{1}{2},0)$.
When $|n_1|\lec 1$, it holds that $|n_0|, |n_2|\lec 1$, and trivially $W_j\lec 1$.
When $|n_1|\gg 1$, we see that $\LR{n_0}\sim \LR{n_1}\sim \LR{n_2}$, and hence $W_1\sim W_2\sim W_3\sim W_4\sim 1$.
\end{proof}

%%%%%%%%%%%%%%%%%%%%%%%%%%%%%%%%%%%
%%%%%%%%%%%%%%%%%%%%%%%%%%%%%%%%%%%

%\bigskip
\section{Existence of weak solutions}\label{section:W}

As seen in the previous sections, the infinite NFR machinery is useful to establish a priori estimates for a \emph{rough} solution and difference of two rough solutions.
This method has also been used to obtain a priori (difference) estimates for \emph{regular} solutions, which yield existence of certain weak solutions at low regularity such that the nonlinearity may not be well-defined in the distributional framework; see \cite{GKO13} and subsequent works.
In this section, we discuss such a use of NFR for the abstract equation \eqref{Abs}.

\subsection{Statement of the result}

We rewrite the equation as
\begin{gather}
\label{Abs'} \p _t\om _n=\Sc{N}[\om ]_n+\Sc{R}[\om ]_n,\qquad (t,n)\in (0,T)\times \Bo{Z}^d,\\
\notag \Sc{N}[\om ]_n=\sum _{n=n_1+\dots +n_p}e^{it\phi}m\om _{n_1}(t)\om _{n_2}(t)\cdots \om _{n_p}(t),
\end{gather}
where the notations are the same as before.

Let us first recall the definition of weak solutions in the extended sense due to Christ~\cite{C05p,C07}.
Consider the following nonlinear dispersive equation:
\eq{Abs'-2}{\p _tu=i\psi (D_x)u+N[u]+R[u],\qquad (t,x)\in (0,T)\times \T ^d,}
where $\psi (D_x):=\F ^{-1}_n\psi (n)\F_x$ is a spatial Fourier multiplier with a real-valued symbol $\psi$, and $N[u]$ is the principal part of the nonlinearity which may not be well-defined at the level of regularity under consideration, whereas $R[u]$ denotes the other nonlinear part which is basically assumed to be well-defined.

\begin{defn}\label{defn:weak'}
(i) We define a \emph{sequence of Fourier cutoff operators} as a sequence of Fourier multipliers $\{ T_k=\F ^{-1}m_k\F \} _{k\in \Bo{N}}$ with symbols $m_k:\Bo{Z}^d\to \Bo{C}$ such that
\eq{cond:fco}{
\left\{ \begin{aligned}
~&\text{$m_k$ is compactly supported for each $k\in \Bo{N}$,}\\
&\text{$\sup _{k\in \Bo{N}}\tnorm{m_k}{\ell ^\I (\Bo{Z}^d)}<\I$,\hx $\lim _{k\to \I}m_k(n)=1$ for each $n\in \Bo{Z}^d$.}
\end{aligned}
\right.
}

(ii)
Let $u\in C_TH^{-\I}(\T ^d)$, and suppose there exists a distribution $w\in \Sc{D}'((0,T)\times \T ^d)$ such that for any sequence of Fourier cutoff operators $\{ T_k\}_{k\in \Bo{N}}$,  $N[T_ku]$ is a well-defined distribution for each $k$ and the sequence $\{ N[T_ku]\} _{k\in \Bo{N}}$ converges to $w$ in $\Sc{D}'((0,T)\times \T ^d)$.
Then, we say \emph{$N[u]$ exists and is equal to $w$}.

(iii)
Let $u\in C_TH^{-\infty}(\T^d)$, and suppose that $R[u]$ is a well-defined distribution on $(0,T)\times \T^d$.
We say $u$ is a \emph{weak solution of \eqref{Abs'-2} in the extended sense} if the nonlinearity $N[u]$ exists in the sense of (ii) and \eqref{Abs'-2} holds in $\Sc{D}'((0,T)\times \T ^d)$.
\end{defn}

Based on the above definition, we define the weak solutions of \eqref{Abs'} as follows.
Let $\Sc{D}(\Bo{Z}^d):=\F C^\I (\T^d)=\ell ^2_\I (\Bo{Z}^d)$, and define $\Sc{D}'((0,T)\times \Bo{Z}^d)$ as the dual space of $C^\I _c((0,T);\Sc{D}(\Bo{Z}^d))$.

\begin{defn}\label{defn:weak}
(i)
For $\om \in C_T\ell^2_{-\I}(\Bo{Z}^d)$, we say $\Sc{N}[\om ]$ exists and is equal to $\zeta\in \Sc{D}'((0,T)\times \Bo{Z}^d)$ if for any sequence $\{ m_k\} _{k\in \Bo{N}}$ of functions on $\Bo{Z}^d$ satisfying \eqref{cond:fco}, $\Sc{N}[m_k\om ]$ is well-defined and the sequence $\{ \Sc{N}[m_k\om ]\} _{k\in \Bo{N}}$ converges to $\zeta$ in $\Sc{D}'((0,T)\times \Bo{Z}^d)$.

(ii)
Let $\om \in C_T\ell^2_{-\infty}(\Bo{Z}^d)$, and suppose that $\Sc{R}[\om ]$ is a well-defined distribution on $(0,T)\times \Bo{Z}^d$.
We say $\om$ is a weak solution of \eqref{Abs'} if the nonlinearity $\Sc{N}[\om ]$ exists in the sense of (i) and \eqref{Abs'} holds in $\Sc{D}'((0,T)\times \Bo{Z}^d)$.
\end{defn} 

The goal of this section is to show the following:
\begin{thm}\label{thm:weak}
Let $s\in \R$.
Assume that for any $s'\ge s$ and $T>0$, $\Sc{R}[\om]\in C_T\ell^2_{s'}$ if $\om \in C_T\ell^2_{s'}$ and
\[
(R)'\qquad \left\{
\begin{aligned}
\norm{\Sc{R}[\om ]}{C_T\ell ^2_{s'}}&\le C\big( s',\tnorm{\om}{C_T\ell ^2_s}\big) \tnorm{\om}{C_T\ell ^2_{s'}},\\
~\norm{\Sc{R}[\om ]-\Sc{R}[\ti{\om}]}{C_T\ell ^2_{s'}}&\le C\big( s',\tnorm{\om}{C_T\ell ^2_{s'}},\,\tnorm{\ti{\om}}{C_T\ell ^2_{s'}}\big) \tnorm{\om-\ti{\om}}{C_T\ell ^2_{s'}}.
\end{aligned}
\right. 
\]
Moreover, assume one of the following $[A]'$, $[B]'$:

$[A]'$ There exists $s_2>s$ such that
\eqq{
&(A1)\qquad \norm{\sum _{n=n_1+\dots +n_p}\frac{|m|}{\LR{\phi}^{1/2}}\om ^{(1)}_{n_1}\cdots \om^{(p)}_{n_p}}{\ell^2_s}\le C\prod _{j=1}^p\norm{\om ^{(j)}}{\ell^2_s},\\
&(A1)'\qquad \norm{\sum _{n=n_1+\dots +n_p}|m|\om ^{(1)}_{n_1}\cdots \om^{(p)}_{n_p}}{\ell^2_s}\le C\prod _{j=1}^p\norm{\om ^{(j)}}{\ell^2_{s_2}}.
}

$[B]'$ There exist $s_1<s$ and $s_2>s$ such that
\eqq{
&(B1)\qquad \sup _{\mu \in \Bo{Z}}\norm{\sum _{\mat{n=n_1+\dots +n_p\\ \mu \le \phi <\mu +1}}|m|\om ^{(1)}_{n_1}\cdots \om^{(p)}_{n_p}}{\ell^2_{s_1}}\le C\prod _{j=1}^p\norm{\om ^{(j)}}{\ell^2_{s_1}},\\
&(B1)'\qquad \norm{\sum _{n=n_1+\dots +n_p}|m|\om ^{(1)}_{n_1}\cdots \om^{(p)}_{n_p}}{\ell^2_{s_2}}\le C\prod _{j=1}^p\norm{\om ^{(j)}}{\ell^2_{s_2}}.
}

Then, for any $r>0$ there exists $T=T(r)>0$ such that the following holds.
For any $\om _0\in \ell^2_s$ with $\tnorm{\om _0}{\ell^2_s}\le r$, there exists a weak solution $\om \in C_T\ell ^2_s$ of \eqref{Abs'} on $(0,T)\times \T ^d$ in the sense of Definition~\ref{defn:weak} with initial condition $\om (0)=\om _0$, satisfying the following properties:
\begin{enumerate}
\item If $\om _0\in \ell ^2_{s'}$ for some $s'>s$, then the solution $\om$ belongs to $C_T\ell ^2_{s'}$.
\item If $\om _0\in \ell ^2_{s_2}$, then $\om$, which is in $C_T\ell ^2_{s_2}$ and satisfies the equation in the sense of distribution,%
\footnote{More precisely, $\om$ satisfies the equation in $C_T\ell ^2_s$ for Case $[A]'$ and in $C_T\ell ^2_{s_2}$ for Case $[B]'$, by the estimates $(R)'$ and $(A1)'$ or $(B1)'$.}
is the unique solution in $C_T\ell ^2_{s_2}$; i.e., unconditional uniqueness holds in $\ell ^2_{s_2}$.
\item Let $T'>0$, and let $\zeta \in C_{T'}\ell ^2_s$ be any \emph{function} which is the $C_{T'}\ell^2_s$-limit of some sequence of regular \emph{solutions} of \eqref{Abs'} in $C_{T'}\ell ^2_{s_2}$ and satisfies $\tnorm{\zeta (0)}{\ell ^2_s}\le r$.
Then, $\zeta$ coincides with the weak solution $\om \in C_{T(r)}\ell ^2_s$ constructed above with initial condition $\om (0)=\zeta (0)$ on the time interval $[0,\min \{ T',T(r)\} ]$.%
\footnote{In \cite{OWp}, this kind of uniqueness assertion was referred to as enhanced uniqueness.
Note that this does not claim uniqueness in the class of all weak solutions in the sense of Definition~\ref{defn:weak}.}
\item The solution map $\om _0\mapsto \om$ from $\{ \om _0\in \ell^2_s:\tnorm{\om _0}{\ell ^2_s}\le r\}$ into $C_T\ell ^2_s$ is Lipschitz continuous, with the bounds
\[ \tnorm{\om}{C_T\ell ^2_s}\le 2\tnorm{\om _0}{\ell ^2_s},\qquad \tnorm{\om -\ti{\om}}{C_T\ell ^2_s}\le 2\tnorm{\om _0-\ti{\om}_0}{\ell ^2_s}.\]
\end{enumerate}
\end{thm}

\begin{rem}
(i) Consider the situation that the equation \eqref{Abs'} comes from \eqref{Abs'-2} by the transform $u\mapsto \om (t):=\F e^{-it\psi (D_x)}u(t)$, with
\[ \Sc{N}[\om] =e^{-it\psi (n)}\F N[\F ^{-1}e^{it\psi (n)}\om ],\qquad \Sc{R}[\om ]=e^{-it\psi (n)}\F R[\F ^{-1}e^{it\psi (n)}\om ]. \]
Then, existence of weak solutions for the equation \eqref{Abs'-2} can be recovered from that for \eqref{Abs'}; i.e., the weak solution $\om \in C_T\ell ^2_s$ of \eqref{Abs'} constructed in Theorem~\ref{thm:weak} gives a weak solution $v\in C_TH^s$ of \eqref{Abs'-2} in the extended sense by $v(t):=\F ^{-1}e^{it\psi (n)}\om (t)$.
In fact, for any sequence $\{ T_k=\F ^{-1}m_k\F\} _{k\in \Bo{Z}}$ of Fourier cutoff operators, we have the convergence $e^{it\psi (n)}\Sc{N}[m_k\om ]\to e^{it\psi (n)}\Sc{N}[\om]$ in $\Sc{D}'((0,T)\times \Bo{Z}^d)$, where $\Sc{N}[\om]$ is the unique distributional limit as in Definition~\ref{defn:weak}, which exists since $\om$ is a weak solution of \eqref{Abs'} in the sense of Definition~\ref{defn:weak}.
This implies that $N[T_kv]\to \F^{-1}e^{it\psi (n)}\Sc{N}[\om ]$ in $\Sc{D}'((0,T)\times \T ^d)$, and hence $N[v]$ exists and is equal to $\F^{-1}e^{it\psi (n)}\Sc{N}[\om ]$.
On the other hand, we have $\Sc{N}[\om ]=\p _t\om -\Sc{R}[\om ]$ by the equation and
\eqq{&\LR{\F^{-1}e^{it\psi (n)}(\p _t\om -\Sc{R}[\om ]),\phi}
=\LR{\p _t\om ,e^{it\psi (n)}\F ^{-1}\phi}-\LR{\F^{-1}e^{it\psi (n)}\Sc{R}[\om ],\phi}\\
&=-\LR{e^{it\psi (n)}\om ,\F ^{-1}\p _t\phi}-\LR{i\psi (n)e^{it\psi (n)}\om ,\F ^{-1}\phi}-\LR{R[v],\phi}\\
&=\LR{\p _tv-i\psi (D_x)v-R[v],\phi},\qquad \phi \in \Sc{D}((0,T)\times \T ^d).}
Hence, $N[v]=\p _tv-i\psi (D_x)v-R[v]$ in $\Sc{D}'((0,T)\times \T ^d)$, and $v$ is a weak solution of \eqref{Abs'-2} in the extended sense.

(ii) One may consider changing the space of initial data $\ell ^2_s$ similarly to Theorem~\ref{thm:abstract} (see Remark~\ref{rem:thm2} (v)).
When the condition $[B]'$ holds, one can replace $\ell^2_s$ with any Banach space $Y$ satisfying the property \eqref{absolutenorm} and embedded into $\ell ^2_{-\infty}(\mathbb{Z}^d)$ (so that the definition of the weak solution makes sense for $\om \in C_TY$).
Then, the spaces $\ell ^2_{s_1}$, $\ell ^2_{s_2}$ and $\ell^2_{s'}$ appearing in the statement may be replaced with $\LR{\cdot}^{-(s_1-s)}Y$, $\LR{\cdot}^{-(s_2-s)}Y$ and $\LR{\cdot}^{-(s'-s)}Y$, respectively.
Regarding the case $[A]'$, the same is true if existence of local-in-time regular solutions (Lemma~\ref{lem:regular} below) still holds in the new space $\LR{\cdot}^{-(s_2-s)}Y$.
We note, however, that our proof of Lemma~\ref{lem:regular} to be given in Section~\ref{subsec:regular} relies on energy estimates based on the Hilbert space structure of $\ell ^2_s$. 
\end{rem}

\subsection{A priori estimates on regular solutions}

A basic tool to prove Theorem~\ref{thm:weak} is the following a priori estimates for regular solutions.
\begin{lem}\label{lem:apriori}
Let $s\in \R$, and assume the hypotheses in Theorem~\ref{thm:weak}.
Then, for any $R>0$ there exists $T_s=T_s(R)>0$ such that any solution $\om \in C_T\ell ^2_{s_2}$ of \eqref{Abs'} on $(0,T)\times \Bo{Z}^d$ with $T\in (0,T_s]$ satisfying $\tnorm{\om}{C_T\ell ^2_s}\le R$ solves the limit equation \eqref{eq:limit} in $C_T\ell ^2_s$.
Moreover, the following estimates hold:
\eq{apriori-s}{\tnorm{\om}{C_T\ell ^2_s}\le 2\tnorm{\om (0)}{\ell ^2_s},\qquad \tnorm{\om-\ti{\om}}{C_T\ell ^2_{s}}\le 2\tnorm{\om (0)-\ti{\om}(0)}{\ell ^2_{s}},}
where $\ti{\om}\in C_T\ell ^2_{s_2}$ is another solution of \eqref{Abs'} with $\tnorm{\ti{\om}}{C_T\ell ^2_s}\le R$.
Furthermore, for any $s'>s$ there exists $T_{s'}=T_{s'}(R)\in (0,T_s(R)]$ (which is decreasing in $s'$) such that if in addition $T\in (0,T_{s'}]$ and $\om \in C_T\ell ^2_{s'}$, then the limit equation \eqref{eq:limit} is satisfied in $C_T\ell ^2_{s'}$ and 
\eq{apriori-s'}{\tnorm{\om}{C_T\ell ^2_{s'}}\le 2\tnorm{\om (0)}{\ell ^2_{s'}}.}
\end{lem}

\begin{proof}
To verify \eqref{eq:limit}, we follow the argument in Section~\ref{section:proof}.
In fact, with the assumptions $(R)'$ and $(A1)+(A1)'$ or $(B1)+(B1)'$, the same argument as in the proof of Proposition~\ref{prop:abstract} derives the multilinear estimates
\eqq{\norm{\Sc{N}^{(J)}_R[\om ]-\Sc{N}^{(J)}_R[\ti{\om}]}{\ell ^2_s}&\le CM\Big[ CM^{-\de}\big( \tnorm{\om}{\ell ^2_s}+\tnorm{\ti{\om}}{\ell ^2_s}\big) ^{p-1}\Big] ^J\tnorm{\om -\ti{\om}}{\ell ^2_s},\\
\norm{\Sc{N}^{(J)}_0[\om ]-\Sc{N}^{(J)}_0[\ti{\om}]}{\ell ^2_s}&\le C\Big[ CM^{-\de}\big( \tnorm{\om}{\ell ^2_s}+\tnorm{\ti{\om}}{\ell ^2_s}\big) ^{p-1}\Big] ^J\tnorm{\om -\ti{\om}}{\ell ^2_s},\\
\norm{\Sc{R}^{(J)}[\om ]-\Sc{R}^{(J)}[\ti{\om}]}{\ell ^2_s}&\le C\Big[ CM^{-\de}\big( \tnorm{\om}{\ell ^2_s}+\tnorm{\ti{\om}}{\ell ^2_s}\big) ^{p-1}\Big] ^JC\big( \tnorm{\om}{\ell ^2_s},\tnorm{\ti{\om}}{\ell ^2_s}\big) \tnorm{\om -\ti{\om}}{\ell ^2_s}}
for any $\om ,\ti{\om}\in \ell ^2_s$ and
\[ \norm{\Sc{N}^{(J+1)}[\om ]}{\ell ^2_s}\le C\Big[ CM^{-\de}\tnorm{\om}{\ell ^2_s}^{p-1}\Big] ^{J} \tnorm{\om}{\ell ^2_{s_2}}^p\]
for any $\om \in \ell ^2_{s_2}$, where $\de =\frac12$ in Case $[A]'$ and $\de =\frac{s-s_1}{s_2-s_1}$ in Case $[B]'$.
These estimates (with $\ti{\om}=0$) suffice to justify, for any \emph{regular} solution $\om \in C_T\ell ^2_{s_2}$ of \eqref{Abs'}, the derivation of the equations \eqref{eq:Jth} for any $J\ge 1$%
\footnote{By a similar argument as for Theorem~\ref{thm:abstract}, we can show $\om _n\in C^1([0,T])$ for each $n$ and a solution $\om \in C_T\ell ^2_{s_2}$ of \eqref{Abs'}.
Here, the estimate $(A1)'$ or $(B1)'$ plays the same role as $(A3)$ (or $(B3)$) in the proof of Theorem~\ref{thm:abstract}.
} 
as well as convergence of the sums over $j$ and the limit $\Sc{N}^{(J)}[\om ]\to 0$ in $C_T\ell^2_s$ as $J\to \I$, whenever $\tnorm{\om}{C_T\ell ^2_s}\le R$ and $\eta :=CT^{\de}R^{p-1}\ll 1$ (after taking $M=T^{-1}$).
Letting $J\to \I$ we have \eqref{eq:limit}.

From \eqref{eq:limit} and the above estimates on $\Sc{N}_R^{(J)}$, $\Sc{N}_0^{(J)}$ and $\Sc{R}^{(J)}$, we obtain
\eqs{\tnorm{\om}{C_T\ell ^2_s}\le \tnorm{\om (0)}{\ell ^2_s}+C\Big( \sum _{j=1}^\I \eta ^j+C(R)T\sum _{j=0}^\I \eta ^j\Big) \tnorm{\om}{C_T\ell ^2_s},\\
\tnorm{\om -\ti{\om}}{C_T\ell ^2_s}\le \tnorm{\om (0)-\ti{\om}(0)}{\ell ^2_s}+C\Big( \sum _{j=1}^\I \eta ^j+C(R)T\sum _{j=0}^\I \eta ^j\Big) \tnorm{\om -\ti{\om}}{C_T\ell ^2_s}.
}
By possibly choosing smaller $T$ so that $C(R)T\ll 1$ (still depending only on the $\ell ^2_s$-norm of the solution), we have the desired a priori estimates \eqref{apriori-s}.

Using the above estimates on $\Sc{N}_R^{(J)}$, $\Sc{N}_0^{(J)}$, $\Sc{R}^{(J)}$ and the inequality
\eq{est:weight}{\LR{n_1+n_2+\dots +n_N}^{\al}\le N^{\al}\max _{1\le j\le N}\LR{n_j}^{\al}\qquad (N\ge 1,\hx \al >0),}
we see that for any $s'>s$
\eqq{\norm{\Sc{N}^{(J)}_R[\om ]}{\ell ^2_{s'}}&\le CM\Big[ CM^{-\de}\tnorm{\om}{\ell ^2_s}^{p-1}\Big] ^J\tnorm{\om}{\ell ^2_{s'}},\\
\norm{\Sc{N}^{(J)}_0[\om ]}{\ell ^2_{s'}}&\le C\Big[ CM^{-\de}\tnorm{\om}{\ell ^2_s}^{p-1}\Big] ^J\tnorm{\om}{\ell ^2_{s'}},\\
\norm{\Sc{R}^{(J)}[\om ]}{\ell ^2_{s'}}&\le C\Big[ CM^{-\de}\tnorm{\om}{\ell ^2_s}^{p-1}\Big] ^JC\big( \tnorm{\om}{\ell ^2_s}\big) \tnorm{\om}{\ell ^2_{s'}},
}
where we have applied the first inequality in $(R)'$ in the estimate on $\Sc{R}^{(J)}$ and the constants $C$ are now depending also on $s'$.
Then, the same argument as above shows that the right-hand side of \eqref{eq:limit} converges in $C_T\ell ^2_{s'}$ and the estimate \eqref{apriori-s'} holds, if we take possibly smaller $T$ depending also on $s'$ but still independent of the $\ell ^2_{s'}$-norm of the solution.
\end{proof}

To construct weak solutions at the $\ell ^2_s$ level, we need regular solutions with approximating initial data.
The next lemma, which we will prove at the end of this section, ensures existence of regular solutions.
\begin{lem}\label{lem:regular}
Let $s\in \R$, and assume the hypotheses in Theorem~\ref{thm:weak}.
Then, for any $\om _0\in \ell ^2_{s_2}$ there exist $T=T(\tnorm{\om _0}{\ell ^2_{s_2}})>0$ and a solution $\om \in C_T\ell ^2_{s_2}$ of \eqref{Abs'} on $(0,T)\times \Bo{Z}^d$ satisfying $\om (0)=\om _0$.
\end{lem}

Assuming Lemma~\ref{lem:regular} for the moment, we can prove Theorem~\ref{thm:weak} by almost the same argument as in \cite[Section~4]{GKO13}.
We give a proof for completeness.
\begin{proof}[Proof of Theorem~\ref{thm:weak}] 
Let $\om _0\in \ell ^2_s$, and set $R:=2\tnorm{\om _0}{\ell ^2_s}+1$.
Take a sequence $\{ \om _{N,0}\}_{N\in \Bo{N}}\subset \ell ^2_{s_2}$ such that $\om _{N,0}\to \om _0$ in $\ell ^2_s$ as $N\to \I$.
(For instance, we can take $\om _{N,0}=\chf{\{ |n|\le N\}}\om _0$.)
We may assume that $\tnorm{\om _{N,0}}{\ell ^2_s}\le \tnorm{\om _0}{\ell ^2_s}+1/3$ for any $N$.

First of all, we need to show existence of solutions for approximating regular initial data $\{ \om _{N,0}\}$ with a uniform existence time.
Set $T:=T_{s_2}(R)~(\le T_s(R))$, with $T_{s'}(R)$ given in Lemma~\ref{lem:apriori}.
For $N\ge 1$, let $\om _N\in C_{T'}\ell ^2_{s_2}$ be the solution of \eqref{Abs'} with $\om _N(0)=\om _{N,0}$ given in Lemma~\ref{lem:regular} with the existence time $T'$ corresponding to the size $2\tnorm{\om _{N,0}}{\ell ^2_{s_2}}$, and suppose $T'<T$.
Let $T'':=\sup \{ t\in [0,T']:\tnorm{\om _N}{C_t\ell ^2_s}\le R\}$.
Since $\tnorm{\om _{N,0}}{\ell ^2_s}\le \tnorm{\om _0}{\ell ^2_s}+1/3\le R/2$ and the map $t\mapsto \tnorm{\om _N}{C_t\ell ^2_s}$ is continuous, we see $T''>0$.
For any $t\in (0,T'']\cap (0,T')$, the a priori estimate \eqref{apriori-s} shows $\tnorm{\om _N}{C_t\ell ^2_s}\le 2\tnorm{\om _{N,0}}{\ell ^2_s}\le 2\tnorm{\om _0}{\ell ^2_s}+2/3<R$.
This implies $t<T''$, and hence $T''=T'$.
Therefore, we see that $\tnorm{\om _N}{C_{T'}\ell ^2_s}\le R$, and by \eqref{apriori-s'}, that $\tnorm{\om _N(T')}{\ell ^2_{s_2}}\le 2\tnorm{\om _{N,0}}{\ell ^2_{s_2}}$.
Applying Lemma~\ref{lem:regular} with initial datum $\om _N(T')$, we can extend the solution up to time $2T'$.%
\footnote{%
More precisely, we apply Lemma~\ref{lem:regular} with initial datum $\om _N(T'/2)$ and obtain a solution $\ti{\om}\in C_{T'}\ell ^2_{s_2}$ with $\ti{\om}(0)=\om _N(T'/2)$.
Then, it follows from the difference estimate \eqref{apriori-s} and a continuity argument that $\ti{\om}(t)=\om _{N}(t+T'/2)$ for $t\in [0,T'/2]$. 
Therefore, we have a solution on $[0,3T'/2]$ with initial condition $\om _{N,0}$ at $t=0$.} 
With a uniform time increment $T'$, this procedure can be iterated until the existence time exceeds $T$, giving a solution $\om _N\in C_T\ell ^2_{s_2}$ on the time interval $[0,T]$ which is uniform in $N$.

From the above argument we observe that $\tnorm{\om _N}{C_T\ell ^2_s}\le 2\tnorm{\om _{0,N}}{\ell ^2_s}\le R$ for any $N$.
The difference estimate \eqref{apriori-s} then shows that $\{ \om _{N} \} _{N\in \Bo{N}}$ is a Cauchy sequence in $C_T\ell ^2_s$.
Let $\om$ be the limit of $\{ \om _N \} _N$ in $C_T\ell ^2_s$. 
Note that 
\eq{limit-N}{\tnorm{\om}{C_T\ell ^2_s}\le 2\tnorm{\om}{\ell ^2_s},\qquad \tnorm{\om -\om _N}{C_T\ell ^2_s}\le 2\tnorm{\om _0-\om _{N,0}}{\ell ^2_s}.}

We next prove that $\om$ is a weak solution of \eqref{Abs'} in the sense of Definition~\ref{defn:weak}.
Since $\Sc{R}[\om _N]\to \Sc{R}[\om ]$ in $C_T\ell ^2_s$ by $(R)'$ and \eqref{limit-N}, we see from the equation that the sequence $\{ \Sc{N}[\om _N]\} _N=\{ \p _t\om _N-\Sc{R}[\om _N]\} _N$ has the limit $\zeta \in \Sc{D}'((0,T)\times \Bo{Z}^d)$ which satisfies $\p _t\om =\zeta +\Sc{R}[\om ]$ in $\Sc{D}'((0,T)\times \Bo{Z}^d)$.
It then suffices to show that, for any sequence $\{ m_k \}_{k\in \Bo{N}}$ of functions on $\Bo{Z}^d$ satisfying the conditions \eqref{cond:fco}, the limit $\lim\limits _{k\to \I} \Sc{N}[m_k\om ]$ exists in $\Sc{D}'((0,T)\times \Bo{Z}^d)$ and coincides with $\zeta$.
To this end, fix a test function $\phi \in C_c^\I ((0,T);\ell ^2_\I (\Bo{Z}^d))$, and evaluate
\eqq{&\big| \BLR{\zeta -\Sc{N}[m_k\om ]}{\phi}\big| \\
&\quad \le \big| \BLR{\zeta -\Sc{N}[\om _N]}{\phi}\big| +\big| \BLR{\Sc{N}[\om _N]-\Sc{N}[m_k\om _N]}{\phi}\big| +\big| \BLR{\Sc{N}[m_k\om _N]-\Sc{N}[m_k\om ]}{\phi}\big| \\
&\quad =:I(N)+II(N,k)+III(N,k).
}
By definition, we have $I(N)\to 0$ as $N\to \I$.
Next, the assumption $(A1)'$ or $(B1)'$ implies
\[ II(N,k)\le C\big( \tnorm{\om _N}{C_T\ell ^2_{s_2}}+\tnorm{m_k\om _N}{C_T\ell ^2_{s_2}}\big) ^{p-1}\tnorm{(1-m_k)\om _N}{C_T\ell ^2_{s_2}}\tnorm{\phi}{L^1_T\ell ^2_{-s}},\]
which tends to $0$ as $k\to \I$ for any fixed $N\in \Bo{N}$ by the assumption \eqref{cond:fco} and the dominated convergence theorem.%
\footnote{This argument has to be modified if we consider an $\ell^\infty$-type space $Y$ instead of $\ell^2_s$ (and then $\tilde{Y}:=\LR{\cdot}^{-(s_2-s)}Y$ instead of $\ell ^2_{s_2}$).
Noticing that the estimate $(A1)'$ or $(B1)'$ still holds if $\tilde{Y}$ is changed to a stronger norm $\LR{\cdot}^{-\alpha}\tilde{Y}$ with any $\alpha >0$ (in view of the property \eqref{absolutenorm} of the $Y$ norm and the inequality \eqref{est:weight}), we may initially choose $\tilde{Y}$ so that the estimate $(A1)'$ or $(B1)'$ is valid in a slightly weaker norm $\LR{\cdot}^{\e}\tilde{Y}$.
Then, we have
\[ \| \Sc{N}[\om _N]-\Sc{N}[m_k\om _N]\|_{Y}\lesssim \big( \| \om _N\|_{\tilde{Y}}+\| m_k\om _N \|_{\tilde{Y}}\big) ^{p-1}\tnorm{(1-m_k)\LR{n}^{-\e}\om _N}{\tilde{Y}}\lesssim \| (1-m_k)\LR{n}^{-\e}\|_{\ell ^\infty (\mathbb{Z}^d)}\| \om _N\|_{\tilde{Y}}^p.\]
Since $\| (1-m_k)\LR{n}^{-\e}\|_{\ell ^\infty}\to 0$ ($k\to \infty$) by the assumption \eqref{cond:fco}, it follows that $\Sc{N}[\om _N]-\Sc{N}[m_k\om _N]\to 0$ in $L^\infty _TY$, and that $II(N,k)\to 0$, as $k\to \infty$ (for each fixed $N$).}
Therefore, we have only to show that $III(N,k)\to 0$ as $N\to \I$ uniformly in $k$.

Now, we apply the NFR procedure as in Section~\ref{section:proof} to the nonlinear term $\Sc{N}[m_k\om _N]$ for each $k,N\in \Bo{N}$.
(Recall that $\om _N\in C_T\ell ^2_{s_2}$, $\tnorm{\om _N}{C_T\ell ^2_s}\le R$ and $T\le T_s(R)$ which is defined in Lemma~\ref{lem:apriori}.)
We obtain the equations
\[ \int _0^t\Sc{N}[m_k\om _N]_n=\sum _{j=1}^{J-1}\Sc{N}^{(j),k}_{0}[\om _N]_n\Big| _0^t+\int _0^t\,\Big( \sum _{j=1}^{J-1}\Sc{N}^{(j),k}_R[\om _N]_n+\sum _{j=1}^{J-1}\Sc{R}^{(j),k}[\om _N]_n+\Sc{N}^{(J),k}[\om _N]_n\Big) \]
for $J\ge 1$, where $\Sc{N}^{(j),k}_R$ stands for the $[(j-1)p+1]$-linear form $\Sc{N}^{(j)}_R$ with $m^1(n,n_1,\dots ,n_p)$ replaced by $m_k(n_1)\cdots m_k(n_p)m^1(n,n_1,\dots ,n_p)$, and similarly for $\Sc{N}^{(j),k}_0$, $\Sc{R}^{(j),k}$, and $\Sc{N}^{(J),k}$.
Since $m_k$'s are uniformly bounded, these multilinear forms obey the same bounds as those without $m_k$'s, and the same argument as in the proof of Lemma~\ref{lem:apriori} verifies the limit equation
\[ \int _0^t\Sc{N}[m_k\om _N]_n=\sum _{j=1}^{\I}\Sc{N}^{(j),k}_{0}[\om _N]_n\Big| _0^t+\int _0^t\sum _{j=1}^{\I}\Big( \Sc{N}^{(j),k}_R[\om _N]_n+\Sc{R}^{(j),k}[\om _N]_n\Big) \]
as well as the difference estimate
\[ \bigg\| \int _0^t\Big( \Sc{N}[m_k\om _N] -\Sc{N}[m_k\om _{N'}]\Big) \bigg\| _{L^\infty_T\ell ^2_s}\le C\Big( \sup _{k}\tnorm{m_k}{\ell ^\I}\Big) \tnorm{\om _N-\om _{N'}}{C_T\ell ^2_s}.\]
This in particular implies
\eqq{\big| \BLR{\Sc{N}[m_k\om _N] -\Sc{N}[m_k\om _{N'}]}{\phi}\big| &=\left| \BLR{\int _0^t\Big( \Sc{N}[m_k\om _N] -\Sc{N}[m_k\om _{N'}]\Big)}{\p _t\phi}\right| \\
&\le C\Big( \sup _{k}\tnorm{m_k}{\ell ^\I}\Big) \tnorm{\om _N-\om _{N'}}{C_T\ell ^2_s}\tnorm{\p _t\phi}{L^1_T\ell ^2_{-s}},}
and hence the sequence $\{ \BLR{\Sc{N}[m_k\om _N]}{\phi}\} _{N}$ converges uniformly in $k$.
On the other hand, $(A1)'$ or $(B1)'$ and the assumptions \eqref{cond:fco} on $m_k$'s imply
\eqq{III(N,k)&\le C\big( \tnorm{m_k\om _N}{C_T\ell ^2_{s_2}}+\tnorm{m_k\om}{C_T\ell ^2_{s_2}}\big) ^{p-1}\tnorm{m_k(\om _N-\om )}{C_T\ell ^2_{s_2}}\tnorm{\phi}{L^1_T\ell ^2_{-s}}\\
&\le C\Big( \tnorm{m_k}{\ell ^\I},\,|\mathrm{supp}(m_k)|\Big) \big( \tnorm{\om _N}{C_T\ell ^2_{s}}+\tnorm{\om}{C_T\ell ^2_{s}}\big) ^{p-1}\tnorm{\om _N-\om}{C_T\ell ^2_{s}}\tnorm{\phi}{L^1_T\ell ^2_{-s}},}
which shows that $\{ \BLR{\Sc{N}[m_k\om _N]}{\phi}\} _N$ converges to $\BLR{\Sc{N}[m_k\om ]}{\phi}$ (for each fixed $k$).
We therefore see that $III(N,k)\to 0$ as $N\to \I$ uniformly in $k$.

It remains to show the properties (i)--(iv).
In view of Lemma~\ref{lem:apriori}, the bounds in (iv) are already verified in the case of regular data $\om _0\in \ell ^2_{s_2}$, and (ii) is a consequence of the difference bound.
It also shows that the map $\om _0\mapsto \om$ is well-defined in $\om _0\in \ell ^2_s$ with $\tnorm{\om _0}{\ell ^2_s}\le r$; i.e., the limit $\om$ is independent of the approximating sequence of the initial datum $\om _0$.
This fact is in turn combined with uniqueness of regular solutions to imply (iii).
Approximation by regular solutions also verifies the bounds in (iv) for general $\om _0\in \ell ^2_s$ with $\tnorm{\om _0}{\ell ^2_s}\le r$.
To show (i), we first consider the case $s'\ge s_2$.
Since the estimates $(A1)'$, $(B1)'$ still hold if $\ell ^2_{s_2}$ is replaced by $\ell ^2_{s'}$, we deduce from Lemma~\ref{lem:regular} and uniqueness of regular solutions that $\om \in C_{T'}\ell ^2_{s'}$ for possibly smaller $T'$, and then $\om \in C_T\ell^2_{s'}$ by a continuity argument.
When $s'\in (s,s_2)$, we approximate the initial datum $\om _0\in \ell ^2_{s'}$ (not only in $\ell ^2_s$, but) in $\ell ^2_{s'}$.
Then, $\tnorm{\om _{N,0}}{\ell ^2_{s'}}$ is bounded, and so \eqref{apriori-s'} shows that $\tnorm{\om _N}{C_{T'}\ell ^2_{s'}}$ is bounded for possibly smaller $T'$.
Now, we can take the difference of the limit equation \eqref{eq:limit} and estimate in $C_{T'}\ell ^2_{s'}$ as in the proof of \eqref{apriori-s}, where the smallness of $T'$ is determined (not by $\tnorm{\om _N}{C_{T'}\ell ^2_s}$, but) by $\tnorm{\om _N}{C_{T'}\ell ^2_{s'}}$.
This implies the convergence of the approximating solution in $C_{T'}\ell ^2_{s'}$, and therefore, that $\om \in C_{T'}\ell ^2_{s'}$.
Again, a continuity argument shows $\om \in C_T\ell ^2_{s'}$.
\end{proof}

\subsection{Construction of local-in-time regular solutions}
\label{subsec:regular}

Finally, we prove Lemma~\ref{lem:regular}.
\begin{proof}[Proof of Lemma~\ref{lem:regular}]
In Case $[B]'$ the claim follows from a standard fixed point argument with the estimates $(B1)'$ and $(R)'$, so we concentrate on Case $[A]'$.
We shall construct the solution by regularizing the equation and the initial data and then taking the limit on the basis of a priori energy estimates.%
\footnote{Past use of the method of infinite NFR for the purpose of deriving energy estimates can be found in \cite{OST18,OW18}, where the differential equality of an energy quantity (i.e., an inner product of two copies of a solution) was transformed and the notion of ordered bi-trees was introduced to represent the transformed equation.
We do not have to argue with bi-trees in our proof, however.}
First of all, the assumptions $(A1)$, $(A1)'$ and the inequality \eqref{est:weight} imply
\eqq{
&(A1)_+\qquad \norm{\sum _{n=n_1+\dots +n_p}\frac{|m|}{\LR{\phi}^{1/2}}\om ^{(1)}_{n_1}\cdots \om^{(p)}_{n_p}}{\ell^2_{s_2}}\le C\prod _{j=1}^p\norm{\om ^{(j)}}{\ell^2_{s_2}},\\[5pt]
&(A1)'_+\qquad \norm{\sum _{n=n_1+\dots +n_p}|m|\om ^{(1)}_{n_1}\cdots \om^{(p)}_{n_p}}{\ell^2_{s_2}}\le C\prod _{j=1}^p\norm{\om ^{(j)}}{\ell^2_{s_2+\al}},
}
where $\al :=s_2-s>0$.

For $\e \in (0,1)$, consider the equation
\eq{Abs'-p}{\p _t\om _n=-\e \LR{n}^{2\al}\om _n+\Sc{N}[\om ]_n+\Sc{R}[\om ]_n,\qquad (t,n)\in (0,T)\times \Bo{Z}^d.}
We can solve the Cauchy problem associated with \eqref{Abs'-p} in $\ell ^2_{s_2}$ by applying a fixed point argument to the integral equation 
\[ \om (t)=e^{-\e\LR{n}^{2\al}t}\om _0+\int _0^t e^{-\e \LR{n}^{2\al}(t-t')}\Big( \Sc{N}[\om ]+\Sc{R}[\om ]\Big) (t')\,dt',\]
with an existence time $\tau =\tau (\e ,\tnorm{\om _0}{\ell ^2_{s_2}})>0$.
In fact, noticing that
\[ \sup _{n}e^{-\e \LR{n}^{2\al}(t-t')}\LR{n}^{s_2-s}\lec \e ^{-1/2}(t-t')^{-1/2},\qquad t>t'>0,\]
we deduce from $(A1)'$ and $(R)'$ that
\[ \norm{\int _0^t e^{-\e \LR{n}^{2\al}(t-t')}\Big( \Sc{N}[\om ]+\Sc{R}[\om ]\Big) (t')\,dt'}{L^\infty_\tau \ell ^2_{s_2}}\lec \e ^{-1/2}\tau ^{1/2}\tnorm{\om}{C_\tau \ell ^2_{s_2}}^p+\tau C(\tnorm{\om}{C_\tau \ell ^2_{s_2}}),\]
as well as the corresponding difference estimate.
Note that by a standard argument we also have uniqueness in $C_\tau \ell ^2_{s_2}$ and persistence of regularity (i.e., $\om _0\in \ell ^2_{s'}$ for some $s'>s_2$ implies $\om \in C_\tau \ell ^2_{s'}$).

The main ingredient of the proof is the following:
\begin{lem}\label{lem:apriori-p}
Let $s\in \R$, and assume the hypotheses $(R)'$, $[A]'$ in Theorem~\ref{thm:weak}.
Then, for any $R>0$ there exists $\tau _{s_2}=\tau _{s_2}(R)>0$ independent of $\e$ such that any solution $\om ^\e \in C_\tau \ell ^2_{s_2+2\al}$ of \eqref{Abs'-p} on $(0,\tau )\times \Bo{Z}^d$ with $\e \in (0,1)$, $\tau \in (0,\tau _{s_2}]$ satisfying $\tnorm{\om ^\e}{C_\tau \ell ^2_{s_2}}\le R$ solves the limit equation 
\eq{eq:limit-p}{\p _t\bigg( \om _n^\e -\sum _{j=1}^{\I}\Sc{N}^{(j)}_{0}[\om ^\e ]_n\bigg) &=-\e \LR{n}^{2\al}\om _n^\e +\sum _{j=1}^{\I}\Big( \Sc{N}^{(j)}_R[\om ^\e ]_n-\e \Sc{P}^{(j)}[\om ^\e ]_n\Big) +\sum _{j=0}^{\I}\Sc{R}^{(j)}[\om ^\e ]_n}
in $C_\tau \ell ^2_{s_2}$, where $\Sc{P}^{(j)}[\om ]$ is the $[(p-1)j+1]$-linear form obtained by replacing one of $\om _n$ in $\Sc{N}^{(j)}_0[\om ]$ by $\LR{n}^{2\al}\om _n$;
\[ \Sc{P}^{(j)}[\om ]_n:=(-1)^j\sum _{\Sc{T}\in \mathfrak{T}(j)}\sum _{\mat{\mathbf{n}\in \mathfrak{N}_n(\Sc{T})\\ (\phi ^k)_{k=1}^{j}\in \Phi _{N\!R}^{(j)}}}\frac{e^{it\ti{\phi}^j}}{\prod\limits _{k=1}^{j}i\ti{\phi}^k}\Big[ \prod _{k=1}^jm^k\Big] \Big( \sum _{a\in \Sc{T}_\I}\LR{n_a}^{2\al}\Big) \prod _{a\in \Sc{T}_\I}\om _{n_a}.\]
Moreover, the following estimates hold:
\eq{apriori-p-s}{\tnorm{\om ^\e}{C_\tau \ell ^2_{s_2}}\le 6\tnorm{\om ^\e (0)}{\ell ^2_{s_2}},}
and
\eq{apriori-p-s2}{\tnorm{\om ^{\e _1}-\om ^{\e _2}}{C_\tau \ell ^2_{s_2}}\le 6\tnorm{\om ^{\e _1}(0)-\om ^{\e _2}(0)}{\ell ^2_{s_2}}+\e _1\tnorm{\om ^{\e _1}}{C_\tau \ell ^2_{s_2+2\al}}+\e _2\tnorm{\om ^{\e _2}}{C_\tau \ell ^2_{s_2+2\al}}}
for any $0<\e _1,\e _2<1$ and any solutions $\om ^{\e _j}\in C_\tau \ell ^2_{s_2+2\al}$ of \eqref{Abs'-p} with $\e =\e _j$ ($j=1,2$) satisfying $\tnorm{\om ^{\e _j}}{C_\tau \ell ^2_{s_2}}\le R$.
Furthermore, for any $s'>s_2$ there exists $\tau _{s'}=\tau _{s'}(R)\in (0,\tau _{s_2}(R)]$ (which is decreasing in $s'$) such that if in addition $\tau \in (0,\tau _{s'}]$ and $\om ^\e \in C_\tau \ell ^2_{s'+2\al}$, then the limit equation \eqref{eq:limit-p} is satisfied in $C_\tau \ell ^2_{s'}$ and 
\eq{apriori-p-s'}{\tnorm{\om ^\e}{C_\tau \ell ^2_{s'}}\le 6\tnorm{\om ^\e (0)}{\ell ^2_{s'}}.}
\end{lem}

\begin{proof}
Proof is similar to that of Lemma~\ref{lem:apriori}.
It is easy to formally derive the equations in the hierarchy;
\eqq{\p _t\om _n&=-\e \LR{n}^{2\al}\om _n+\p _t\bigg( \sum _{j=1}^{J-1}\Sc{N}^{(j)}_{0}[\om ]_n\bigg)\\
&\quad +\sum _{j=1}^{J-1}\Big( \Sc{N}^{(j)}_R[\om ]_n-\e \Sc{P}^{(j)}[\om ]_n\Big) +\sum _{j=0}^{J-1}\Sc{R}^{(j)}[\om ]_n+\Sc{N}^{(J)}[\om ]_n,\qquad J\ge 1,}
and we need the multilinear estimates for the terms in the above equations.  
The fundamental estimates $(R)'$, $(A1)_+$, $(A1)'_+$ and the argument in Section~\ref{section:proof} or in the proof of Lemma~\ref{lem:apriori} yield:
\eqq{\norm{\Sc{N}^{(J)}_R[\om ]}{\ell ^2_{s'}}&\le C_{s'}M\Big[ C_{s'}M^{-1/2}\tnorm{\om}{\ell ^2_{s_2}}^{p-1}\Big] ^J\tnorm{\om}{\ell ^2_{s'}},\\
\norm{\Sc{N}^{(J)}_0[\om ]}{\ell ^2_{s'}}&\le C_{s'}\Big[ C_{s'}M^{-1/2}\tnorm{\om}{\ell ^2_{s_2}}^{p-1}\Big] ^J\tnorm{\om}{\ell ^2_{s'}},\\
\norm{\Sc{R}^{(J)}[\om ]}{\ell ^2_{s'}}&\le C_{s'}\Big[ C_{s'}M^{-1/2}\tnorm{\om}{\ell ^2_{s_2}}^{p-1}\Big] ^JC\big( s',\tnorm{\om}{\ell ^2_{s_2}}\big) \tnorm{\om}{\ell ^2_{s'}},\\
\norm{\Sc{P}^{(J)}[\om ]}{\ell ^2_{s'}}&\le C_{s'}\Big[ C_{s'}M^{-1/2}\tnorm{\om}{\ell ^2_{s_2}}^{p-1}\Big] ^J\tnorm{\om}{\ell ^2_{s'+2\al}},\\
\norm{\Sc{N}^{(J+1)}[\om ]}{\ell ^2_{s'}}&\le C_{s'}\Big[ C_{s'}M^{-1/2}\tnorm{\om}{\ell ^2_{s_2}}^{p-1}\Big] ^{J} \tnorm{\om}{\ell ^2_{s_2+\al}}^{p-1}\tnorm{\om}{\ell ^2_{s'+\al}}}
for any $s'\ge s_2$ and $\om \in \ell ^2_{s'+2\al}$, and 
\eqq{\norm{\Sc{N}^{(J)}_R[\om ]-\Sc{N}^{(J)}_R[\ti{\om}]}{\ell ^2_{s_2}}&\le CM\Big[ CM^{-1/2}\big( \tnorm{\om}{\ell ^2_{s_2}}+\tnorm{\ti{\om}}{\ell ^2_{s_2}}\big) ^{p-1}\Big] ^J\tnorm{\om -\ti{\om}}{\ell ^2_{s_2}},\\
\norm{\Sc{N}^{(J)}_0[\om ]-\Sc{N}^{(J)}_0[\ti{\om}]}{\ell ^2_{s_2}}&\le C\Big[ CM^{-1/2}\big( \tnorm{\om}{\ell ^2_{s_2}}+\tnorm{\ti{\om}}{\ell ^2_{s_2}}\big) ^{p-1}\Big] ^J\tnorm{\om -\ti{\om}}{\ell ^2_{s_2}},\\
\norm{\Sc{R}^{(J)}[\om ]-\Sc{R}^{(J)}[\ti{\om}]}{\ell ^2_{s_2}}&\le C\Big[ CM^{-1/2}\big( \tnorm{\om}{\ell ^2_{s_2}}+\tnorm{\ti{\om}}{\ell ^2_{s_2}}\big) ^{p-1}\Big] ^JC\big( \tnorm{\om}{\ell ^2_{s_2}},\tnorm{\ti{\om}}{\ell ^2_{s_2}}\big) \tnorm{\om -\ti{\om}}{\ell ^2_{s_2}},
}
for any $\om ,\ti{\om}\in \ell ^2_{s_2}$.
In particular, the sums over $j$ in \eqref{eq:limit-p} converge in $C_\tau \ell ^2_{s'}$ (and all formal computations are justified) if $\om \in C_\tau \ell ^2_{s'+2\al}$, $\tnorm{\om}{C_\tau \ell ^2_{s_2}}\le R$ and $C_{s'}\tau ^{1/2}R^{p-1}\ll 1$ (after setting $M=\tau ^{-1}$).

Note that the sum of $\Sc{P}^{(j)}$ requires additional decay of $\om$.
To prove \eqref{apriori-p-s} and \eqref{apriori-p-s'}, we need to absorb it into the ``negative'' part $-\e \LR{n}^{2\al}\om _n$, which forces us to consider an energy quantity instead of the solution itself. 
For fixed $\e \in (0,1)$ and $s'\ge s_2$, let $\om \in C_\tau \ell ^2_{s'+2\al}$ be a solution of \eqref{Abs'-p} satisfying $\tnorm{\om}{C_\tau \ell ^2_{s_2}}\le R$ with $C_{s'}\tau ^{1/2}R^{p-1}\ll 1$.
Set $\ze :=\om -\sum _{j=1}^\I \Sc{N}^{(j)}_0[\om ]$, and observe
\eq{om-zeta}{\frac{1}{2}\tnorm{\om (t)}{\ell^2_{s'}}\le \tnorm{\zeta (t)}{\ell ^2_{s'}}\le \frac{3}{2}\tnorm{\om (t)}{\ell ^2_{s'}},\qquad t\in [0,\tau ].}
Taking the $\ell ^2_{s'}$-inner product of \eqref{eq:limit-p} with $\zeta$, 
\eqq{\frac{1}{2}\p _t\tnorm{\zeta (t)}{\ell ^2_{s'}}^2&=-\e \tnorm{\om (t)}{\ell ^2_{s'+\al}}^2+\e \sum _{j=1}^\I \BLR{\LR{n}^{\al}\om (t)}{\LR{n}^{\al} \Sc{N}^{(j)}_0[\om (t)]}_{\ell ^2_{s'}} -\e \sum _{j=1}^{\I}\BLR{\Sc{P}^{(j)}[\om (t)]}{\zeta (t)}_{\ell ^2_{s'}}\\
&\quad +\sum _{j=1}^{\I}\BLR{\Sc{N}^{(j)}_R[\om (t)]}{\zeta (t)}_{\ell ^2_{s'}} +\sum _{j=0}^{\I}\BLR{\Sc{R}^{(j)}[\om (t)]}{\zeta (t)}_{\ell ^2_{s'}}
,}
where $\BLR{\phi}{\psi}_{\ell ^2_{s'}}:=\Re \sum _n\LR{n}^{2s'}\phi _n\bbar{\psi _n}$.
Using the multilinear estimates obtained above and \eqref{om-zeta}, we see that
\eqq{\p _t\tnorm{\zeta (t)}{\ell ^2_{s'}}^2&\le -2\e \tnorm{\om (t)}{\ell ^2_{s'+\al}}^2+C_{s',\al}\e \bigg[ \sum _{j=1}^\I \Big( C_{s',\al}\tau ^{1/2}\tnorm{\om}{C_\tau \ell ^2_{s_2}}^{p-1}\Big) ^j \bigg] \tnorm{\om (t)}{\ell ^2_{s'+\al}}^2\\
&\quad +C_{s'}M\bigg[ \sum _{j=1}^\I \Big( C_{s'}\tau ^{1/2}\tnorm{\om}{C_\tau \ell ^2_{s_2}}^{p-1}\Big) ^j \bigg] \tnorm{\om (t)}{\ell ^2_{s'}}^2\\
&\quad +C\big( s',\tnorm{\om}{C_\tau \ell ^2_{s_2}}\big) \bigg[ \sum _{j=0}^\I \Big( C_{s'}\tau ^{1/2}\tnorm{\om}{C_\tau \ell ^2_{s_2}}^{p-1}\Big) ^j \bigg] \tnorm{\om (t)}{\ell ^2_{s'}}^2.}
Here, we have estimated $\BLR{\Sc{P}^{(j)}[\om ]}{\zeta}_{\ell ^2_{s'}}$ by moving a half of $\LR{n}^{2\al}$ on one of $\om$'s in $\Sc{P}^{(j)}$ onto either $\zeta$ or another $\om$, as
\[ \Big| \BLR{\Sc{P}^{(j)}[\om ]}{\zeta}_{\ell ^2_{s'}}\Big| \lec \Big( C_{s'}\tau ^{1/2}\tnorm{\om}{\ell ^2_{s_2}}^{p-1}\Big) ^j\tnorm{\om}{\ell ^2_{s'+\al}}\Big( \tnorm{\zeta}{\ell ^2_{s'+\al}}+\frac{\tnorm{\om}{\ell ^2_{s_2+\al}}}{\tnorm{\om}{\ell ^2_{s_2}}}\tnorm{\zeta}{\ell ^2_{s'}}\Big) ,\]
together with an interpolation inequality $\tnorm{\om}{\ell ^2_{s_2+\al}}\tnorm{\om}{\ell ^2_{s'}}\le \tnorm{\om}{\ell ^2_{s_2}}\tnorm{\om}{\ell ^2_{s'+\al}}$.
By integrating on $(0,\tau)$ with possibly smaller $\tau$ depending only on $s',\al$ and $R$, we obtain
\[ \tnorm{\zeta (t)}{\ell ^2_{s'}}^2\le \tnorm{\zeta (0)}{\ell ^2_{s'}}^2 +\frac{1}{4}\tnorm{\om}{C_\tau \ell ^2_{s'}}^2,\qquad t\in [0,\tau ],\]
and by \eqref{om-zeta},
\[ \tnorm{\om}{C_\tau \ell ^2_{s'}}^2\le 2\tnorm{\zeta}{L^\infty_\tau \ell ^2_{s'}}^2\le 2\tnorm{\zeta (0)}{\ell ^2_{s'}}^2+\frac{1}{2}\tnorm{\om}{C_\tau \ell ^2_{s'}}^2\le 3\tnorm{\om (0)}{\ell ^2_{s'}}^2+\frac{1}{2}\tnorm{\om}{C_\tau \ell ^2_{s'}}^2,\]
which yields \eqref{apriori-p-s} and \eqref{apriori-p-s'}.

We next set $s'=s_2$ and prove the difference estimate \eqref{apriori-p-s2}.
For $\ze ^{\e _j}:=\om ^{\e _j}-\sum _{j=1}^\I \Sc{N}^{(j)}_0[\om ^{\e _j}]$, $j=1,2$, the difference estimate for $\Sc{N}^{(j)}_0$ mentioned above shows
\eq{om-zeta2}{\frac{1}{2}\tnorm{\om ^{\e _1}(t)-\om ^{\e_2}(t)}{\ell^2_{s_2}}\le \tnorm{\zeta ^{\e _1}(t)-\zeta ^{\e_2}(t)}{\ell ^2_{s_2}}\le \frac{3}{2}\tnorm{\om ^{\e _1}(t)-\om ^{\e_2}(t)}{\ell ^2_{s_2}},\quad t\in [0,\tau ],}
if $\tnorm{\om ^{\e_j}}{C_\tau \ell ^2_{s_2}}\le R$ and $C_{s_2}\tau ^{1/2}R^{p-1}\ll 1$.
From \eqref{eq:limit-p}, we have
\eqq{\p _t( \zeta _n^{\e _1}-\zeta _n^{\e _2}) &= -\e _1\Big( \LR{n}^{2\al}\om _n^{\e _1}+\sum _{j=1}^\I \Sc{P}^{(j)}[\om ^{\e _1}]_n\Big) +\e _2\Big( \LR{n}^{2\al}\om _n^{\e _2}+\sum _{j=1}^\I \Sc{P}^{(j)}[\om ^{\e _2}]_n\Big)\\
&\quad +\sum _{j=1}^\I \Big( \Sc{N}_R^{(j)}[\om ^{\e_1}]_n-\Sc{N}_R^{(j)}[\om ^{\e_2}]_n\Big) +\sum _{j=0}^\I \Big( \Sc{R}^{(j)}[\om ^{\e_1}]_n-\Sc{R}^{(j)}[\om ^{\e_2}]_n\Big) .}
We integrate it in $t$, apply the estimates on multilinear forms, and take the $C_\tau \ell ^2_{s_2}$-norm, possibly choosing smaller $\tau$, to obtain
\eqq{&\tnorm{\zeta ^{\e_1}(t)-\zeta ^{\e_2}(t)}{\ell ^2_{s_2}}\\
&\quad \le \tnorm{\zeta ^{\e _1}(0)-\zeta ^{\e_2}(0)}{\ell ^2_{s_2}}+C\tau \sum _{k=1,2}\e _k\Big( 1+\sum _{j=1}^\I \big( C\tau ^{1/2}R^{p-1}\big) ^j\Big) \tnorm{\om ^{\e _k}}{C_\tau \ell ^2_{s_2+2\al}}\\
&\qquad +\bigg\{ C\tau M\sum _{j=1}^\I \big( C\tau ^{1/2}R^{p-1}\big) ^j +C(R)\tau \sum _{j=0}^\I \big( C\tau ^{1/2}R^{p-1}\big) ^j\bigg\} \tnorm{\om ^{\e _1}-\om ^{\e_2}}{C_\tau \ell ^2_{s_2}}\\
&\quad \le \tnorm{\zeta ^{\e _1}(0)-\zeta ^{\e_2}(0)}{\ell ^2_{s_2}}+\frac{1}{4}\sum _{k=1,2}\e _k\tnorm{\om ^{\e _k}}{C_\tau \ell ^2_{s_2+2\al}}+\frac{1}{4}\tnorm{\om ^{\e _1}-\om ^{\e_2}}{C_\tau \ell ^2_{s_2}},\qquad t\in [0,\tau ].
}
Finally, we use \eqref{om-zeta2} so that
\eqq{\tnorm{\om ^{\e_1}-\om ^{\e_2}}{C_\tau \ell ^2_{s_2}}&\le 2\tnorm{\zeta ^{\e_1}-\zeta ^{\e_2}}{L^\infty_\tau \ell ^2_{s_2}}\\
&\le 2\tnorm{\zeta ^{\e _1}(0)-\zeta ^{\e_2}(0)}{\ell ^2_{s_2}}+\frac{1}{2}\sum _{k=1,2}\e _k\tnorm{\om ^{\e _k}}{C_\tau \ell ^2_{s_2+2\al}}+\frac{1}{2}\tnorm{\om ^{\e _1}-\om ^{\e_2}}{C_\tau \ell ^2_{s_2}}\\
&\le 3\tnorm{\om ^{\e _1}(0)-\om ^{\e_2}(0)}{\ell ^2_{s_2}}+\frac{1}{2}\sum _{k=1,2}\e _k\tnorm{\om ^{\e _k}}{C_\tau \ell ^2_{s_2+2\al}}+\frac{1}{2}\tnorm{\om ^{\e _1}-\om ^{\e_2}}{C_\tau \ell ^2_{s_2}},}
which implies \eqref{apriori-p-s2}.
\end{proof}

\noindent \emph{Proof of Lemma~\ref{lem:regular} (continued)}.
Let $\om _0\in \ell ^2_{s_2}$, and set $R:=6\tnorm{\om _0}{\ell ^2_{s_2}}+1$.
We shall construct a solution $\om \in C_T\ell ^2_{s_2}$ of \eqref{Abs'} 
with existence time $T=\tau _{s_2+2\al}(R)$, where $\tau _{s'}(\cdot )$ is given in Lemma~\ref{lem:apriori-p}.

We define an approximating family of initial data $\{ \om ^\e _0\} _{\e \in (0,1)}$ by
\[ \om ^\e_0:=\chf{\{ \LR{n}\le N_\e \}}\om _0,\qquad N_\e :=\e^{-\frac{1}{4\al}},\]
so that $\om ^\e _0\in \ell ^2_\I$, $\tnorm{\om ^\e _0}{\ell ^2_{s_2}}\le \tnorm{\om _0}{\ell ^2_{s_2}}$, $\tnorm{\om ^\e _0}{\ell^2_{s_2+2\al}}\le \e ^{-1/2}\tnorm{\om _0}{\ell ^2_{s_2}}$, and $\om ^\e _0\to \om _0$ in $\ell ^2_{s_2}$ as $\e \to 0$.
Let $\om ^\e$ be the local-in-time solution of the perturbed equation \eqref{Abs'-p} with $\om ^\e (0)=\om ^\e _0$, which belongs to $C_t\ell ^2_{s_2+2\al}$ and is unique.
Then, similarly as in the proof of Theorem~\ref{thm:weak} above, a continuity argument with the help of a priori estimate \eqref{apriori-p-s} shows that we can extend the solution $\om^\e$ up to (uniform-in-$\e$) time $T=\tau _{s_2+2\al}(R)~(\le \tau _{s_2}(R))$ keeping its $\ell ^2_{s_2}$-norm less than $R$.
Moreover, \eqref{apriori-p-s'} implies that
\[ \tnorm{\om ^\e}{C_T\ell ^2_{s_2+2\al}}\le 6\tnorm{\om ^\e _0}{\ell^2_{s_2+2\al}}\le 6\e ^{-1/2}\tnorm{\om _0}{\ell ^2_{s_2}},\qquad \e \in (0,1).\]
Hence, \eqref{apriori-p-s2} shows that
\[ \tnorm{\om ^{\e _1}-\om ^{\e _2}}{C_T\ell ^2_{s_2}}\le 6\tnorm{\om ^{\e_1}_0-\om ^{\e_2}_0}{\ell ^2_{s_2}}+6(\e _1^{1/2}+\e _2^{1/2})\tnorm{\om _0}{\ell^2_{s_2}}\]
for any $\e _1,\e _2\in (0,1)$, and therefore $\{ \om ^\e \}_{\e \in (0,1)}$ is Cauchy in $C_T\ell ^2_{s_2}$ as $\e \to 0$.
It is then easy to see from the assumptions $(R)'$, $(A1)'$ that the limit $\om :=\lim\limits _{\e \to 0}\om ^\e \in C_T\ell ^2_{s_2}$ is a solution of \eqref{Abs'} with initial condition $\om (0)=\om _0$.
\end{proof}

%%%%%%%%%%%%%%%%%%%%%%%%%%%%%%%%%
%%%%%%%%%%%%%%%%%%%%%%%%%%%%%%%%%

\bigskip
\appendix

\section{Normal form reduction in the non-periodic case}\label{appendix:R}

In this section, we see how to extend Theorem~\ref{thm:abstract} to the non-periodic case.

We consider the following abstract equation:
\eq{Abs-R}{\p _t\om (t,\xi )=\int _{\Sc{M}^p_\xi}e^{it\phi}m\om (t,\xi _1)\om (t,\xi _2)\cdots \om (t,\xi _p)+\Sc{R}[\om ](t,\xi ),\quad (t,\xi )\in (0,T)\times \Sc{M},}
where $\Sc{M}$ is an arbitrary product of Euclidean space and integer lattice equipped with Lebesgue and counting measures, and $\int _{\mathcal{M}^p_\xi}$ means the integration over $\mathcal{M}^p$ with respect to the measure $\de (\xi_1+\cdots +\xi _p-\xi )d\xi _1\cdots d\xi _p$. 
Here, the phase part $\phi =\phi (\xi ,\xi _1,\dots ,\xi _p)\in \R$ and the multiplier part $m=m(\xi ,\xi _1,\dots ,\xi _p)\in \Bo{C}$ are measurable functions on $\Sc{M}^{p+1}$ with finite values almost everywhere.
We use weighted $L^p$ spaces: $L^p_s(\Sc{M}):=\LR{\xi }^{-s}L^p(\Sc{M})$ for $s\in \R$ and $1\leq p\leq \infty$. 

We say $\om \in L^1_{\text{loc}}((0,T)\times \mathcal{M})$ is a solution of \eqref{Abs-R} if the right-hand side of \eqref{Abs-R} is well-defined as a function in $L^1_{\text{loc}}((0,T)\times \mathcal{M})$ and $\om$ satisfies \eqref{Abs-R} in the sense of distribution.
In the following, we mainly consider the situation where both $\om$ and the right-hand side of \eqref{Abs-R} belong to $C([0,T];\ti{X})$ for some Banach space $\ti{X}$ which is continuously embedded into $L^1_{\text{loc}}(\mathcal{M})$.
In this case, a solution $\om$ is in $C^1([0,T];\ti{X})$ and satisfies the equation \eqref{Abs-R} in $\ti{X}$ for each $t\in [0,T]$.

Let us focus on the case [A] in Theorem~\ref{thm:abstract} for simplicity.
We shall prove the following result:

\begin{thm}\label{thm:abstract-R}
Let $s\in \R$, $T>0$.
Assume that $\Sc{R}[\om]\in C_TL^2_s:=C([0,T];L^2_s(\mathcal{M}))$ if $\om \in C_TL^2_s$, and that the following estimates hold; 
\begin{align*}
&(\ti{R})\hspace{25pt} \left\{
\begin{aligned}
\norm{\Sc{R}[\om ]}{C_TL^2_s}&\le C\big( \tnorm{\om}{C_TL^2_s}\big) ,\\
~\norm{\Sc{R}[\om ]-\Sc{R}[\ti{\om}]}{C_TL^2_s}&\le C\big( \tnorm{\om}{C_TL^2_s},\,\tnorm{\ti{\om}}{C_TL^2_s}\big) \tnorm{\om-\ti{\om}}{C_TL^2_s},
\end{aligned} \right. \\
&(\ti{A}1)\qquad \norm{\int _{\Sc{M}^p_\xi}\frac{|m|}{\LR{\phi}^{1/2}}\om ^{(1)}(\xi _1)\cdots \om^{(p)}(\xi _p)}{L^2_s}\le C\prod _{k=1}^p\norm{\om ^{(k)}}{L^2_s},\\[5pt]
&(\ti{A}2)\qquad \norm{\int _{\Sc{M}^p_\xi}\frac{|m|}{\LR{\phi}^{1-\de}}\om ^{(1)}(\xi _1)\cdots \om ^{(p)}(\xi _p)}{X}\le C\min _{1\le k\le p}\Big[ \norm{\om ^{(k)}}{X}\prod _{\mat{l=1\\ l\neq k}}^p\norm{\om ^{(l)}}{L^2_s}\Big] ,\\[5pt]
&(\ti{A}3)\qquad \norm{\int _{\Sc{M}^p_\xi}|m|\om ^{(1)}(\xi _1)\cdots \om ^{(p)}(\xi _p)}{X}\le C\prod _{k=1}^p\norm{\om ^{(k)}}{L^2_s}
\end{align*}
for some $\de \in (0,\frac{1}{2})$ and some Banach space $X$ which is continuously embedded into $L^1_{\mathrm{loc}}(\mathcal{M})$ and satisfies 
\[ |\om (\xi )|\le |\ti{\om} (\xi )|\quad (\mathrm{a.e.}~~\xi \in \Sc{M})\qquad \Longrightarrow \qquad \tnorm{\om}{X}\leq C\tnorm{\ti{\om}}{X}.\]
In addition, assume one of the following conditions:
\begin{enumerate}
\item If $\{ \mathcal{M}_n\} _{n\in \mathbb{N}}$ is an increasing sequence of measurable subsets of $\mathcal{M}$ such that $\bigcup _{n\in \mathbb{N}}\mathcal{M}_n$ is of full measure, then for $\om \in X$ it holds that $\| \chi _{\mathcal{M}\setminus \mathcal{M}_n}\om \|_{X}\to 0$.%
\footnote{This condition is fulfilled if $X$ is a weighted $L^p$ space with $p<\I$, for instance.
We see that the condition (i) implies the following property of $X$ (and the converse is trivial), which we will use in the proof: If a sequence $\{ \om _n\} _{n\in \mathbb{N}}\subset X$ satisfies $|\om _n|\leq \ti{\om}$ and $\om _n\to 0$ almost everywhere for some non-negative $\ti{\om}\in X$, then $\| \om _n\|_{X}\to 0$.
To show this, for given $\e >0$ we set $\mathcal{M}_n=\{ \xi : \sup_{k\geq n}|\om _k(\xi )|\leq \e \ti{\om}(\xi )\}$.
The condition implies $\| \chi _{\mathcal{M}\setminus \mathcal{M}_n} \om _n\| _X\leq C\| \chi _{\mathcal{M}\setminus \mathcal{M}_n}\ti{\om}\|_X\to 0$, so that $\limsup \| \om _n\|_X\leq \sup \| \chi _{\mathcal{M}_n}\om _n\|_{X}\leq C\e \| \ti{\om}\|_{X}$.
Taking $\e \to 0$, we have $\lim \| \om _n\|_X=0$.}
\item $\phi (\sum _{k=1}^p\xi _k,\xi _1,\dots ,\xi _p)$ is locally bounded on $\Sc{M}^{p}$.
\item $\phi =\psi (\xi )-\sum _{k=1}^p\psi (\xi _k)$ for some function $\psi$ on $\Sc{M}$ taking finite values almost everywhere.
\end{enumerate}

Then, there is at most one solution to the Cauchy problem associated with \eqref{Abs-R} in $C_TL^2_s$.
\end{thm}

\begin{rem}
The strategy is basically the same as the periodic case, and the only essential difference appears in justification of formal calculations including (a) exchange of time differentiation and integration in $(\xi _j)$, (b) application of the product rule in time differentiation.
The additional assumption (i)--(iii) in the theorem will be used for this purpose. 

In the existing results for specific equations in the non-periodic setting (see, for instance, \cite[Section~4]{KOY20}), justification of (b) relied on the fact that the nonlinear part of the equation belongs to $C_TL^1_x$ for solutions in $C_TH^s_x$, which ensures that the function $t\mapsto \om (t,\xi )$ is differentiable in the classical sense for each fixed $\xi$.
In our setting \eqref{Abs-R}, this situation corresponds to the case where one can take the space $X$ so that it is embedded in $C(\mathcal{M})$.

We will justify these operations (a), (b) in an abstract setting, assuming \emph{one of} the conditions (i)--(iii).
This assumption is quite general and it may be verified in many problems that have not been considered in the literature.
For instance, (ii) or (iii) is a condition on $\phi$ only and it allows us to take $X$ which may not be embedded in $C(\mathcal{M})$.
The condition (iii) seems naturally satisfied in applications to unconditional uniqueness problems for semilinear dispersive equations, and it also admits singular dispersion relations such as the KP-type equations.
\end{rem}

\begin{proof}[Proof of Theorem~\ref{thm:abstract-R}]
Following the argument in Section~\ref{section:proof} for the periodic case, define various multilinear terms as follows:
\[\Sc{N}^{(1)}[\om ](t,\xi ):=\int _{\Sc{M}^p_\xi}e^{it\phi}m\prod _{k=1}^p\om (t,\xi _k),\qquad
\Sc{R}^{(0)}[\om ]:=\Sc{R}[\om ],\]
and for $J\ge 1$,
\eqq{
\Sc{N}^{(J)}_R[\om ](t,\xi )&:=(-1)^{J-1}\sum _{\Sc{T}\in \mathfrak{T}(J)}\int _{\mat{\boldsymbol{\xi} \in \mathfrak{X}_\xi (\Sc{T})\\ (\phi ^j)_{j=1}^{J}\in \Phi _R^{(J)}}}\frac{e^{it\ti{\phi}^J}}{\prod\limits _{j=1}^{J-1}i\ti{\phi}^j}\Big[ \prod _{j=1}^Jm^j\Big] \prod _{a\in \Sc{T}_\I}\om (t, \xi _a),\\
\Sc{N}^{(J)}_0[\om ](t,\xi )&:=(-1)^{J-1}\sum _{\Sc{T}\in \mathfrak{T}(J)}\int _{\mat{\boldsymbol{\xi} \in \mathfrak{X}_\xi (\Sc{T})\\ (\phi ^j)_{j=1}^{J}\in \Phi _{N\!R}^{(J)}}}\frac{e^{it\ti{\phi}^J}}{\prod\limits _{j=1}^{J}i\ti{\phi}^j}\Big[ \prod _{j=1}^Jm^j\Big] \prod _{a\in \Sc{T}_\I}\om (t, \xi _a),\\
\Sc{R}^{(J)}[\om ](t, \xi )&:=(-1)^J\sum _{\Sc{T}\in \mathfrak{T}(J)}\int _{\mat{\boldsymbol{\xi} \in \mathfrak{X}_\xi (\Sc{T})\\ (\phi ^j)_{j=1}^{J}\in \Phi _{N\!R}^{(J)}}}\frac{e^{it\ti{\phi}^J}}{\prod\limits _{j=1}^{J}i\ti{\phi}^j}\Big[ \prod _{j=1}^Jm^j\Big] \sum _{a\in \Sc{T}_\I}\Big[ \prod _{\mat{b\in \Sc{T}_\I \\b\neq a}}\om (t, \xi _b)\Big] \Sc{R}[\om ](t, \xi _a),\\
\Sc{N}^{(J+1)}[\om ](t, \xi )&:=(-1)^J\sum _{\Sc{T}\in \mathfrak{T}(J+1)}\int _{\mat{\boldsymbol{\xi} \in \mathfrak{X}_\xi (\Sc{T})\\ (\phi ^j)_{j=1}^{J+1}\in \Phi _R^{(J+1)}\cup \Phi _{N\!R}^{(J+1)}}}\frac{e^{it\ti{\phi}^{J+1}}}{\prod\limits _{j=1}^{J}i\ti{\phi}^j}\Big[ \prod _{j=1}^{J+1}m^j\Big] \prod _{a\in \Sc{T}_\I}\om (t, \xi _a),
}
where the integration over $\FR{X}_\xi (\Sc{T}):=\{ \boldsymbol{\xi}:\Sc{T}\ni a\mapsto \xi _a\in \Sc{M}:\xi_{\text{root}}=\xi \}$ is made with respect to the measure $\prod _{a\in \Sc{T}_0}\de (\xi _a-\sum _{k=1}^p\xi _{a^k})d\xi _{a^1}\dots d\xi _{a^p}$ ($a^1,\dots ,a^p$ denote the children of $a\in \Sc{T}_0$).
We use the same notation as before: $\phi ^j:=\phi (\xi _{a_j},\xi _{a_j^1},\dots ,\xi _{a_j^p})$ with the $j$-th node $a_j$ of $\Sc{T}$ and its children $a_j^1,\dots ,a_j^p$, $\ti{\phi}^j:=\phi ^1+\phi ^2+\cdots +\phi ^j$, and for $M\geq 1$, 
\eqq{
\Phi _R^{1}&:=\Shugo{\phi ^1}{|\phi ^1|\le 16M},\qquad \Phi _{N\!R}^{1}:=\Shugo{\phi ^1}{|\phi ^1|>16M},\\
\Phi _R^{J}&:=\Shugo{(\phi ^j)_{j=1}^J}{|\phi ^1|>16M,\,|\ti{\phi}^{j}|>16|\ti{\phi}^{j-1}|~(2\le j\le J-1),\, |\ti{\phi}^{J}|\le 16|\ti{\phi}^{J-1}|},\\
\Phi _{N\!R}^{J}&:=\Shugo{(\phi ^j)_{j=1}^J}{|\phi ^1|>16M,\,|\ti{\phi}^{j}|>16|\ti{\phi}^{j-1}|~(2\le j\le J)}\qquad (J\ge 2).
}
By the same argument as in the periodic case, we \emph{formally} derive the equation of the $J$-th generation after the $(J-1)$-th NFR:
\eq{eq:Jth-R}{\om (\xi)\Big| _0^t=\sum _{j=1}^{J-1}\Sc{N}^{(j)}_{0}[\om ](\xi )\Big| _0^t+\int _0^t\,\Big( \sum _{j=1}^{J-1}\Sc{N}^{(j)}_R[\om ](\xi )+\sum _{j=0}^{J-1}\Sc{R}^{(j)}[\om ](\xi )+\Sc{N}^{(J)}[\om ](\xi )\Big),}
which is satisfied by any solution $\om \in C_TL^2_s$ of \eqref{Abs}.
Verification of the following multilinear estimates using the hypotheses ($\ti{R}$)--($\ti{A}$3) is also the same:
\begin{prop}\label{prop:abstract-R}
For any $J\in \Bo{N}$ and $\om \in L^2_s$, the integrals in $\boldsymbol{\xi}$ appearing in $\Sc{N}^{(J)}_R[\om ](\xi )$, $\Sc{N}^{(J)}_0[\om ](\xi )$, $\Sc{R}^{(J)}[\om ](\xi )$, $\Sc{N}^{(J)}[\om ](\xi )$ converge absolutely for almost every $\xi \in \Sc{M}$.
Moreover, for $M\ge 1$, (with $\Phi _R^{J}$ and $\Phi _{N\!R}^{J}$ defined according to $M$) we have
\eqq{
\norm{\Sc{N}^{(J)}_R[\om ]-\Sc{N}^{(J)}_R[\ti{\om}]}{L^2_s}&\le CM\Big[ CM^{-\de}\big( \tnorm{\om}{L^2_s}+\tnorm{\ti{\om}}{L^2_s}\big) ^{p-1}\Big] ^J\norm{\om -\ti{\om}}{L^2_s},\\
\norm{\Sc{N}^{(J)}_0[\om ]-\Sc{N}^{(J)}_0[\ti{\om}]}{L^2_s}&\le C\Big[ CM^{-\de}\big( \tnorm{\om}{L^2_s}+\tnorm{\ti{\om}}{L^2_s}\big) ^{p-1}\Big] ^J\norm{\om -\ti{\om}}{L^2_s},\\
\norm{\Sc{R}^{(J)}[\om ]-\Sc{R}^{(J)}[\ti{\om}]}{L^2_s}&\le C\Big[ CM^{-\de}\big( \tnorm{\om}{L^2_s}+\tnorm{\ti{\om}}{L^2_s}\big) ^{p-1}\Big] ^JC'\big( \tnorm{\om}{L^2_s},\tnorm{\ti{\om}}{L^2_s}\big) \norm{\om -\ti{\om}}{L^2_s},\\
\norm{\Sc{N}^{(J)}[\om ]-\Sc{N}^{(J)}[\ti{\om}]}{X}&\le C\Big[ CM^{-\de}\big( \tnorm{\om}{L^2_s}+\tnorm{\ti{\om}}{L^2_s}\big) ^{p-1}\Big] ^{J-1}\big( \tnorm{\om}{L^2_s}+\tnorm{\ti{\om}}{L^2_s}\big) ^{p-1}\norm{\om -\ti{\om}}{L^2_s}
}
for any $\om ,\ti{\om}\in L^2_s$, where $C,C'(\cdot ,\cdot )>0$ are independent of $J$ and $M$.
\end{prop}
Using these estimates and choosing $M$ appropriately, we see that $\Sc{N}^{(J)}[\om ]$ vanishes in $L^\infty _TX$ as $J\to \I$ and the limit equation is satisfied in $L^\infty _TL^2_s$.
Estimating the difference of two solutions in $L^\infty_{T'}L^2_s$ with $T'=\min \{ M^{-1},T\}$ (with larger $M$ if necessary), we conclude the uniqueness on $[0,T']$, as in the periodic case.

Now, we have only to justify formal calculations in the derivation of \eqref{eq:Jth-R}.
Let $\om \in C_TL^2_s$ be a solution of \eqref{Abs-R}; we notice that the right-hand side of \eqref{Abs-R}, $F[\om ]:=\mathcal{N}^{(1)}[\om ]+\mathcal{R}[\om ]$, is locally integrable in $(0,T)\times \mathcal{M}$ for any $\om \in C_TL^2_s$, by ($\ti{A}$3).
Let us first focus on the simplest case to see the idea:
Consider justification of the following equality for each $t\in [0,T]$ fixed,
\eq{eq:just2}{
&\int _0^tdt' \int _{\mat{\Sc{M}^p_\xi \\ |\phi |>16M}}e^{it'\phi}m\prod _{k=1}^p\om (t',\xi _k)\\
&\quad =\int _{\mat{\Sc{M}^p_\xi \\ |\phi |>16M}}\frac{e^{it\phi}}{i\phi}m\prod _{k=1}^p\om (t,\xi _k)~-~\int _{\mat{\Sc{M}^p_\xi \\ |\phi |>16M}}\frac{1}{i\phi}m\prod _{k=1}^p\om (0,\xi _k)\\
&\qquad -\sum _{k=1}^p\int _0^tdt' \int _{\mat{\Sc{M}^p_\xi \\ |\phi |>16M}}\frac{e^{it'\phi}}{i\phi}mF[\om ](t',\xi _k)\prod _{\mat{l=1\\ l\neq k}}^p\om (t',\xi _l)\qquad\qquad \text{($\mathrm{a.e.}~\xi \in \mathcal{M}$),}
}
which is required in the derivation of the equation \eqref{eq:Jth-R} with $J=2$.
Note that the last line of \eqref{eq:just2} is equal to $\mathcal{N}^{(2)}[\om]+\mathcal{R}^{(1)}[\om]$ by Fubini's theorem.
Since the assumptions ($\ti{A}$1)--($\ti{A}$3) show that each term in \eqref{eq:just2} is in $\ti{X}:=L^2_s+X$, it turns out that the equality actually holds as functions in $\ti{X}$.

We set
\eqq{f(t,\bfxi )&:=\chi_{|\phi (\bfxi )|>16M}m(\bfxi) \prod _{k=1}^p\om (t,\xi _k),\\
g(t,\bfxi) &:=\chi_{|\phi (\bfxi )|>16M}m(\bfxi )\sum _{k=1}^pF[\om ](t,\xi_k)\prod_{\mat{l=1\\ l\neq k}}^p\om (t,\xi _l),}
then \eqref{eq:just2} is rewritten as
\eq{eq:just3}{
\int _{t_1}^{t_2} dt \int _{\mathcal{M}^p_\xi}e^{it\phi}f(t) = \bigg[ \int _{\mathcal{M}^p_\xi} \frac{e^{it\phi}}{i\phi}f(t)\bigg] _{t=t_1}^{t=t_2} - \int _{t_1}^{t_2} dt \int _{\mathcal{M}^p_\xi}\frac{e^{it\phi}}{i\phi}g(t)\qquad \text{($\mathrm{a.e.}~\xi \in \mathcal{M}$),}
}
and we shall prove it for any $0\leq t_1<t_2\leq T$.
We first assume $0<t_1<t_2<T$, and prove \eqref{eq:just3} by taking the limit $h\to 0$ of the following identity:
\begin{align*}
&\int _{t_1}^{t_2}dt \int _{\mathcal{M}^p_\xi} \frac{e^{it\phi}}{i\phi}\frac{f(t+h)-f(t)}{h}\\
&=\int _{t_1+h}^{t_2}dt \int _{\mathcal{M}^p_\xi}\frac{e^{i(t-h)\phi}-e^{it\phi}}{ih\phi}f(t) +\frac{1}{h}\int _{t_2}^{t_2+h}dt\int _{\mathcal{M}^p_\xi}\frac{e^{i(t-h)\phi}}{i\phi}f(t) -\frac{1}{h}\int _{t_1}^{t_1+h}dt\int _{\mathcal{M}^p_\xi}\frac{e^{it\phi}}{i\phi}f(t).
\end{align*}

First, by ($\ti{A}$3) we have 
\[ \norm{\int _0^Tdt\int _{\mathcal{M}^p_\xi}|f(t,\bfxi )|}{X}\lesssim T\| \om \|_{C_TL^2_s}^p<\infty, \]
which implies that $f\in L^1((0,T)\times \mathcal{M}^p_\xi )$ for almost every $\xi \in \mathcal{M}$.
Hence,
\[ \int _{t_1+h}^{t_2}dt \int _{\mathcal{M}^p_\xi}\frac{e^{i(t-h)\phi}-e^{it\phi}}{ih\phi}f(t) ~\to ~ -\int _{t_1}^{t_2} dt \int _{\mathcal{M}^p_\xi}e^{it\phi}f(t) \qquad (h\to 0)\]
for $\mathrm{a.e.}~\xi$ by the dominated convergence theorem.

Next, we prove
\[ \frac{1}{h}\int _{t_2}^{t_2+h}dt\int _{\mathcal{M}^p_\xi}\frac{e^{i(t-h)\phi}}{i\phi}f(t) -\frac{1}{h}\int _{t_1}^{t_1+h}dt\int _{\mathcal{M}^p_\xi}\frac{e^{it\phi}}{i\phi}f(t) ~\to ~ \bigg[ \int _{\mathcal{M}^p_\xi}\frac{e^{it\phi}}{i\phi}f(t) \bigg] _{t=t_1}^{t=t_2}\qquad \text{in $L^2_s$},\]
which implies almost everywhere convergence along some sequence $h=h_n\to 0$ ($n\to \infty$). 
In fact, the $L^2_s$ convergence follows once we prove that the map
\eq{map1}{t\quad \mapsto \quad \int _{\mathcal{M}^p_\xi}\frac{e^{it\phi}}{i\phi}f(t)~\in L^2_s}
is continuous on $[0,T]$.
It then suffices to show
\[
\left. 
\begin{aligned}
\norm{\int _{\mathcal{M}^p_\xi}\frac{e^{it'\phi}}{i\phi}f(t')-\int _{\mathcal{M}^p_\xi}\frac{e^{it'\phi}}{i\phi}f(t)}{L^2_s}~\to ~0\\
\norm{\int _{\mathcal{M}^p_\xi}\frac{e^{it'\phi}-e^{it\phi}}{i\phi}f(t)}{L^2_s}~\to ~0
\end{aligned}\quad 
\right\} \qquad (t'\to t).
\]
The first limit is verified from ($\ti{A}$1) and the continuity of $t\mapsto \om (t)$ in $L^2_s$.
The second one also follows from ($\ti{A}$1) and the estimate
\[ \Big| \frac{e^{it'\phi}-e^{it\phi}}{i\phi}\Big| \lesssim \frac{\min \{ 1,\,|t'-t||\phi|\}}{\LR{\phi}}\leq \frac{|t-t'|^{1/2}}{\LR{\phi}^{1/2}}.\]

Finally, we shall prove, along some subsequence $h=h_{n'}\to 0$ of $\{ h_n \}$, that
\eq{limit1}{
\int _{t_1}^{t_2}dt \int _{\mathcal{M}^p_\xi} \frac{e^{it\phi}}{i\phi}\frac{f(t+h)-f(t)}{h} ~\to ~ \int _{t_1}^{t_2} dt \int _{\mathcal{M}^p_\xi}\frac{e^{it\phi}}{i\phi}g(t)\qquad (\mathrm{a.e.}~\xi ).}
It is only this part that we need the additional assumption (i)--(iii).
We begin by showing that $\mathcal{N}^{(1)}[\om]\in C_TX$ (and hence $F[\om ]\in C_T\ti{X}
$) for $\om \in C_TL^2_s$, under the assumption (i) or (ii).
As we have just done, it is enough to verify that
\[
\left. 
\begin{aligned}
\norm{\int _{\mathcal{M}^p_\xi}e^{it'\phi}m\Big( \prod _{k=1}^p\om (t',\xi _k) -\prod _{k=1}^p\om (t,\xi _k)\Big)}{X}~\to ~0\\
\norm{\int _{\mathcal{M}^p_\xi}\big( e^{it'\phi}-e^{it\phi}\big) m\prod _{k=1}^p\om (t,\xi_k)}{X}~\to ~0
\end{aligned}\quad
\right\} \qquad (t'\to t).
\]
The first one follows from ($\ti{A}$3) and continuity of $t\mapsto \om (t)$.
To verify the second limit in the case (i), we see from ($\ti{A}$3) and the dominated convergence theorem that the integral is bounded by $2\int _{\mathcal{M}^p_\xi}|m|\prod _k|\om (t,\xi _k)|\in X$ and converges to $0$ for $\mathrm{a.e.}~\xi$. 
This is enough to obtain the result; see the footnote to the condition (i) in Theorem~\ref{thm:abstract-R}.
For (ii), we divide the integral over $\mathcal{M}^p_\xi$ into two regions $\{ |\phi |\leq L\}$ and $\{ |\phi |>L\}$.
On one hand, the integral over $\{ |\phi |\leq L\}$ converges for any fixed $L>0$, by using $|e^{it'\phi}-e^{it\phi}|\lesssim |t-t'|L$.
On the other hand, we can show
\eq{limit2'}{\lim_{L\to \I}\norm{\int _{\Sc{M}^p_\xi}\chi _{|\phi |>L}|m|\prod _{k=1}^p|\om (t, \xi _k)|}{X}=0.}
In fact, the relation $\xi =\sum _{k=1}^p\xi _k$ and local boundedness of $\phi$ imply that $\chi _{|\phi |>L}\leq \sum _{k=1}^p\chi _{|\xi _k|>C(L)}$ for some increasing function $C(\cdot )$ with $C(L)\to \I$ as $L\to \I$.
Then, \eqref{limit2'} follows from the estimate ($\ti{A}$3) and the fact that  $\norm{\chi _{|\xi |\ge C}\om}{L^2_s}\to 0$ as $C\to \I$ whenever $\om \in L^2_s$.
Therefore, we obtain the continuity of the map $t\mapsto \mathcal{N}^{(1)}[\om ](t)$ in $X$.

As a consequence, under the assumption (i) or (ii), the solution $\om \in C_TL^2_s$ of \eqref{Abs-R} belongs to $C^1_T\ti{X}$, and in particular,
\[ \norm{\om (t+h)-\om (t)}{L^2_s}\to 0,\qquad \norm{\frac{\om (t+h)-\om (t)}{h}-F[\om ](t)}{\ti{X}}\to 0\qquad (h\to 0).\]
Note that both of the above convergence are uniform in $t\in [0,T]$.
On the other hand, we have
\eq{est:A12}{\norm{\int _{\Sc{M}^p_\xi}\frac{|m|}{\LR{\phi}^{1-\de}}\om ^{(1)}(\xi _1)\cdots \om ^{(p)}(\xi _p)}{\ti{X}}\lesssim \min _{1\le k\le p}\Big[ \norm{\om ^{(k)}}{\ti{X}}\prod _{\mat{l=1\\ l\neq k}}^p\norm{\om ^{(l)}}{L^2_s}\Big]}
from ($\ti{A}$1) and ($\ti{A}$2).
Using these estimates, we can show the convergence \eqref{limit1} in $\ti{X}$.
Since $\ti{X}$ is continuously embedded into $L^1_{\mathrm{loc}}$, this implies almost everywhere convergence along some subsequence $\{ h_{n'}\} \subset \{ h_n\}$.

It remains to consider the case (iii).
Since $\psi$ is finite almost everywhere, we have $\| \chi _{|\psi|>L}\om \|_{L^2_s}\to 0$ as $L\to \infty$ whenever $\om \in L^2_s$.
Also, when $\phi (\bfxi )=\psi (\xi)-\sum _{k=1}^p\psi (\xi_k)$ and $L>2L'$, it holds that $\chi _{|\psi (\xi )|\leq L'}\chi _{|\phi (\bfxi )|>L}\leq \sum _{k=1}^p\chi _{|\psi (\xi_k)|>L/(2p)}$.
Then, a similar argument as for the case (ii) above shows that, for each $L'>0$ and $\om \in C_TL^2_s$, the map $t\mapsto \chi _{|\psi (\cdot )|\leq L'}F[\om ](t,\cdot )$ is continuous from $[0,T]$ to $\ti{X}$.
Therefore, we have
\eq{limitex}{\norm{\chi _{|\psi|\leq L'}\Big[ \frac{\om (t+h)-\om (t)}{h}-F[\om ](t)\Big]}{\ti{X}}\to 0\qquad (h\to 0)}
uniformly in $t$, for each fixed $L'>0$.
This will in turn imply
\eq{limit1'}{
\norm{\chi _{|\psi |\leq L''}\Big[ \int _{t_1}^{t_2}dt \int _{\mathcal{M}^p_\xi} \frac{e^{it\phi}}{i\phi}\frac{f(t+h)-f(t)}{h} - \int _{t_1}^{t_2} dt \int _{\mathcal{M}^p_\xi}\frac{e^{it\phi}}{i\phi}g(t)\Big]}{\ti{X}}\to 0\quad (h\to 0)}
for each fixed $L''>0$, which is sufficient for proving almost everywhere convergence \eqref{limit1} along a subsequence.
To show \eqref{limit1'}, it suffices to verify the limits of
\[ \norm{\chi _{|\psi |\leq L''}\!\!\int _{t_1}^{t_2}\!\!dt \int _{\hspace{-10pt}\mat{\mathcal{M}^p_\xi \\ |\phi |>16M}}\hspace{-10pt} \frac{e^{it\phi}}{i\phi}m\Big[ \prod _{l=1}^{k-1}\om (t+h,\xi _l)\Big] \Big[ \frac{\om (t+h,\xi _k)-\om (t,\xi _k)}{h} -F[\om] (t,\xi _k)\Big] \Big[ \!\!\prod _{l'=k+1}^p\!\!\om (t,\xi_{l'})\Big] }{\ti{X}}\]
for $k=1,2,\dots ,p$ and
\[ \norm{\int _{t_1}^{t_2}dt \int _{\mat{\mathcal{M}^p_\xi \\ |\phi |>16M}} \frac{e^{it\phi}}{i\phi}m\Big[ \prod _{l=1}^{k-1}\om (t+h,\xi _l)-\prod _{l=1}^{k-1}\om (t,\xi _l)\Big] F[\om] (t,\xi _k) \Big[ \prod _{l'=k+1}^p\om (t,\xi_{l'})\Big] }{\ti{X}}\]
for $k=2,\dots ,p$.
The latter norm can be treated by the estimate \eqref{est:A12} and (uniform) continuity of $t\mapsto \om (t)$.
To estimate the former, let $L>L''$ and divide the integral into three regions: (a) $|\psi (\xi_k)|\leq (p+1)L$, (b) $|\phi (\bfxi )|>L$, (c) $|\psi (\xi _k)|>(p+1)L$ and $|\phi (\bfxi )|\leq L$.
In (a) we can use \eqref{est:A12} and the convergence \eqref{limitex} for each fixed $L$, while in (b) we exploit the power $\delta$ in the estimate \eqref{est:A12} to evaluate the norm as $O(L^{-\delta})$ uniformly in $h$.
In the region (c) it follows from the restriction $|\psi (\xi )|\leq L''$ that $|\psi (\xi _l)|>L$ for at least one $l\neq k$.
Then, after applying \eqref{est:A12} we have the term $\| \chi_{|\psi|>L}\om\|_{C_TL^2_s}$, which is of $o(1)$ as $L\to \infty$ by uniform continuity of $t\mapsto \om (t)$.
Therefore, we obtain \eqref{limit1'} by taking $L>L''$ large to make the contribution from (b), (c) small and then letting $h\to 0$.

So far, we have shown the equality \eqref{eq:just3} for $0<t_1<t_2<T$.
It is now easy to extend the equality to the case $t_1=0$ or $t_2=T$ by using continuity of the map \eqref{map1} and the dominated convergence theorem.
This establishes the desired equality \eqref{eq:just2} for each $t\in [0,T]$.

At the end, we see that formal calculations in the $J$-th generation can be justified in the same manner.
We need to establish, for each $\mathcal{T}\in \mathfrak{T}(J)$, the equality
\eq{eq:justJ}{\int_0^tdt' \int_{\mathfrak{X}_\xi (\mathcal{T})}e^{it'\ti{\phi}^J}f(t',\bfxi )=\Big[ \int_{\mathfrak{X}_\xi (\mathcal{T})}\frac{e^{it'\ti{\phi}^J}}{i\ti{\phi}^J}f(t',\bfxi ) \Big] _{t'=0}^{t'=t} -\int_0^tdt' \int_{\mathfrak{X}_\xi (\mathcal{T})}\frac{e^{it'\ti{\phi}^J}}{i\ti{\phi}^J}g(t',\bfxi ),}
where
\begin{align*}
&f(t,\bfxi ):=\chi _{(\phi ^j)_{j=1}^J\in \Phi_{N\!R}^{(J)}}\frac{m^1\cdots m^J}{(i\ti{\phi}^1)\cdots (i\ti{\phi}^{J-1})}\prod_{a\in \mathcal{T}_\infty}\om (t, \xi _a),\\
&g(t,\bfxi ):=\chi _{(\phi ^j)_{j=1}^J\in \Phi_{N\!R}^{(J)}}\frac{m^1\cdots m^J}{(i\ti{\phi}^1)\cdots (i\ti{\phi}^{J-1})}\sum _{a\in \mathcal{T}_\infty}F[\om ](t,\xi _a)\prod _{b\in \mathcal{T}_\infty \setminus \{ a\}}\om (t, \xi _b).
\end{align*}
Similarly to Proposition~\ref{prop:abstract-R}, we can deduce from ($\ti{A}$1)--($\ti{A}$3) that
\eqq{
\norm{\int_{\mat{\bfxi \in \mathfrak{X}_\xi (\mathcal{T}) \\ (\phi ^j)_{j=1}^J\in \Phi_{N\!R}^{(J)}}}\frac{|m^1|\cdots |m^J|}{|\ti{\phi}^1|\cdots |\ti{\phi}^{J-1}|}\prod_{a\in \mathcal{T}_\infty}|\om ^{(a)}(t, \xi _a)|}{X}&\lesssim \prod _{a\in \mathcal{T}_\infty}\| \om ^{(a)}\|_{L^2_s},\\
\norm{\int_{\mat{\bfxi \in \mathfrak{X}_\xi (\mathcal{T}) \\ (\phi ^j)_{j=1}^J\in \Phi_{N\!R}^{(J)}}}\frac{1}{|\ti{\phi}^J|^{1/2}}\frac{|m^1|\cdots |m^J|}{|\ti{\phi}^1|\cdots |\ti{\phi}^{J-1}|}\prod_{a\in \mathcal{T}_\infty}|\om ^{(a)}(t, \xi _a)|}{L^2_s}&\lesssim \prod _{a\in \mathcal{T}_\infty}\| \om ^{(a)}\|_{L^2_s},\\
\norm{\int_{\mat{\bfxi \in \mathfrak{X}_\xi (\mathcal{T}) \\ (\phi ^j)_{j=1}^J\in \Phi_{N\!R}^{(J)}}}\frac{1}{|\ti{\phi}^J|^{1-\delta}}\frac{|m^1|\cdots |m^J|}{|\ti{\phi}^1|\cdots |\ti{\phi}^{J-1}|}\prod_{a\in \mathcal{T}_\infty}|\om ^{(a)}(t, \xi _a)|}{\ti{X}}&\lesssim \min _{a\in \mathcal{T}_\infty}\Big[ \| \om ^{(a)}\|_{\ti{X}}\prod _{b\in \mathcal{T}_\infty \setminus \{ a\}}\| \om ^{(b)}\|_{L^2_s}\Big] .
}
These estimates substitute for ($\ti{A}$3), ($\ti{A}$1) and \eqref{est:A12}, respectively.
The condition (ii) implies that $\chi_{|\ti{\phi}^J(\bfxi )|>L}\leq \sum _{a\in \mathcal{T}_\infty}\chi _{|\xi _a|>C(L)}$ for some increasing function $C(\cdot )$ satisfying $C(L)\to \infty$ ($L\to \infty$), since $\bfxi$ is determined from its values on $\mathcal{T}_\infty$ by the relation $\xi _{a}=\sum _{k=1}^p\xi _{a^k}$ ($a\in \mathcal{T}_0$) imposed in the integral over $\mathfrak{X}_\xi (\mathcal{T})$.
In the case (iii), we have $\ti{\phi}^J(\bfxi )=\sum _{j=1}^J\phi (\xi _{a_j},\xi _{a_j^1},\dots ,\xi _{a_j^p})=\psi (\xi )-\sum _{a\in \mathcal{T}_\infty}\psi (\xi _a)$ for $\bfxi \in \mathfrak{X}_\xi (\mathcal{T})$.
Noticing these facts, we can justify the equality \eqref{eq:justJ} analogously to \eqref{eq:just2}.
\end{proof}

\begin{rem}
We can also restate Theorem~\ref{thm:weak} on existence of weak solutions in the non-periodic setting.
A note at technical level is that we consider the equation \eqref{Abs'} and Definition~\ref{defn:weak} in $[C_c^\I ((0,T);\Sc{S}(\Sc{M}))]'$ rather than in $[C_c^\I ((0,T);L^2_\infty (\Sc{M}))]'$ or $[C_c^\I ((0,T)\times \Sc{M})]'$, where $\mathcal{S}(\mathcal{M})$ is the Schwartz class on $\Sc{M}$ (only imposing the decay in the discrete directions), so that we can recover existence of weak solutions in the extended sense for \eqref{Abs'-2} by the inverse Fourier transform.
Since we do not use the space $X$ for this purpose, the proof is a straightforward adaptation of the argument for the periodic case and we omit it.
\end{rem}

\section*{Acknowledgments}

The author would like to express his sincere gratitude to Tadahiro Oh for valuable comments and constant encouragement.
He is also grateful to Soonsik Kwon and Zihua Guo for helpful discussions.
The author has been partially supported by Japan Society for the Promotion of Science (JSPS), Grant-in-Aid for Young Researchers (B) \#%
JP24740086 and \#%
JP16K17626.

%%%%%%%%%%%%%%%%%%%%%%%%%%%%%%%%%%%
%%%%%%%%%%%%%%%%%%%%%%%%%%%%%%%%%%%
%%%%%%%%%%%%%%%%%%%%%%%%%%%%%%%%%%%

\bibliographystyle{abbrv}
\bibliography{mybibfile}

\begin{thebibliography}{10}

\bibitem{A88}
V.~I. Arnol'd.
\newblock {\em Geometrical methods in the theory of ordinary differential
  equations}, volume 250 of {\em Grundlehren der Mathematischen Wissenschaften
  [Fundamental Principles of Mathematical Sciences]}.
\newblock Springer-Verlag, New York, second edition, 1988.
\newblock Translated from the Russian by Joseph Sz\"{u}cs [J\'{o}zsef M.
  Sz\H{u}cs].

\bibitem{BIT11}
A.~V. Babin, A.~A. Ilyin, and E.~S. Titi.
\newblock On the regularization mechanism for the periodic {K}orteweg-de
  {V}ries equation.
\newblock {\em Comm. Pure Appl. Math.}, 64(5):591--648, 2011.
\newblock \url{https://doi.org/10.1002/cpa.20356}.

\bibitem{BL01}
H.~A. Biagioni and F.~Linares.
\newblock Ill-posedness for the derivative {S}chr\"{o}dinger and generalized
  {B}enjamin-{O}no equations.
\newblock {\em Trans. Amer. Math. Soc.}, 353(9):3649--3659, 2001.
\newblock \url{https://doi.org/10.1090/S0002-9947-01-02754-4}.

\bibitem{CHKP19b}
L.~Chaichenets, D.~Hundertmark, P.~Kunstmann, and N.~Pattakos.
\newblock Knocking out teeth in one-dimensional periodic nonlinear
  {S}chr\"{o}dinger equation.
\newblock {\em SIAM J. Math. Anal.}, 51(5):3714--3749, 2019.
\newblock \url{https://doi.org/10.1137/19M1249679}.

\bibitem{CHKP19a}
L.~Chaichenets, D.~Hundertmark, P.~Kunstmann, and N.~Pattakos.
\newblock Nonlinear {S}chr\"{o}dinger equation, differentiation by parts and
  modulation spaces.
\newblock {\em J. Evol. Equ.}, 19(3):803--843, 2019.
\newblock \url{https://doi.org/10.1007/s00028-019-00501-z}.

\bibitem{CH19}
X.~Chen and J.~Holmer.
\newblock The derivation of the {$\mathbb{T}^3$} energy-critical {NLS} from
  quantum many-body dynamics.
\newblock {\em Invent. Math.}, 217(2):433--547, 2019.
\newblock \url{https://doi.org/10.1007/s00222-019-00868-3}.

\bibitem{CHKL15}
Y.~Cho, G.~Hwang, S.~Kwon, and S.~Lee.
\newblock Well-posedness and ill-posedness for the cubic fractional
  {S}chr\"{o}dinger equations.
\newblock {\em Discrete Contin. Dyn. Syst.}, 35(7):2863--2880, 2015.
\newblock \url{https://doi.org/10.3934/dcds.2015.35.2863}.

\bibitem{C05p}
M.~Christ.
\newblock Nonuniqueness of weak solutions of the nonlinear {S}chr\"{o}dinger
  equation.
\newblock {\em preprint}, 2005.
\newblock \url{https://arxiv.org/abs/math/0503366}.

\bibitem{C07}
M.~Christ.
\newblock Power series solution of a nonlinear {S}chr\"{o}dinger equation.
\newblock In {\em Mathematical aspects of nonlinear dispersive equations},
  volume 163 of {\em Ann. of Math. Stud.}, pages 131--155. Princeton Univ.
  Press, Princeton, NJ, 2007.

\bibitem{CGKO17}
J.~Chung, Z.~Guo, S.~Kwon, and T.~Oh.
\newblock Normal form approach to global well-posedness of the quadratic
  derivative nonlinear {S}chr\"{o}dinger equation on the circle.
\newblock {\em Ann. Inst. H. Poincar\'{e} Anal. Non Lin\'{e}aire},
  34(5):1273--1297, 2017.
\newblock \url{https://doi.org/10.1016/j.anihpc.2016.10.003}.

\bibitem{DPST07}
D.~De~Silva, N.~Pavlovi\'{c}, G.~Staffilani, and N.~Tzirakis.
\newblock Global well-posedness for a periodic nonlinear {S}chr\"{o}dinger
  equation in 1{D} and 2{D}.
\newblock {\em Discrete Contin. Dyn. Syst.}, 19(1):37--65, 2007.
\newblock \url{https://doi.org/10.3934/dcds.2007.19.37}.

\bibitem{DET16}
S.~Demirbas, M.~B. Erdo\u{g}an, and N.~Tzirakis.
\newblock Existence and uniqueness theory for the fractional {S}chr\"{o}dinger
  equation on the torus.
\newblock In {\em Some topics in harmonic analysis and applications}, volume~34
  of {\em Adv. Lect. Math. (ALM)}, pages 145--162. Int. Press, Somerville, MA,
  2016.

\bibitem{FT03}
G.~Furioli and E.~Terraneo.
\newblock Besov spaces and unconditional well-posedness for the nonlinear
  {S}chr\"{o}dinger equation in {$\dot{H}^s(\mathbb{R}^n)$}.
\newblock {\em Commun. Contemp. Math.}, 5(3):349--367, 2003.
\newblock \url{https://doi.org/10.1142/S0219199703001002}.

\bibitem{GKO13}
Z.~Guo, S.~Kwon, and T.~Oh.
\newblock Poincar\'{e}-{D}ulac normal form reduction for unconditional
  well-posedness of the periodic cubic {NLS}.
\newblock {\em Comm. Math. Phys.}, 322(1):19--48, 2013.
\newblock \url{https://doi.org/10.1007/s00220-013-1755-5}.

\bibitem{HF13}
Z.~Han and D.~Fang.
\newblock On the unconditional uniqueness for {NLS} in {$\dot{H}^s$}.
\newblock {\em SIAM J. Math. Anal.}, 45(3):1505--1526, 2013.
\newblock \url{https://doi.org/10.1137/120871808}.

\bibitem{HGKVp}
B.~Harrop-Griffiths, R.~Killip, and M.~Vi\c{s}an.
\newblock Large-data equicontinuity for the derivative {NLS}.
\newblock {\em preprint}, 2021.
\newblock \url{https://arxiv.org/abs/2106.13333}.

\bibitem{H06}
S.~Herr.
\newblock On the {C}auchy problem for the derivative nonlinear
  {S}chr\"{o}dinger equation with periodic boundary condition.
\newblock {\em Int. Math. Res. Not.}, 2006:Art. ID 96763, 33 pp, 2006.
\newblock \url{https://doi.org/10.1155/IMRN/2006/96763}.

\bibitem{HS19}
S.~Herr and V.~Sohinger.
\newblock Unconditional uniqueness results for the nonlinear {S}chr\"{o}dinger
  equation.
\newblock {\em Commun. Contemp. Math.}, 21(7):1850058, 33 pp, 2019.
\newblock \url{https://doi.org/10.1142/S021919971850058X}.

\bibitem{K95}
T.~Kato.
\newblock On nonlinear {S}chr\"{o}dinger equations. {II}. {$H^s$}-solutions and
  unconditional well-posedness.
\newblock {\em J. Anal. Math.}, 67:281--306, 1995.
\newblock \url{https://doi.org/10.1007/BF02787794}.

\bibitem{KTp}
T.~K. Kato and K.~Tsugawa.
\newblock {\em preprint.}
\newblock An announcement of the results in this article has appeared:
  T.~K.~Kato. {U}nconditional well-posedness of fifth order {K}d{V} type
  equations with periodic boundary condition. In {\em {H}armonic {A}nalysis and
  {N}onlinear {P}artial {D}ifferential {E}quations}, {\em {RIMS}
  {K}\^{o}ky\^{u}roku {B}essatsu}, 70:105--129, 2018.
  \url{https://hdl.handle.net/2433/243753}.

\bibitem{K13}
N.~Kishimoto.
\newblock Local well-posedness for the {Z}akharov system on the
  multidimensional torus.
\newblock {\em J. Anal. Math.}, 119:213--253, 2013.
\newblock \url{https://doi.org/10.1007/s11854-013-0007-0}.

\bibitem{K-announcement}
N.~Kishimoto.
\newblock Unconditional uniqueness of solutions for nonlinear dispersive
  equations.
\newblock In {\em Proceedings of the 40th {S}apporo {S}ymposium on {P}artial
  {D}ifferential {E}quations}, pages 78--82, 2015.
\newblock \url{https://doi.org/10.14943/81526}. Part of the results in the
  present article has been announced there without detailed proofs.

\bibitem{K-BOp}
N.~Kishimoto.
\newblock Unconditional uniqueness for the periodic {B}enjamin-{O}no equation
  by normal form approach.
\newblock {\em preprint}, 2019.
\newblock \url{https://arxiv.org/abs/1911.11108}.

\bibitem{K-Zp}
N.~Kishimoto.
\newblock Remarks on the periodic {Z}akharov system.
\newblock {\em preprint}, 2020.
\newblock \url{https://www.kurims.kyoto-u.ac.jp/preprint/file/RIMS1915.pdf}.

\bibitem{K-NLSp}
N.~Kishimoto.
\newblock Unconditional local well-posedness for periodic {NLS}.
\newblock {\em J. Differential Equations}, 274:766--787, 2021.
\newblock \url{https://doi.org/10.1016/j.jde.2020.10.025}.

\bibitem{K-mBOp}
N.~Kishimoto.
\newblock {Unconditional uniqueness for the periodic modified Benjamin–Ono
  equation by normal form approach}.
\newblock {\em Int. Math. Res. Not. IMRN}, 2021.
\newblock \url{https://doi.org/10.1093/imrn/rnab079}.

\bibitem{KT18}
N.~Kishimoto and Y.~Tsutsumi.
\newblock Ill-posedness of the third order {NLS} equation with {R}aman
  scattering term.
\newblock {\em Math. Res. Lett.}, 25(5):1447--1484, 2018.
\newblock \url{https://doi.org/10.4310/MRL.2018.v25.n5.a5}.

\bibitem{KO12}
S.~Kwon and T.~Oh.
\newblock On unconditional well-posedness of modified {K}d{V}.
\newblock {\em Int. Math. Res. Not. IMRN}, 2012(15):3509--3534, 2012.
\newblock \url{https://doi.org/10.1093/imrn/rnr156}.

\bibitem{KOY20}
S.~Kwon, T.~Oh, and H.~Yoon.
\newblock Normal form approach to unconditional well-posedness of nonlinear
  dispersive {PDE}s on the real line.
\newblock {\em Ann. Fac. Sci. Toulouse Math. (6)}, 29(3):649--720, 2020.
\newblock \url{https://doi.org/10.5802/afst.1643}.

\bibitem{MN09}
N.~Masmoudi and K.~Nakanishi.
\newblock Uniqueness of solutions for {Z}akharov systems.
\newblock {\em Funkcial. Ekvac.}, 52(2):233--253, 2009.
\newblock \url{https://doi.org/10.1619/fesi.52.233}.

\bibitem{MYp}
R.~Mosincat and H.~Yoon.
\newblock Unconditional uniqueness for the derivative nonlinear
  {S}chr\"{o}dinger equation on the real line.
\newblock {\em Discrete Contin. Dyn. Syst.}, 40(1):47--80, 2020.
\newblock \url{https://doi.org/10.3934/dcds.2020003}.

\bibitem{OST18}
T.~Oh, P.~Sosoe, and N.~Tzvetkov.
\newblock An optimal regularity result on the quasi-invariant {G}aussian
  measures for the cubic fourth order nonlinear {S}chr\"{o}dinger equation.
\newblock {\em J. \'{E}c. polytech. Math.}, 5:793--841, 2018.
\newblock \url{https://doi.org/10.5802/jep.83}.

\bibitem{OW18}
T.~Oh and Y.~Wang.
\newblock Global well-posedness of the periodic cubic fourth order {NLS} in
  negative {S}obolev spaces.
\newblock {\em Forum Math. Sigma}, 6:e5, 80 pp, 2018.
\newblock \url{https://doi.org/10.1017/fms.2018.4}.

\bibitem{OWp}
T.~Oh and Y.~Wang.
\newblock Normal form approach to the one-dimensional periodic cubic nonlinear
  {S}chr\"{o}dinger equation in almost critical {F}ourier-{L}ebesgue spaces.
\newblock {\em J. Anal. Math.}, 2021.
\newblock \url{https://doi.org/10.1007/s11854-021-0168-1}.

\bibitem{P19}
N.~Pattakos.
\newblock N{LS} in the modulation space {$M_{2,q}({\mathbb{R}})$}.
\newblock {\em J. Fourier Anal. Appl.}, 25(4):1447--1486, 2019.
\newblock \url{https://doi.org/10.1007/s00041-018-09655-9}.

\bibitem{R07}
K.~M. Rogers.
\newblock Unconditional well-posedness for subcritical {NLS} in {$H^s$}.
\newblock {\em C. R. Math. Acad. Sci. Paris}, 345(7):395--398, 2007.
\newblock \url{https://doi.org/10.1016/j.crma.2007.09.003}.

\bibitem{T99}
H.~Takaoka.
\newblock Well-posedness for the one-dimensional nonlinear {S}chr\"{o}dinger
  equation with the derivative nonlinearity.
\newblock {\em Adv. Differential Equations}, 4(4):561--580, 1999.

\bibitem{T99-2}
H.~Takaoka.
\newblock Well-posedness for the {Z}akharov system with the periodic boundary
  condition.
\newblock {\em Differential Integral Equations}, 12(6):789--810, 1999.

\bibitem{S08}
Y.~Y.~S. Win.
\newblock Unconditional uniqueness of the derivative nonlinear
  {S}chr\"{o}dinger equation in energy space.
\newblock {\em J. Math. Kyoto Univ.}, 48(3):683--697, 2008.
\newblock \url{https://doi.org/10.1215/kjm/1250271390}.

\bibitem{ST08}
Y.~Y.~S. Win and Y.~Tsutsumi.
\newblock Unconditional uniqueness of solution for the {C}auchy problem of the
  nonlinear {S}chr\"{o}dinger equation.
\newblock {\em Hokkaido Math. J.}, 37(4):839--859, 2008.
\newblock \url{https://doi.org/10.14492/hokmj/1249046372}.

\bibitem{Yth}
H.~Yoon.
\newblock {\em Normal {F}orm {A}pproach to {W}ell-posedness of {N}onlinear
  {D}ispersive {P}artial {D}ifferential {E}quations}.
\newblock PhD thesis, Korea Advanced Institute of Science and Technology, 2017.

\bibitem{Z97}
Y.~Zhou.
\newblock Uniqueness of weak solution of the {K}d{V} equation.
\newblock {\em Internat. Math. Res. Notices}, 1997(6):271--283, 1997.
\newblock \url{https://doi.org/10.1155/S1073792897000202}.

\end{thebibliography}

\end{document}